\documentclass[11pt]{amsart}
\usepackage{amsmath}
\usepackage{amsfonts}
\usepackage{amssymb}
\usepackage{amscd}
\usepackage[arrow,matrix,curve,cmtip,ps]{xy}
\usepackage{graphicx}

\numberwithin{equation}{section}

\usepackage{amsthm}

\theoremstyle{plain}
\newtheorem{thm}{Theorem}[section] 
\newtheorem{theorem}[thm]{Theorem}
\newtheorem{proposition}[thm]{Proposition}
\newtheorem{corollary}[thm]{Corollary}
\newtheorem{lemma}[thm]{Lemma}
\newtheorem{conjecture}[thm]{Conjecture}
\theoremstyle{definition}
\newtheorem{definition}[thm]{Definition}
\newtheorem{example}[thm]{Example}
\newtheorem{question}[thm]{Question}
\theoremstyle{remark}
\newtheorem{remark}[thm]{Remark}

\newcommand\OMIN{\mathop{\rm OMIN}}
\newcommand\OMAX{\mathop{\rm OMAX}}

\newcommand{\cl}[1]{\mathcal{#1}}
\newcommand{\bb}[1]{\mathbb{#1}}

\begin{document}

\title{Nuclearity Related Properties in Operator Systems}
\author{Ali S. Kavruk}
\date{July 13, 2011}
\thanks{...}
\begin{abstract} 

Some recent research on the tensor products of operator systems and ensuing nuclearity properties in this setting
raised many stability problems. In this paper we examine the preservation of these
nuclearity properties including
exactness, local liftability and the double commutant expectation property under basic
algebraic operations such as quotient, duality, coproducts and tensorial products. We
show that, in the finite dimensional case, exactness and the lifting property are dual pairs, that is,
an operator system $\cl S$ is exact if and only if the dual operator system $\cl S^d$ has the lifting property.
Moreover, the lifting property is preserved under quotients by null subspaces.

Again in the finite dimensional
case we prove that every operator system has the $k$-lifting property in the sense that whenever $\varphi : \cl S \rightarrow \cl A /I$
is a unital and completely positive map, where $\cl A$ is a C*-algebra and $I$ is an ideal, then $\varphi$ possess a unital $k$-positive lift on $\cl A$, for every $k$.
This property provides a novel proof of  a classical result of Smith and Ward on the preservation of matricial numerical ranges of an operator.

The Kirchberg conjecture naturally falls into this context. We show that the Kirchberg conjecture is equivalent to
the statement that the five dimensional universal operator system generated by two contraction ($\cl S_2$) has
the double commutant expectation property. In addition to this we give several equivalent statements to this conjecture regarding
the preservation of various nuclearity properties under basic algebraic operations.

We show that the Smith Ward problem is equivalent to the statement that every three dimensional operator
system has the lifting property (or exactness). If we suppose that both the Kirchberg conjecture and
the Smith Ward problem have an affirmative answer then this implies that every three dimensional
operator system is C*-nuclear.  We see that this property, even under most favorable conditions,
seems to be hard to verify.

\end{abstract}
\maketitle

The study of tensor products and therefore the behavior of objects under the tensorial 
operations is fundamental in operator theory. Exactness, local liftability, approximation 
property and weak expectation are some structural properties of C*-algebras
which are known to be deeply connected with the tensor product. The operator space
versions and non-selfadjoint analogues of these properties have been worked out in the
last decade (see \cite[Sec. 15,16,17]{pi} and \cite{blecher}). After being abstractly characterized by Choi and Effros, operator systems
played an important role in the understanding of tensor products of C*-algebras, nuclearity, injectivity, etc. (see \cite{La}, \cite{CE0}, \cite{CE2}, e.g.).
Some special tensor products of two operator systems are also used in quantum mechanics (\cite{SWT}, e.g.).
However a systematic study of tensor products on this category along with the characterization of nuclearity properties waited
till \cite{kptt} and \cite{kptt2} (see also \cite{HP}). This series of papers raised several questions;
namely, the stability of these properties under certain operations which is the main subject of
the present paper. More precisely we try to illuminate the behavior of the nuclearity properties
under basic algebraic constructions such as quotients, coproducts, duality, tensors etc.

$ $

We start with a brief introduction to operator systems together with their abstract characterization.
We also include some special C*-covers generated by an operator system and we continue with the
basic duality results in this category. In Section 2 we recall basic facts on the quotient theory of operator systems.
This especially allows us to utilize exactness in this category.

$ $

Section 3 includes a brief overview on the tensor products of operator systems. After
giving the axiomatic definition we recall basic facts on the minimal (min), the maximal (max), the (maximal) commuting (c),
enveloping left (el) and  enveloping right (er) tensor products. The set of  tensor products admits a natural partial
order and the primary tensor products we have considered exhibit the following relations:
$$
min \;\; \leq \;\; el\;\;  ,\;\;  er\;\; \leq \;\; c \;\; \leq \;\; max.
$$

Nuclearity forms the integral part of Section 4. Given two operator system tensor products $\alpha \leq \beta$, an operator system
$\cl S$ is said to be $(\alpha,\beta)${\it -nuclear} if 
$
\cl S \otimes_{\alpha} \cl T = \cl S \otimes_{\beta} \cl T
$
for every operator system $\cl T$. One of the main goals of \cite{kptt2} (see also \cite{HP}) is to characterize the
nuclearity properties among the primary tensor products above which forms the following equivalences:

(min,max)-nuclearity = completely positive factorization property (CPFP),

(min,el)-nuclearity = exactness,

(min,er)-nuclearity = (operator system) local lifting property (osLLP),

(el,c)-nuclearity = double commutant expectation property (DCEP),

(el,max)-nuclearity = weak expectation property (WEP).

\noindent We remark that WEP and DCEP coincides for C*-algebras. Also, again for C*-algebras, Kirchberg's 
local lifting property (LLP) and osLLP coincides. For finite dimensional operator systems we simply use the term ``lifting property''.

$ $

We consider Sections 1,2,3 and 4 as the basic part of the paper. Since many of the constructions
in later sections are applicable to the Kirchberg conjecture we put the related discussion in Section 5. Recall
that the Kirchberg conjecture is equivalent to an outstanding problem in von Neumann algebra theory, namely Connes' embedding problem, and it states
that every C*-algebra that has LLP has WEP.
Since these properties extend to general operator systems it is natural to approach this conjecture
from an operator system perspective. In \cite{kptt2} it was shown that the Kirchberg conjecture
has an affirmative answer if and only if every finite dimensional operator system with the lifting
property has DCEP. One of our main goals in Section 5 is to obtain an even simpler form of this. Let
$C^*(\mathbb{F}_n)$ represent the full C*-algebra of the free group $\mathbb{F}_n$ on $n$ generators (equipped with the
discrete topology). We define
$$
\cl S_n = span\{g_1,...,g_n,e,g_1^*,...,g_n^*\} \subset C^*(\mathbb{F}_n)
$$
where the $g_i$'s are the unitary generators of $C^*(\mathbb{F}_n)$. One can consider $\cl S_n$ as the universal
operator system generated by $n$ contractions as it is the unique operator system with the following property:
Whenever $y_1,...,y_n$ are contractive elements of an operator system $\cl T$ then there is a unique unital and completely 
positive (ucp) map $\varphi: \cl S_n \rightarrow \cl T$ satisfying $\varphi(g_i) = y_i$ for $i=1,...,n$. 
As pointed out in \cite{kptt2}, $\cl S_n$ has the lifting property
for every $n$. One of our main results in Section 5 is the operator system analogue of Kirchberg's WEP characterization (\cite{Ki2}, see also \cite[Thm. 15.5]{pi}):
A unital C*-algebra $\cl A$ has WEP is and only if $\cl A \otimes_{min} \cl S_2  = \cl A \otimes_{max} \cl S_2$. 
Turning back to the Kirchberg conjecture we obtain the following five dimensional operator system variant.

\begin{theorem}
 The following are equivalent:
\begin{enumerate}
 \item The Kirchberg conjecture has an affirmative answer.
 \item $\cl S_2$ has DCEP.
 \item $\cl S_2 \otimes_{min} \cl S_2 =  \cl S_2 \otimes_{c} \cl S_2$.
\end{enumerate}
\end{theorem}

$ $

When $E$ and $F$ are Banach spaces then the natural algebraic inclusion of the minimal
Banach space tensor product $E \Hat{\otimes} F $
into $B(E^*,F)$  is an isometry. Moreover, when $E$ is finite dimensional this inclusion is bijective.
A similar embedding and bijectivity are also true in the non-commutative setting, that is, the same inclusion is a
complete isometry if one uses the minimal operator space tensor product and
considers completely bounded maps. Since the dual of a finite
dimensional operator system is again an operator system we have a similar
representation of the minimal operator system tensor product. In Section 6 we give
several applications of this result. In particular we show that
exactness and the lifting property are dual pairs. We also show that the lifting property of a 
finite dimensional operator system is preserved
under quotient by a null subspace, in contrast to C*-algebra ideal quotients.

$ $

In Section 7 we adapt some of the results of Ozawa and Pisier in the operator space setting to operator systems. 
We primarily show that $\mathbb{B} = B(H)$ and
$\mathbb{K} = K(H)$, the ideal of compact operators,  where $H = l^2$, are universal objects for the verification of exactness and the lifting property. More precisely
we prove that an operator system $\cl S$ is exact if and only if the induced map
$$
(\cl S \hat{\otimes}_{min} \mathbb{B})/(\cl S \bar{\otimes} \mathbb{K}) = \cl S \hat{\otimes}_{min} (\mathbb{B}/\mathbb{K}) 
$$  
is a complete order isomorphism. (Here $ \hat{\otimes}_{min}$ represents the completed minimal tensor product
and $\bar{\otimes}$ is the closure of the algebraic tensor product.) Likewise a finite dimensional operator system $\cl S$ has the lifting property if and only if
every ucp map $\varphi: \cl S \rightarrow \mathbb{B}/\mathbb{K}$ has a ucp lift on $\mathbb{B}$.

$ $

The amalgamated sum of two operator systems over the unit introduced in \cite{KL} (or coproduct
of two operator systems in the language of \cite{TF}) seem to be another natural structure to seek
 the stability of several nuclearity properties.  In Section 8  we first describe the coproduct of two operator systems in terms of
operator system quotients and then we show that the lifting property is preserved under this operation.
The stability of the double commutant expectation property, with some additional assumptions, seems to
be a hard problem. We show that an affirmative answer to such a question is directly related to the Kirchberg Conjecture.
More precisely if $ \cl S = span \{ 1,z,z^* \}\subset C(\mathbb{T})$, where $z$ is the coordinate function on the unit circle $\mathbb{T}$,
then the Kirchberg conjecture is equivalent to the statement that the five dimensional operator system $\cl S \oplus_1 \cl S$,
the coproduct of $\cl S$ with itself, 
has the double commutant expectation property. (Note: Here $\cl S$ coincides with $\cl S_1$ and $\cl S \oplus_1 \cl S$ coincides with
$\cl S_2$.)

$ $

In \cite{blerina}, Xhabli introduces the $k$-minimal and $k$-maximal structure on an operator system $\cl S$. After
recalling the universal properties of these constructions we studied the nuclearity within this context.
In particular, we show that if an operator system is equipped with the $k$-minimal structure
it has exactness and, in the finite dimensional case, the $k$-maximal
structure automatically implies the lifting property. This allow us to show that
every finite dimensional operator system has the $k$-lifting property, that is, if $\varphi: \cl S \rightarrow \cl A / I$ is a
ucp map, where $\cl A$ is a C*-algebra and $I$ is an ideal in $\cl A$, then $\varphi$ has a unital $k$-positive lifting
on $\cl A$ (for every $k$).
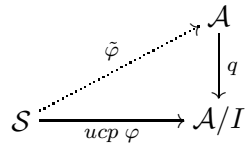
$$
\xymatrix{
   &     &  \cl A \ar[d]^q \\
 \cl S \ar[rr]_-{ucp \; \varphi} 
\ar@{.>}[rru]^{\tilde{\varphi}} &     &   \cl A/ I}
$$

$ $

From the nuclearity point of view matrix algebras are the best understood objects: In addition to being nuclear,
for an arbitrary C*-algebra, completely positive factorization property through matrix algebras
is equivalent to nuclearity (see \cite{CE0} e.g.). 
However, the quotients of the matrix algebras by some special kernels, or certain operator
subsystems of the these algebras under duality raise several difficult problems. In Section 10 we first recall these
quotient and duality results given in \cite{pf}. We simplify some of the proofs and discuss the Kirchberg
conjecture in this setting. In fact we see that this conjecture is a quotient and duality problem in the category of operator systems. 
We also look at the triple Kirchberg conjecture (Conjecture \ref{con tripleKC}).
The property $\mathbb{S}_2$, which coincides with Lance's weak
expectation for C*-algebras, appear to be at the center of understanding of these conjectures.

$ $

The Smith Ward problem (SWP), which is a question regarding the preservation of matricial numerical range of an operator under
compact perturbation, goes back to 1980. In Section 11 we abstractly characterize this problem. More
precisely, we see that SWP is a general three dimensional operator system problem rather than
a proper compact perturbation of an operator in the Calkin algebra. The following is our main result in Section 11:
\begin{theorem}
The following are equivalent:
\begin{enumerate}
 \item SWP has an affirmative answer.
 \item Every three dimensional operator system has the lifting property.
 \item Every three dimensional operator system is exact.
\end{enumerate}
\end{theorem}
\noindent This version allows us to
combine this problem with the Kirchberg conjecture (KC). In fact, if we assume both
SWP and KC then this would imply that every three dimensional operator system is C*-nuclear.
On the other hand the latter condition implies SWP. This lower dimensional
operator system problem seems to be very hard. Even for an operator system of the form $\cl S = span\{1,z,z^*\} \subset C(X)$, where
$X$ is a compact subset of the unit disk $\{z:\; |z|\leq 1\}$ and $z$ is the coordinate function,
we don't know whether $\cl S$ is C*-nuclear.

$ $

\section{Preliminaries}

In this section we establish the terminology and state the definitions and
basic results that shall be used throughout the paper. By a \textit{$*$-vector
space} we mean a complex vector space $V$ together with a map $*:V\rightarrow V$
that is involutive (i.e. $(v^*)^* = v$ for all $v$ in $V$) and conjugate linear
(i.e. $(\alpha v + w)^* = \bar{\alpha} v^* + w^*$ for all scalar $\alpha$ and
$v,w \in V$). An element $v \in V$ is called \textit{hermitian} (or
\textit{selfadjoint}) if $v = v^*$. We let $V_h$ denote the set of all
hermitian elements of $V$. By $M_{n,k}(V)$ we mean $n\times k$ matrices
whose entries are elements of $V$, that is, $M_{n,k}(V) = \{(v_{ij})_{i,j}
: v_{ij} \in V \mbox{ for } i = 1,...,n \mbox{ and } j = 1,...,m\}$ and we use
the notation $M_n(V)$ for $M_{n,n}(V)$. Note that $M_n(V)$ is again a $*$-vector
space with $(v_{ij})^* = (v_{ji}^*)$. We let $M_{n,k}$ denote the $n\times k$
matrices with complex entries and set $M_n = M_{n,n}$. If $A = (a_{ij})$ is in
$M_{m,n}$ and $X = (v_{ij})$ is in $M_{n,k}(V)$ then the multiplication $AX$ 
is an element of $M_{m,k}(V)$ whose $ij^{th}$ entry is equal $\Sigma_{r}
a_{ir}v_{rj}$ for $ i = 1,...,m$ and $ j = 1,...,k$. We define a right
multiplication with appropriate size of matrices in a similar way. 

$ $

If $V$ is a $*$-vector space then by a \textit{matrix ordering} (or \textit{a
matricial order}) on $V$ we mean a collection $\{C_n\}_{n=1}^\infty$ where each
$C_n$ is a cone in $M_n(V)_h$ and the following axioms are satisfied:

\begin{enumerate}
 \item $C_n$ is strict, that is, $C_n \cap (-C_n) = \{0\}$ for every $n$.

\item $\{C_n\}$ is compatible, that is, $A^* C_n A \subseteq C_m$ for every $A$
in $M_{n,m}$ and for every $n,m$.
\end{enumerate}

\noindent The $*$-vector space $V$ together with the matricial order structure
$\{C_n\}$ is called a \textit{matrix ordered $*$-vector space}. An
element in $C_n$ is called a \textit{positive} element of $M_n(V)$. There is a
natural (partial) order structure on $M_n(V)_h$ given by $A \leq B$ if $ B - A
$ is in $C_n$. We finally remark that we often use the notation
$M_n(V)^+$ for $C_n$. Perhaps the most important examples of these spaces are
$*$-closed subspaces of a $B(H)$, bounded linear operators on a Hilbert space
$H$, together with the induced matricial positive cone structure. More
precisely, if $V$ is such a subspace then $M_n(V)$ is again a $*$-closed
subspace of $M_n(B(H))$ which can be identified with $B(H \oplus \cdots \oplus
H)$, bounded operators on direct sum of $n$ copies of $H$. By using this
identification we will set $C_n = M_n(V) \cap M_n(B(H))^+$, where $M_n(B(H))^+$
denotes the positive elements of $M_n(B(H))$. It is elementary to verify that
the collection $\{C_n\}$ forms a matrix ordering on the $*$-vector space $V$.

An element $e$ of a matrix ordered $*$-vector space $V$ is called an
\textit{order unit} if for every selfadjoint element $v$ of $V$ there is a
positive real number $\alpha$ such that $\alpha e + v \geq 0$. Note that $e$
must be a positive element. We say that $e$ is \textit{matrix order unit} if
the corresponding $n \times n$ matrix given by
$$
e_n  = \left(
\begin{array}{ccc}
 e & & 0    \\
  &\ddots &   \\
0 && e
\end{array}
\right)
$$
is an order unit in $M_n(V)$ for every $n$. We say that $e$ is
\textit{Archimedean matrix order unit} if it is a matrix order unit and
satisfies the following: For any $v$ in $V$ if $\epsilon e + v$ is positive for
every $\epsilon > 0$ then $v$ is positive. A matrix ordered $*$-vector space
$V$ (with cone structure $\{C_n\}$) and Archimedean matrix order unit $e$ is
called an (abstract) operator system. We often drop the term ``Archimedean
matrix order'' and simply use ``unit'' for $e$. We sometimes use the notation
$(V, \{C_n\},e)$ for an operator system however to avoid excessive syntax we
simply prefer to use $\cl S$ (or $\cl T$, $\cl R$). The positive elements of
$\cl S$, i.e. $C_1$, is denoted by $\cl S^+$ and for the upper levels we
use $M_n(\cl S)^+$ rather than $C_n$. Sometimes we use $e{_\cl S}$ for the unit.
A subspace $V$ of $B(H)$ (or in general a unital C*-algebra $\cl A$) that
contains the unit $I$ and is closed under $*$ (i.e. a unital selfadjoint
subspace) is called a \textit{concrete operator system}. Note that $V$ together
with the induced matrix order structure, i.e. $C_n = M_n(V) \cap M_n(B(H))^+$ for
every $n$, and $I$ forms an (abstract) operator system. In the next paragraph we
work on the morphisms of operator systems and see that abstract and concrete
operator systems are ``essentially'' same.

$ $

Let $\cl S$ and $\cl T$ be two operator systems and $\varphi: \cl S 
\rightarrow \cl T$ be a linear map. We say that $\varphi$ is \textit{unital} if
$\varphi(e_{\cl S}) = e_{\cl T}$. $\varphi$ is called\textit{ positive} if it
maps positive elements of $\cl S$ to positive elements of $\cl T$, that is,
$\varphi(\cl S^+) \subset \cl T^+$, and \textit{completely positive} if its
$n^{th}$-\textit{amplification} $\varphi^n : M_n(\cl S) \rightarrow M_n(\cl T)$
given by $(s_{ij}) \mapsto (\varphi(s_{ij}))$ is positive for every $n$, in
other words, $\varphi^n( M_n(\cl S)^+) \subset  M_n(\cl T)^+$ for all $n$. The
term unital and completely positive will abbreviated as \textit{ucp}. $\varphi$
will be called a \textit{complete order embedding} if it is injective ucp map
such that whenever $(\varphi(s_{ij}))$ is positive in $M_n(\cl T)$ then
$(s_{ij})$ is positive in $M_n(\cl S)$. Two operator system $\cl S$ and $\cl T$
are called \textit{unitally completely order isomorphic} if there is a bijective
map $\varphi : \cl S \rightarrow \cl T$ that is unital and a complete order
isomorphism. A subspace $\cl S_0$ of operator system $\cl S$ which is unital and
selfadjoint is again an operator system  together with the induced matrix order
structure. In this case we say that $\cl S_0$ is an \textit{operator subsystem}
of $\cl S$. Note that the inclusion $\cl S_0 \hookrightarrow
\cl S$ is a unital complete order embedding. $\cl O$ stands for the category
whose objects are the operator systems and morphisms are the ucp maps. We are
now ready to state the celebrated theorem of Choi and Effros (\cite{CE2}).

\begin{theorem}
Up to a unital complete order isomorphism all the abstract and concrete operator
systems coincide. That is, if $\cl S$ is an operator system then there is a
Hilbert space $H$ and a unital  $*$-linear map $\varphi: \cl S\rightarrow B(H)$
which is a complete order embedding.
\end{theorem}

Of course, in the above theorem $B(H)$ can be replaced with a unital
C*-algebra. A subspace $X$ of a C*-algebra $\cl A$ together with the induced
matrix norm structure is called a \textit{concrete operator space.} We refer
the reader to \cite{Pa} for an introductory exposition of these objects along with
their abstract characterization due to Ruan. If $\cl S$ is an operator system
then a concrete representation of $\cl S$ into a $B(H)$ endows $\cl S$ with an
operator space structure. It follows that this structure is independent of the
particular representation and, moreover, it can be intrinsically given as
$$
\|(s_{ij}) \|_n = \inf \{\alpha > 0 :  \; \left(
\begin{array}{cc}
 \alpha e_n & (s_{ij}) \\
 (s_{ji}^*)   &   \alpha e_n
\end{array} \right)
\mbox{ is in } M_{2n} (\cl S)^+
 \}.
$$ 
This is known as the \textit{canonical operator space structure }of $\cl S$.
We also assume some familiarity with the \textit{injectivity} in the category of
operator systems. We refer to \cite[Chp. 15]{Pa} for an excellent survey, however
for an immediate use in the sequel we remark that every injective operator
system is completely order isomorphic to a C*-algebra \cite[Thm 15.2]{Pa}. We
also need the fact that if $\cl S$ is an operator system then its injective
envelope $I(\cl S)$ is ``rigid'' in the sense that the only ucp map $\varphi:
I(\cl S) \rightarrow I(\cl S)$ with the property that $\varphi(s) = s$ for every
$s$ in $\cl S$ is the identity \cite[Thm 15.7]{Pa}.

\subsection{Some Special C*-covers} $\mbox{ }$

$ $

\noindent A C*-cover $(\cl A,i)$ of an operator system $\cl S$ is a C*-algebra
$\cl A$ with
a unital complete order embedding $i:\cl S \hookrightarrow \cl A$ such that
$i(\cl S)$ generates $\cl A$ as a C*-algebra, that is, $\cl A$ is the smallest
C*-algebra containing $i(\cl S)$. We occasionally identify $\cl S$ with $i(\cl
S)$ and consider $\cl S$ as an operator subsystem of $\cl A$. Every operator
system $\cl S$ attains two
special C*-covers namely the universal and the enveloping C*-algebras denoted by
$C_u^*(\cl S)$ and $C_e^*(\cl S)$, respectively. The universal C*-algebra
satisfies the following universal ``maximality'' property: Every ucp map
$\varphi : \cl S \rightarrow \cl A$, where $\cl A$ is a C*-algebra extends
uniquely to a unital $*$-homomorphism $\pi: C_u^*(\cl S) \rightarrow \cl A$. If
$\varphi: \cl S \rightarrow \cl T$ is a ucp map then the unital $*$-homomorphism
$\pi : C_u^*(\cl S) \rightarrow C_u^*(\cl T) $ associated with $\varphi$, of
course, constructed by enlarging the range space by $C_u^*(\cl T)$ first. We
also remark that if $\cl S \subset \cl T$ then $C_u^*(\cl S) \subset C_u^*(\cl
T)$, in other words, the C*-algebra generated by $\cl S$ in $C_u^*(\cl T)$
coincides with the universal C*-algebra of $\cl S$. This special C*-cover is used
extensively in \cite{kptt}, \cite{kptt2} and \cite{kw}. As it connects
operator systems to C*-algebras it has fundamental role in the tensor theory of
operator systems and, in particular, in the present paper.

$ $

The enveloping C*-algebra $C_e^*(\cl S)$ of $\cl S$ is defined as the C*-algebra
generated by $\cl S$ in its injective envelope $I(\cl S)$. It has the following
universal ``minimality'' property: For any C*-cover $i: \cl S \hookrightarrow
\cl A$ there is a unique unital $*$-homomorphism $\pi: \cl A  \rightarrow
C_e^*(\cl S)$ such that $\pi(i(s)) = s$ for every $s$ in $\cl S$ (we assume $\cl
S \subset C_e^*(\cl S)$). The enveloping C*-algebra of an operator system is
\textit{rigid} in the sense that if $\varphi: C^*_e(\cl S) \rightarrow \cl T$ is a ucp
map such that $\varphi|_{\cl S}$ is a complete order embedding then
$\varphi$ is a complete order embedding. We refer to \cite{HA} for the proof of these results and
further properties of enveloping C*-algebras.

\subsection{Duality} $\mbox{ }$

$ $

\noindent Duality, especially on the finite dimensional operator systems,
is a strong tool in the study of the stability of various
nuclearity properties and in this subsection we review basic properties on this topic.
If $\cl S$ is an operator system  then the Banach dual $\cl S^d$ has a natural
matrix ordered $*$-vector space structure. For $f$ in $\cl S^d$, the involution is
given by $f^*(s) = \overline{f(s^*)}$. The matricial order structure is
described as follows:
$$
(f_{ij}) \in M_n(\cl S^d) \mbox{ is positive if the map } \cl S \ni
s \mapsto (f_{ij}(s)) \in M_n \mbox{ is completely positive}. 
$$
Throughout the paper $\cl S^d$ will always represent this
matrix ordered vector space. The bidual Banach space $\cl S^{dd}$ has also
a natural matricial order structure arising from the fact that it is the dual of $\cl S^d$. The following
is perhaps well known, see \cite{kptt}, e.g.:
\begin{theorem}
$\cl S^{dd}$ is an operator system with unit $\hat{e}$, the canonical image of $e$ in $\cl S^{dd}$. 
Moreover, the canonical embedding of $\cl S$ into $\cl S^{dd}$ is a complete order embedding.
\end{theorem}

A state $f$ on $\cl S$ is said to be {\it faithful} if $s\geq 0$ and $f(s) = 0$ implies that
$s = 0$, in other words, $f$ maps non-zero positive elements to positive scalars.
When $\cl S$ is a finite dimensional operator system then it possesses a faithful state
which is an Archimedean matrix order unit for the dual structure \cite[Sec. 4]{CE2}:

\begin{theorem}[Choi-Effros] Suppose $\cl S$ is a finite dimensional operator system.
Then there are faithful states on $\cl S$ and each faithful state is an Archimedean
order unit for the matrix ordered space $\cl S^d$. 
\end{theorem}

\noindent Consequently, the dual of a finite dimensional operator system is again an operator
system when we fix a faithful state. Also note that when we pass to the second dual, $\hat{e} \in \cl S^{dd}$
is a faithful state on $\cl S^d$. The following will be useful in later sections:

\begin{lemma}\label{lem dualofkpos.map}
Let $\cl S$ and $\cl T$ be two operator systems and $\varphi: \cl S \rightarrow
\cl T$ be a linear map. Then $\varphi$ is $k$-positive if and only if
$\varphi^d : \cl T^d \rightarrow \cl S^d$ is $k$-positive.
\end{lemma}

\begin{proof} First suppose that $\varphi$ is $k$-positive. Let $(g_{ij})$ be in $M_k(\cl T^d)^+$. We need to show that
$(\varphi^d(g_{ij}))$ is in $M_k(\cl S^d)^+$, that is, the map 
$$
\cl S \ni s \mapsto (\varphi^d(g_{ij}(s)))  = (g_{ij}(\varphi(s))) \in M_k
$$
is completely positive. By using a result of Choi, see \cite[Thm. 6.1]{Pa} e.g., it is
enough to show that this map is $k$-positive. So let  $(s_{lm})$ be positive
in $M_k(\cl S)$. Since $\varphi$ is $k$-positive we have that $(\varphi(s_{lm}))$
is positive in $M_n(\cl T)$. Now using the definition of positivity of
$(g_{ij})$ we have that $\big( g_{ij}(\varphi(s_{lm})) \big)$ is positive in
$M_k(M_k)$. Conversely, suppose that $\varphi^d$ is $k$-positive. By using
the above argument, we have that $\varphi^{dd}: \cl S^{dd} \rightarrow \cl T^{dd}$ is $k$-positive.
Since $\cl S \subset \cl S^{dd}$ and $\cl T \subset \cl T^{dd}$ completely order isomorphically
we have that $\varphi = \varphi^{dd}|_{\cl S}$ is $k$-positive.
\end{proof}

$ $

\section{Operator System Quotients}\label{sec quotients}

In this section we recall some basic results about operator system quotients
introduced in  \cite[Sec. 3, 4]{kptt2}. This quotient theory is also studied
and used extensively in \cite{pf} and some of them are included in the sequel.
We exhibit some relations between the quotient theory and  duality
for finite dimensional operator systems. We establish some universal
objects, namely the coproducts of operator systems, by using the quotient
theory in a later section. 

$ $

A subspace $J$ of an operator system $\cl S$ is called a \textit{kernel} if it
is the kernel of some ucp map defined from $\cl S$ into an operator system $\cl
T$. Note that a kernel $J$ has to be a $*$-closed, non-unital subspace of $\cl
S$, however, these properties,  in general, do not characterize a kernel. The
following is Proposition 3.2 of \cite{kptt2}.

\begin{proposition}
Let $J$ be a subspace of $\cl S$. Then the following are equivalent:
\begin{enumerate}
 \item $J$ is a kernel,
 \item $J$ is the kernel of a cp map defined from $\cl S$ into an operator system
$\cl T$,
  \item $J$ is the kernel of a positive map defined from $\cl S$ into an operator
system $\cl T$,
 \item there is a collection of states $f_\alpha$ such that $J =
\cap_{\alpha}ker(f_{\alpha})$.
\end{enumerate}
\end{proposition}

The algebraic quotient $\cl S/J$ has a natural involution given by $(s+J)^* =
s^*+J$. To define the matricial order structure we first consider the following
cones:
$$
D_n = \{(s_{ij}+J)_{i,j=1}^n:\;\;(s_{ij}) \in M_n(\cl S)^+\}.
$$
It is elementary to show that $\{D_n\}_{n=1}^\infty$ forms a strict, compatible
order structure. Moreover, $e+J$ is a matrix order unit. However, it fails to be
Archimedean, that is, if $(s+J) + \epsilon (e+J)$ is in $D_1$ for every
$\epsilon > 0$, then $s+J$ may not be in $D_1$. To solve this problem we
use the Archimedeanization process introduced in \cite{pt}. More precisely, we
enlarge the cones in such a way that they still form a strict compatible
matricial order structure and $e+J$ is an Archimedean matrix order unit.
Consider
$$
C_n = \{(s_{ij}+J)_{i,j=1}^n:\;\;(s_{ij}) +\epsilon (e+J)_n \in D_n \mbox{ for
every } \epsilon > 0 \}.
$$
The $*$-vector space $\cl S/J$ together with the
matricial order structure $\{C_n\}_{n=1}^\infty$ and unit $e+J$ form an operator
system. We refer to \cite[Sec. 3]{kptt2} for the proof of this result.
The operator system $\cl S / J$ is called the \textit{quotient operator
system}.  A kernel $J$  is called \textit{proximinal} if $D_1 = C_1$ and
\textit{completely proximinal} if $D_n = C_n$ for every $n$. We remark that the proximinality in this context is different
than the norm-proximinality in the Banach or operator space quotients.

$ $

One of the fundamental properties of an operator system quotient $\cl S / J$ is
its relation with morphisms. If $\varphi: \cl S \rightarrow \cl T$ is a ucp
map with $ J \subseteq ker(\varphi) $ then the associated map
$\bar{\varphi} : \cl S / J \rightarrow T$ is again a ucp map. Conversely,
if $\psi: \cl S  / J \rightarrow \cl T$ is a ucp map then there exists a
unique ucp map $\phi : \cl S \rightarrow \cl T$ with, necessarily, $J \subseteq
ker(\phi)$ such that $\phi = q \circ \psi$ where $q$ is the quotient map from
$\cl S$ onto $\cl S / J$. We also remark that if one considers completely
positive maps and drop the condition on the unitality then both of these
universal properties still hold.

$ $

\noindent \textbf{Remark:} If one starts with a $*$-closed, non-unital subspace
$J$ of an operator system
$\cl S$ then, on the algebraic quotient $\cl S / J$ the involution is still
well-defined. We can still define $D_n$ in similar fashion and it is
elementary to show that $\{D_n\}$ is a compatible matricial cone structure. It
is possible that $\{D_n\}$ is strict as well. However, in order to obtain the
Archimedeanization property of $e+J$ we again need to enlarge the cones and define
$\{C_n\}$ in a similar way. Now it can be shown that $C_1$ is strict, that is,
$C_1 \cap (-C_1) = \{0\}$, if and only if $J$ is a kernel. Consequently starting
with a kernel is essential in the operator system quotient. (See \cite[Sec.
3]{kptt2} for an extended discussion on this topic).

$ $

\begin{remark}
Let $\cl A$ be a unital C*-algebra and $I$ be an ideal in $\cl A$. (It is easy
to see that $I$ is a kernel, in fact it is the kernel of the quotient map $\cl
A \rightarrow \cl A/I$). Then the C*-algebraic quotient of $\cl A$ by $I$ is
unitally completely order isomorphic to the operator system quotient $\cl A/I$.
Moreover, $I$ is proximinal.
\end{remark}

Proximinality is a useful tool and we want to consider some special cases in
which the kernels are automatically proximinal. The first part of the following
is essentially \cite[Lemma 4.3.]{kptt2}.

\begin{lemma}\label{npnn}
Let $y$ be a selfadjoint element of an operator system $\cl S$ which is neither positive nor
negative. Then $span\{y\}$ is a kernel in $\cl S$. Moreover, $span\{y\}$ is
proximinal.
\end{lemma}

\begin{proof}
The first part of the proof can be found in  \cite[Lemma 4.3]{kptt2}. To prove
the second part we first consider the case where $y$ is such an element in a
unital C*-algebra $\cl A$. Let $J = span\{y\}$ and let $x+J \geq 0$  in $\cl A /
J$. Clearly we may assume that $x = x^*$. We have that for each $\epsilon >
0$ there is an element in $J$, say $\alpha_{\epsilon} y$ such that $x+
\alpha_{\epsilon} y + \epsilon e$ is positive in $\cl A$. Note that
$\alpha_{\epsilon}$ must be a real number. Let $X_{\epsilon} = \{\alpha :  x +
\alpha y + \epsilon e \in \cl A^+ \}$ then $X_\epsilon$ is a non-empty subset
of $\mathbb{R}$ such that for any $0<\epsilon_1 \leq \epsilon_2$ we have
$X_{\epsilon_1} \subseteq X_{\epsilon_2}$. Moreover since $\cl A^+$ is closed
in $\cl A$, each of $X_\epsilon$ is closed. We will show that $X_1$ is
bounded. Let $y = y_1 - y_2$ be the Jordan decomposition of $y$, that is, $y_1$
and $y_2$ are positives such that $y_1y_2 = 0$. Let $\alpha$ be in $X_1$. Now
multiplying both side of $ x + \alpha y_1 - \alpha y_2 + e \geq 0$ by $y_2$
from right and left we get $y_2x y_2 + y_2^2 \geq \alpha y_2^3 $. Since $y_2$ is
non-zero this inequality puts an upper bound on $\alpha$. Similarly multiplying
both side by $y_1$ we obtain a lower bound for $\alpha$. Consequently
$\{X_\epsilon\}_{0<\epsilon \leq 1}$ is a decreasing net of compact sets in
$\mathbb{R}$ and hence have a non-empty intersection. Let $\alpha_0$ be an
element belongs to the intersection. Since $x + \alpha_0 y +\epsilon \geq 0$
for every $\epsilon>0$ we have that $x+ \alpha_0 y \geq 0$. This proves the
particular case we assumed. Now suppose $y$ is such an element in $\cl S$. Let
$\cl A$ be a C*-algebra containing $\cl S$ as an operator subsystem. We have
that $J = span\{y\}$ is a proximinal kernel in $\cl A$. Let $q$ be the quotient
map from $\cl A$ onto $\cl A / J$ and let $q_0$ be the restriction of $q$ on
$\cl S$. Clearly $q_0$ is ucp with kernel $J$. So $\bar{q_0}: \cl S / J
\rightarrow \cl A / J$ is ucp. Now let $s + J$ be positive in $\cl S/J$. So it
is positive in $\cl A / J$. By the above part there is an element $a$ in $\cl
A^+$ such that $a + J = s + J$. Since $J$ is contained $\cl S$ clearly $a$ must
be an element of $\cl S$. So the proof is done. 
\end{proof}

A finite dimensional $*$-closed subspace $J$ of an operator system $\cl S$ which
contains no positive other than $0$ is called a \textit{null subspace}.
Supposing $y$ is a self-adjoint element of $\cl S$ which is neither positive
nor negative then $span\{y\}$ is a one dimensional null subspace, e.g. Another
important example of null subspaces are kernels of faithful states on finite
dimensional operator systems.

\begin{proposition}\label{prop nullsubspace}
Suppose $J$ is a null subspace of $\cl S$. Then $J$ is a completely proximinal
kernel. If $\cl S$ is finite dimensional, say $\dim(\cl S) = n$, then $J$ is
contained in an $n-1$
dimensional null subspace.
\end{proposition}

\begin{proof} We first show that $J$ is a proximinal kernel. We will argue by
induction on the dimension of $J$. When $J$ is one dimensional Lemma \ref{npnn}
does the job. Suppose every $k$ dimensional null subspace of the operator system
 $\cl S$ is a proximinal kernel and let $J$ be an $k+1$ dimensional null
subspace. It is elementary to see that $J = span \{y_1,...,y_k,y_{k+1}\}$ where
each of $y_i$ is selfadjoint. Let $J_0 = span\{y_1,...,y_k\}$ which is a null
subspace and consequently a proximinal kernel by the induction assumption. We
claim that $y_{k+1} + J_0$ is a selfadjoint element in $\cl S /  J_0$ which is
neither positive nor negative. Clearly it is selfadjoint. Suppose it is
positive, so there is a positive element $x$ in $\cl S$ such that $x + J_0 =
y_{k+1} + J_0 $. This clearly forces $x$ to be in $J$ so it is necessarily $0$
and thus $y_{k+1}$ is in $J_0$ which is a contradiction. Similarly $y_{k+1}$
cannot be negative. Again by using Lemma \ref{npnn} $span\{y_{k+1} + J_0\}$ is a
proximinal kernel in $\cl S/ J_0$. Now consider the sequence of the quotients
maps
$$
\cl S \xrightarrow{\hspace{0.5cm}q_0\hspace{0.5cm}} \cl S / J_0 
\xrightarrow{\hspace{0.5cm}q_1\hspace{0.5cm}} (\cl S / J_0)/span\{y_{k+1}+J_0\}.
$$
Clearly the kernel of $q_1 \circ q_0$ is $J$ and since the first and the second
quotients are proximinal it is easy to show that that $J$ is proximinal. To see
that $J$ is a completely proximinality  we can simply consider the
identification
$$
M_n(\cl S / J) = M_n(\cl S) /M_n(J).
$$
Note that $M_n(J)$ is still a null subspace on $M_n(\cl S)$.

Now we will show that if $\dim(\cl S) = n$ then $J$ is contained in an $n-1$
dimensional null subspace. Let $w$ be a faithful state on $\cl S / J$.
Clearly kernel of $w$ is a null subspace and so proximinal by the above part.
Now;
$$
\cl S \xrightarrow{\hspace{0.3cm}q\hspace{0.3cm}} \cl S / J 
\xrightarrow{\hspace{0.3cm}w\hspace{0.3cm}} \mathbb{C}
$$
is a sequence of ucp maps with $n-1$ dimensional kernel in $\cl S$ which
contains $J$. It is null subspace since a  non-zero positive will map a non-zero
positive by $q$ first and a non-zero positive real number by $w$.
\end{proof}

As we pointed out earlier, the kernel of a faithful state on a finite dimensional
operator system is a null subspace. This led us to construct a very special basis
for the operator system as well as its dual.

\begin{lemma}\label{lem special basis}
Suppose $\cl S$ is an $n$ dimensional operator system and $\delta$ a
faithful state on $\cl S$. Then the kernel of $\delta$, which is an $n\!-\!1$
dimensional null subspace, can be written as a linear
combination of self-adjoint elements $\{s_2,...,s_n\}$. Consequently we have
$$
\cl S = span\{e\! = \! s_1,s_2,...,s_n\}.
$$
Moreover if $\cl S^d = span \{ \delta \! = \! \delta_1, \delta_2,....,\delta_n
\}$ written in the dual basis form (i.e. $\delta_i(s_j) = \delta_{ij}$) then
$\delta_{2},...,\delta_n$ are self-adjoint elements of the dual operator system
such that their span is a null subspace.
\end{lemma}

\begin{proof} It is elementary to see that the kernel of $\delta$ can be written
as a linear combination of selfadjoints. In fact we can start with a selfadjoint
element $s_2$. If $s$ is an element in the kernel which is not in the span of
$s_2$ then one of $s+s^*$ or $(s-s^*)i$ does not belong to span of $s_2$. So
this way we obtain $s_3$. We can apply this procedure successively and form such
a basis. Clearly if we set $s_1 = e$ then we obtain a basis for $\cl S$. To see
that $\delta_i$ is self-adjoint consider an element $\Sigma \alpha_{j} s_j$.
Then
$$
\delta_i^*(\Sigma \alpha_{j} s_j) = \overline{ \delta_i(\Sigma (\alpha_{j}
s_j)^*) } = \overline{ \delta_i(\Sigma \overline{\alpha_{j}} s_j) } = \alpha_i
$$
coincides with $\delta_i(\Sigma \alpha_{j} s_j) $. Finally since
$\hat{e}$, the canonical image of $e$ in the bidual operator system, is a
faithful state on the dual operator system $\cl S^d$ its kernel, namely the
linear span of $\{\delta_2,...,\delta_n\}$, is a null subspace. This finishes
the proof.
\end{proof}

Let $J_n$ be the subspace of $M_n$ containing all diagonal matrices with 0
trace. Then $J_n$ is an $n-1$ dimensional null subspace of $M_n$ and
consequently a kernel. Note that it is contained in the subspace
which includes all the matrices with 0 trace, an $n^2-1$ dimensional null
subspace of $M_n$. In \cite{pf} it has been explicitly shown that $J_n$ is a
kernel. We will turn back to this in later sections. Another interesting example
is the following: Consider $J = span\{ g_1,...,g_n,g_1^*,...,g_n^* \} \subset
C^*(\mathbb{F}_n)$. Then $J$ is a null subspace and hence a kernel in $C^*(\mathbb{F}_n)$.

A surjective completely positive map $\varphi : \cl S \rightarrow \cl T$ is
called a \textit{quotient map} if the induced map $\bar{\varphi} : \cl S /
ker(\varphi) \rightarrow \cl T$, which is bijective and completely positive, is
a complete order isomorphism. Note that if $\varphi$ is unital the induced
map is also unital. We also remark that compositions of quotient maps are again
quotient maps. We frequently use the following property of a quotient map: If
$(t_{ij})$ is positive in $M_k(\cl T)$ then for every $\epsilon> 0 $ there is a
positive element $(s_{ij}^\epsilon)$ in $M_k(\cl S)$ such that
$(\varphi(s_{ij}^\epsilon)) = (t_{ij}) + \epsilon e_n$.

\begin{proposition}\label{qde}
Let $\varphi : \cl S \rightarrow \cl T$ be a quotient map. Then the dual map
$\varphi^d : \cl T^d \rightarrow \cl S^d$ is a complete order embedding.
\end{proposition}

 \begin{proof} We already have that the dual map is completely positive. Suppose
$(g_{ij})$ in $M_n(\cl T^d)$ such that $(\varphi^d(g_{ij}))$ is positive in
$M_n(\cl S^d)$. We will show that $(g_{ij})$ is positive, that is, if $(t_{lm})$
is positive in $M_k(\cl T)$ then $(g_{ij}(t_{lm}))$ is positive (in $M_k\otimes
M_n$). Fix $\epsilon>0$ and let $(t_{lm}^\epsilon) = (t_{lm}) + \epsilon e_k$.
We know that there is positive element $(s_{lm}^\epsilon)$ in $M_k(\cl S)$ such
that $(\varphi(s_{lm}^\epsilon)) = (t_{lm}^\epsilon)$. Note that
$(g_{ij}(t_{lm}^\epsilon))_{i,j,l,m} = (\varphi^d(g_{ij})(s_{lm}^\epsilon))$.
Now using the fact that $(\varphi^d(g_{ij}))$ is positive we get
$(g_{ij}(t_{lm}^\epsilon))_{i,j,l,m}$ is positive. Since $\epsilon$ is
arbitrary and $(t_{lm}^\epsilon) \rightarrow (t_{lm})$ as $\epsilon 
\rightarrow 0$ we have that  $(g_{ij}(t_{lm}))$ is positive. So the proof is
done.
\end{proof}

\begin{proposition}\label{prop dual of quotient}
Let $J$ be a null subspace of a finite dimensional operator system $\cl S$.
Then $(\cl S / J)^d$ is an operator subsystem of $\cl S^d$ with a proper
selection of faithful states. (More precisely if $\delta$ is a faithful state
on $\cl S$ with $J \subset ker(\delta)$ then the induced state
$\overline{\delta}$ on $\cl S/J$ is faithful and satisfies $q^d(\bar{\delta}) = \delta$ where
$q$ is the quotient map from $\cl S$ onto $\cl S/J$).
\end{proposition}

\begin{proof}
Proposition \ref{qde} ensures that $q^d : (\cl S/J)^d\rightarrow \cl S^d$ is a
complete order embedding. So we deal with the proper selection of the faithful
states. In fact let $\delta_0$ be a faithful state on $\cl S / J$. Then we claim
that $\delta_0 \circ q$ is a faithful state on $\cl S$. Clearly it is a state
and if $s$ is non-zero positive then $\varphi(s)$ is non-zero positive in $\cl
S/J$ and $\delta_0(q(s))$ is a positive number.  Finally declaring  $\delta_0
\circ q$ as the unit of $\cl S^d$, we obtain that $q^d$ is unital as
$q^d(\delta_0) = \delta_0 \circ q$.
\end{proof}

We remark that in order to obtain ``unitality'' in the above proposition
starting with a null subspace is important. In fact if $J$ is a kernel and
$\delta_1$ and $\delta_2$ are faithful states on $\cl S/J$ and $\cl S$,
respectively, then $q^d(\delta_1) = \delta_2$ requires that $J$ is in the kernel
of $\delta_2$ and consequently it has to be a null subspace.

The converse of the above result is also true which is referred 
as the First Isomorphism Theorem in \cite{pf}. For completeness of the paper we include the proof.

\begin{theorem}[Farenick, Paulsen]\label{thm dual of inc}
Let $\cl S$ be a finite dimensional operator system and $\cl S_0$ be an operator subsystem of $\cl S$.
Then the adjoint $i^{d} : \cl S^d \rightarrow \cl S_0^d$ of the inclusion $\cl S_0 \hookrightarrow \cl S$ is a quotient
map. By proper selection of faithful states we may also assume that it is unital.  
\end{theorem}

\begin{proof}
Since the inclusion is a cp map its adjoint is again a cp map. It is also elementary to see that $i^d$ is surjective.
Thus, we will only prove that if $(i^d(f_{ij}))$ is positive in $M_n(\cl S_0^d)$ then there is positive $(g_{ij})$
in $M_n(\cl S)$ such that $i^d(f_{ij}) = i^d(g_{ij})$ for every $i,j$. Now, $(i^d(f_{ij}))$ is positive in $M_n(\cl S_0^d)$
means that the linear map
$$
\cl S_0 \ni s \mapsto (i^d(f_{ij})(s)) = (f_{ij}(s)) \in M_n
$$
is a cp map. By Arveson's extension theorem (\cite{Ar1}), this map has a cp extension from $\cl S$ into $M_n$, which we identify
with $(g_{ij})$. Now, clearly $(g_{ij})$ is positive in $M_n(\cl S^d)$ and $i^d(f_{ij}) = i^d(g_{ij})$ for every $i,j$. We will
continue with the unitality problem. In fact it is elementary to show that if $f$ is a faithful state on $\cl S$ then $f$ still
has the same property when it is restricted to $\cl S_0$. Thus $i^d(f)$ is again a faithful state.
\end{proof}

\begin{remark}
In the above theorem we see that adjoint of the inclusion map is a unital quotient map. The kernel of
this map is a null subspace. In fact if $f$ is positive in $\cl S^d$ and $i^d(f) = 0$ together imply that $f$
is a positive linear functional on $\cl S$ such that $f|_{\cl S_0}$ is 0. Since, we have that $\|f\| = \|f(e)\|$,
necessarily $f = 0$.
\end{remark}

$ $

\section{Tensor Products of Operator Systems}\label{sec tensor}

In this section we recall the axiomatic definition of tensor products in the
category of operator systems and review properties of several tensor
products established in \cite{kptt}. Suppose $\cl S$ and $\cl T$ are two
operator systems. A matricial cone structure $\tau = \{C_n\}$ on $\cl S \otimes
\cl T$  where $C_n \subset M_n(\cl S \otimes \cl T)_h$,  is said to be an
\textit{operator system structure} if \begin{enumerate}

\item $(\cl S \otimes \cl T,\{C_n\},e_{\cl S}\otimes e_{\cl T})$ is an operator
system,

\item for any $(s_{ij}) \in M_n(\cl S)^+$ and $(t_{rs}) \in M_k(\cl T)^+$,
$(s_{ij} \otimes t_{rs})$ is in $C_{nk}$ for all $n,k$,

\item if $\phi: \cl S \rightarrow M_n$ and $\psi: \cl T \rightarrow M_k$ are ucp
maps then $\phi\otimes\psi:\cl S \otimes \cl T \rightarrow M_{nk}$ is a ucp map
for every $n$ and $k$.
\end{enumerate}
The resulting operator system is denoted by $\cl S \otimes_{\tau} \cl T$. A
mapping $\tau:\cl O \times \cl O\rightarrow \cl O$ is said to be an
\textit{operator system tensor product} (or simply a \textit{tensor product})
provided $\tau$ maps each pair $(\cl S,\cl T)$ to an operator system structure
on $\cl S \otimes \cl T$, denoted by $\cl S \otimes_{\tau} \cl T$. A tensor
product $\tau$ is said to be \textit{functorial} if for every operator systems
$\cl S_1, \cl S_2, \cl T_1$ and $\cl T_2$ and every ucp maps $\phi: \cl
S_1\rightarrow \cl S_2$ and $\psi: \cl T_1\rightarrow \cl T_2$ the associated
map $\phi \otimes \psi: \cl S_1 \otimes_{\tau} \cl T_1 \rightarrow \cl S_2
\otimes_{\tau} \cl T_2$ is ucp. A tensor product $\tau$ is called
\textit{symmetric} if $\cl S \otimes_{\tau} \cl T = \cl T \otimes_{\tau} \cl S $
and \textit{associative} if $(\cl S \otimes_{\tau} \cl T)\otimes_{\tau} \cl R =
\cl S \otimes_{\tau} (\cl T \otimes_{\tau} \cl R) $ for every $\cl S, \cl T$ and
$\cl R$.

$ $

There is a natural partial order on the operator system tensor products: If
$\tau_1$ and $\tau_2$ are two tensor products then we say that $\tau_1 \leq
\tau_2$ if for every operator systems $\cl S$ and $\cl T$ the identity $id: \cl
S \otimes_{\tau_2} \cl T \rightarrow \cl S \otimes_{\tau_1} \cl T$ is completely
positive. In other words $\tau_1$ is smaller with respect to $\tau_2$ if the
cones it generates are larger. (Recall that larger matricial cones generate
smaller canonical operator space structure.) The partial order on operator
system tensor products forms a lattice as pointed out in \cite[Sec. 7]{kptt}
and raises fundamental nuclearity properties as we shall discuss in the next
section.

$ $

In the remaining of this section we discuss several important tensor products,
namely the minimal (min), maximal (max), maximal commuting (c), enveloping left
(el) and enveloping right (er) tensor products. With respect to the partial
order relation given in the previous paragraph we have the following schema
\cite{kptt} :
$$
min \;\; \leq \;\; el\;\;  ,\;\;  er\;\; \leq \;\; c \;\; \leq \;\; max.
$$

\subsection{Minimal Tensor Product}$\mbox{ }$

Let $\cl S$ and $\cl T$ be two operator systems. We define the matricial cone
structure on the tensor product $\cl S \otimes \cl T$ as follows:
\begin{align*} 
C_n^{min}(\cl S,\cl T) = \{(u_{ij}) &\in M_n(\cl
S\otimes \cl T) : ((\phi\otimes\psi)(u_{ij}))_{ij} \in
M_{nkm}^+ \\
& \mbox{ for every ucp maps } \phi : \cl S \to M_k  \mbox{ and }  \psi :  \cl
T \to M_m \mbox{ for all $k,m$.}\}.
\end{align*} 
The resulting cone structure $\{C_n^{min}\}$ satisfies the axioms (1), (2) and
(3) and the resulting operator system is denoted by $\cl S \otimes_{min} \cl
T$. If $\tau$ is another operator system structure on $\cl S \otimes \cl T$ then
we have that $min \leq \tau$. In other words $\{C_n^{min}\}$ forms the
largest cone structure. The minimal tensor product, of course when considered
as a map $min : \cl O \times \cl O \rightarrow \cl O$, is symmetric and
associative. It is functorial and injective in the sense that if $\cl S_1
\subset \cl S_2$ and $\cl T_1 \subset \cl T_2$ then $\cl S_1 \otimes_{min} \cl
T_1 \subset \cl S_2 \otimes_{min} \cl T_2$ completely order isomorphically. It 
coincides with the the C*-algebraic minimal tensor products when restricted to
C*-algebras (except for completion). It is also spatial in the sense that if
$\cl S \subset B(H)$ and $\cl T \subset B(K)$ then the concrete operator system
structure on $\cl S \otimes \cl T$ arising from the inclusion $B(H\otimes K)$
coincides with their minimal tensor product. All of these result can be directly
found in \cite[Sec. 4]{kptt}.

\subsection{Maximal Tensor Product}\label{subsec max} $\mbox{ }$

The construction of the maximal tensor product of two operator systems $\cl S$
and $\cl T$ involves two steps. We first define
$$
D_n^{max}(\cl S,\cl T) = \{A^* (P\otimes Q) A :
P\in M_k(\cl S)^+, Q\in M_m(\cl T)^+, A \in M_{km,n}, \\ k,m\in \bb{N}\}.
$$
The matricial order structure $\{ D_n^{max}\}$ is strict and compatible (for the
definitions see \cite[Chp. 13]{Pa} e.g.), moreover, $e_{\cl S} \otimes e_{\cl
T}$ is a matrix order unit. However it fails to be an Archimedean order unit. So
the construction requires another step, namely the completion of the cones which
is known as the Archimedeanization process (see \cite{pt} e.g) as follows:
$$
C_n^{max}(\cl S,\cl T) = \{ P \in M_n (\cl S \otimes \cl T) : r(e_1 \otimes
e_2)_n +P \in D_n^{max}(\cl S, \cl T) \text{ $\forall$ $r>0$} \}.
$$ 
Now the matrix order structure $\{ C_n^{max}\}$ satisfies all the axioms and the
resulting operator system is denoted by $\cl S \otimes_{max} \cl T$. If $\tau$
is another operator system structure on $\cl S \otimes \cl T$ then
we have that $\tau \leq max $, that is,  $\{C_n^{max}\}$ is the smallest cone
structure. max, as min, has all properties symmetry, associativity and
functoriality. It coincides with the C*-algebraic maximal tensor product when
restricted to unital C*-algebras (again, except for completion). As it is well
known from C*-algebras, it does not have  the injectivity property that
$min$ possesses. However it is projective as discussed in \cite{Han}. Another important aspect of
the maximal tensor product is the following duality property given by Lance in \cite{La2}: A linear map $f : \cl S \otimes_{max} \cl T \rightarrow \mathbb{C}$
is positive if and only if the corresponding map $\varphi_f : \cl S \rightarrow \cl T^d$ is completely
positive. Here $\varphi_f(s)$ is the linear functional on $T$ given by $\varphi_f(s)(t) = f(s \otimes t)$.
(See also \cite[Lem. 5.7 and Thm. 5.8]{kptt}.) Consequently we obtain the following representation 
of the maximal tensor product:
$$
(\cl S \otimes_{max} \cl T)^{d,+} = CP(\cl S, T^d).
$$
The following property of the maximal tensor product will be useful:
\begin{proposition}\label{prop strong funt.max}
Let $\cl S_i$ and $\cl T_i$ be operator systems and $\varphi_{i}:\cl S_i \rightarrow \cl T_i$ be completely
positive maps for $i=1,2$. Then the associated map 
$\varphi_{1}\otimes \varphi_2  :\cl S_1\otimes_{max} \cl S_2 \rightarrow \cl T_1 \otimes_{max} \cl T_2$
is cp.
\end{proposition}
\begin{proof}
It is elementary to show that $(\varphi_{1}\otimes \varphi_2)^n(D_n^{max}(\cl S_1,\cl S_2) ) \subset D_n^{max}(\cl T_1,\cl T_2) $.
So suppose $u$ is in $C_n^{max}(\cl S_1,\cl S_2)$. For any 
$r>0$,  $r (e_1 \otimes e_2)_n + u \in D_n^{max}(\cl S_1,\cl S_2)  $. This means that, for every 
$r>0$, $r(\varphi_1(e_1) \otimes \varphi_2(e_2))_n +(\varphi_1\otimes \varphi_2)^{n}( u) $ is in $  D_n^{max}(\cl T_1,\cl T_2) $. Now, we
can complete the positive elements $\varphi_1(e_1)$ and $\varphi_2(e_2)$ to a multiple of the units, that is, we can
find positive elements $x \in \cl S_2$ and $y \in \cl T_2$ such that $\varphi_1(e_1)+x$ and $\varphi_2(e_2) + y$ are multiple of
the units. Since $r (x\otimes y)_n$ belongs to   $  D_n^{max}(\cl T_1,\cl T_2) $ we have that sum of these terms
$$
r (x\otimes y)_n + r(\varphi_1(e_1) \otimes \varphi_2(e_2))_n +(\varphi_1\otimes \varphi_2)^{n}( u)  = r k  (e_1 \otimes e_2)_n + (\varphi_1\otimes \varphi_2)^{n}( u) 
$$
is in   $  D_n^{max}(\cl T_1,\cl T_2) $ for every $r>0$. Thus, $ (\varphi_1\otimes \varphi_2)^{n}( u) \in C_n^{max}(\cl T_1,\cl T_2)$.
\end{proof}

\subsection{Maximal Commuting Tensor Product} $\mbox{ }$

Another important tensor product we want to discuss is the maximal commuting (or commuting)
tensor product which is denoted by c. It agrees with the C*-algebraic maximal tensor
products on the category of unital C*-algebras however it is different then max
for general operator systems. The matrix order structure is defined by using the
ucp maps  with commuting ranges. More precisely, if $\cl S$ and $\cl T$ are two
operator systems then $C^{com}_n$ consist of all $(u_{ij}) \in M_n(\cl S\otimes
\cl T)$ with the property that for any Hilbert space $H$,  any ucp $\phi:\cl S
\rightarrow B(H)$ and $\psi:\cl T \rightarrow B(H)$ with commuting ranges
$$
(\phi\cdot \psi)^{(n)} (u_{ij}) \geq 0
$$
where $\phi \cdot \psi: \cl S \otimes \cl T \rightarrow B(H)$ is the map defined
by $\phi \cdot \psi(s\otimes t) = \phi(s)\psi(t)$. The matricial cone
structure $\{C^{com}_n\}$ satisfies the axioms (1), (2) and (3), and the
resulting operator system is denoted by $\cl S \otimes_{c} \cl T $. The
commuting tensor product $c$ is functorial and symmetric however we don't know
whether is it associative  or not.  Before listing the main results concerning the
tensor product c we underline the following fact: If $\tau$ is an operator
system structure on $\cl S \otimes \cl T$ such that $\cl S \otimes_{\tau} \cl T$
attains a representation in a $B(H)$ with ``$\cl S$'' and ``$\cl T$'' portions
are commuting then $\tau \leq c$. This directly follows from the definition of c
and justifies the name ``maximal commuting''. The following are Theorems 6.4 and
6.7 from \cite{kptt}.
\begin{theorem} \label{thm c=max}
If $\cl A$ is a unital $C^*$-algebra and $\cl S$ is an operator
system, then $$
\cl A \otimes_{c} \cl S = \cl A \otimes_{\max} \cl S.
$$
\end{theorem}

\begin{theorem}\label{thm rep of c}
Let $\cl S$ and $\cl T$ are operator systems. Then $\cl S \otimes_{c} \cl T
\subset C^*_u(\cl S)  \otimes_{max} C^*_u(\cl T)$.
\end{theorem}

In fact the following improvement of this  theorem will be more useful in later
sections.

\begin{proposition} \label{prop rep of c2}
Let $\cl S$ and $\cl T$ be operator systems. Then $\cl S \otimes_{c} \cl T
\subset C^*_u(\cl S)  \otimes_{max} \cl T$.
\end{proposition}

\begin{proof}
By using the functoriality of c we have that the following maps
$$
\cl S \otimes_{c} \cl T \xrightarrow{i \otimes id } C^*_u(\cl S)  \otimes_{max}
\cl T \xrightarrow{id \otimes i} C^*_u(\cl S)  \otimes_{max} C^*_u(\cl T),
$$
where \textit{id} is the identity and $i$ is the inclusion, are ucp. Theorem
\ref{thm rep of c} ensures that the composition is a complete order embedding
so the first map, which is unital, has the same property. (Here we use the fact that if the composition of
two ucp maps is a complete order embedding then the first map has the same property.)
\end{proof}

Following result is direct consequence of  (\cite[Cor. 6.5]{kptt2}) which characterizes the ucp map defined by the commuting tensor product of two
operator systems:
\begin{proposition}\label{prop ucp on c}
Let $\cl S$ and $\cl T$ be two operator systems and let $\varphi: \cl S \otimes_{c} \cl T \rightarrow B(H)$ be a ucp map.
Then there is Hilbert space $K$ containing $H$ as a Hilbert subspace and ucp maps $\phi: \cl S \rightarrow B(K)$ and
 $\psi: \cl T \rightarrow B(K)$ with commuting ranges such that $\varphi = P_H \phi \cdot \psi |_H $. Conversely, every
such map is ucp.
\end{proposition}

\subsection{Some Asymmetric Tensor Products} $\mbox{ }$

In this subsection we discuss the enveloping left (el) and enveloping right 
(er) tensor products. Given operator systems $\cl S$ and $\cl T$ we define 
$$
\cl S \otimes_{el} \cl T : \subseteq I(\cl S) \otimes_{max} \cl T \mbox{ and }
\cl S \otimes_{er} \cl T : \subseteq \cl S \otimes_{max} I(\cl T)
$$  
where $I(\cdot)$ is the injective envelope of an operator system. Both
\textit{el} and \textit{er} are functorial tensor products. We don't
know whether these tensor products are associative. They
are not symmetric but asymmetric in the sense that 
$$ \cl S \otimes_{el} \cl T
=\cl T \otimes_{er} \cl S \;\;\; \mbox{ via the map } \;\;\; s \otimes t \mapsto
t \otimes s.
$$ 

el and er have the following one sided injectivity property \cite[Thm.
7.5]{kptt}

\begin{theorem}\label{thm el is li}
The tensor product \textit{el} is the maximal left injective functorial tensor
product, that is, for any $\cl S \subset \cl S_1$ and $\cl T$ we have
$$
\cl S \otimes_{el} \cl T \subseteq \cl S_1 \otimes_{el} \cl T
$$
and it is the maximal functorial tensor product with this property.
\end{theorem}
Likewise, er is the maximal right injective tensor product. It directly follows from the definition that
if $\cl S$ is an injective operator 
system then $\cl S \otimes_{el} \cl T = \cl S \otimes_{max} \cl T$ for every operator system $\cl T$. Now for an 
arbitrary operator system $\cl S$ this allows
us to conclude that the tensor product el is independent of the the injective object that we represent $\cl S$, that is,
if $\cl S \hookrightarrow \cl S_1$ where $\cl S_1$ is injective then for any operator system $\cl T$, the tensor
product on $\cl S \otimes \cl T$ arising from the inclusion $\cl S_1 \otimes_{max} \cl T$ coincides with el. To see this we only need to use
the left injectivity of el:
$$
\cl S \otimes_{el} \cl T \hookrightarrow \cl S_1 \otimes_{el} \cl T = \cl S_1 \otimes_{max} \cl T.
$$
A similar property for the tensor product er holds. el and
er, in general, are not comparable however they both lie between min and
c.

$ $

\section{Characterization of Various Nuclearities}\label{sec nuclearity}

In the previous section we have reviewed the tensor products in the category of
operator systems. In this section we will overview the behavior of the
operator systems under tensor products. More precisely, we will see several
characterizations of the operator systems that fix a pair of tensor products.

$ $

Given two tensor products $\tau_1 \leq \tau_2 $, an operator systems $\cl S$ is
said to be\textit{ $(\tau_1,\tau_2)$-nuclear} provided $ \cl S \otimes_{\tau_1}
\cl T = \cl S \otimes_{\tau_2} \cl T $ for every operator system $\cl T$. 
We remark that the \textit{place} of the operator system $\cl S$ is important as
not all the tensor products are symmetric.

\subsection{Completely Positive Factorization Property (CPFP)} $\mbox{ }$

We want to start with a discussion on the characterization of
(min,max)-nuclearity given in \cite{HP}. An operator system $\cl S$ is said to
have \textit{CPFP} if there is net  of ucp maps $$\cl \phi_{\alpha}: \cl
S \rightarrow M_{n_{\alpha}}  \mbox{ and } \cl \psi_{\alpha}: M_{n_{\alpha}}
\rightarrow \cl S $$ such that the identity $id:\cl S \rightarrow \cl S$
approximated by $\psi_{\alpha}\circ \phi_{\alpha}$ in point-norm topology, that
is, for any $s \in \cl S$, $\psi_\alpha \circ \phi_\alpha (s) \rightarrow s$.
The following is Corollary 3.2 of \cite{HP}.

\begin{theorem}
The following are equivalent for an operator system $\cl S$:
\begin{enumerate}
 \item $\cl S$ is (min,max)-nuclear, that is,  $
\cl S \otimes_{min} \cl T = \cl S \otimes_{max} \cl T 
$ for all $\cl T$.
 \item $\cl S$ has CPFP.
\end{enumerate}
\end{theorem}

The characterization in this theorem extends the characterization of
nuclear unital C*-algebras. Recall that a unital C*-algebra $\cl A$ is said to
be nuclear if $\cl A \otimes_{min} \cl B = \cl A \otimes_{max} \cl B$ for every
C*-algebra $\cl B$. By using Proposition \ref{prop rep of c2}, it is elementary
to show that $\cl A$ is nuclear if and only if it is (min,max)-nuclear operator
system. Consequently the above result extends a well known result of Choi
and Effros \cite{CE0}. We also remark that in \cite{Ki1} and  \cite{kw} an
operator system is defined as nuclear if it satisfies CPFP. Consequently the
classical term ``nuclearity'' coincides with the (min,max)-nuclearity.

\subsection{Operator System Local Lifting Property (osLLP)}\label{sec osLLP}$\;$

Another aspect we want to discuss is the operator system local lifting property
(osLLP) and we will see that it is equivalent to (min,er)-nuclearity. An
operator system $\cl S$ is said to have osLLP if for every unital C*-algebra $\cl
A$ and ideal $I$ in $\cl A$ and for every ucp map $\varphi: \cl S \rightarrow
\cl A / \cl I$ the following holds: For every finite dimensional operator
subsystem $\cl S_0 $ of $\cl S$, the restriction of $\varphi$ on $\cl S_0$, say
$\varphi_0$, lifts to a completely positive map on $\cl A$ so that the following diagram
commutes.
$$
\xymatrix{
   &     &  \cl A \ar[d]^q \\
\hspace{2cm}\cl S_0 \subset \cl S \ar[rr]_-{ucp \; \varphi} 
\ar@{.>}[rru]^{\tilde{\varphi}_0} &     &   \cl A/ \cl I} \hspace{4cm}
$$
Of course, $\cl S$ may possess osLLP without a global lifting. We also remark
that the completely positive local liftings can also be chosen to be ucp  in the
definition of osLLP (see the discussion in \cite[Sec. 8]{kptt2}). The LLP
definition for a C*-algebra given in \cite{Ki2} is the same. So it follows that
a unital C*-algebra has LLP (in the sense of Kirchberg) if and only if it has
osLLP. The following result is from \cite{Ki2}.
\begin{theorem}[Kirchberg] The following are equivalent for a C*-algebra $\cl
A$:
\begin{enumerate}
 \item $\cl A$ has LLP
 \item $\cl A \otimes_{min} B(H) = \cl A \otimes_{max} B(H)$ for every Hilbert
space $H$.
\end{enumerate}
\end{theorem}

\noindent Here is the operator system variant given in \cite{kptt2}:

\begin{theorem} The following are equivalent for an operator system $\cl S$:
\begin{enumerate}
 \item $\cl S$ has osLLP.
 \item $\cl S \otimes_{min} B(H) = \cl S \otimes_{max} B(H)$ for every Hilbert
space $H$.
 \item $\cl S$ is (min,er)-nuclear, that is, $\cl S \otimes_{min} \cl T = \cl S
\otimes_{er} \cl T$ for every $\cl T$.
\end{enumerate}
\end{theorem}

\noindent It is not hard to show that in the above theorem ``every Hilbert
space'' can be replaced by $l^2(\mathbb{N})$. If we denote $\mathbb{B} =
B(l^2(\mathbb{N}))$, the above equivalent conditions, in some similar context,
is also called $\mathbb{B}$-nuclearity. (See \cite{blecher}, e.g.)
Consequently for operator systems osLLP, $\mathbb{B}$-nuclearity and
(min,er)-nuclearity are all equivalent.

\begin{remark}
The definition of LLP of a C*-algebra in \cite[Chp. 16 ]{pi} is different, it
requires completely contractive liftings from finite dimensional operator
subspaces. However, as it can be seen in \cite[Thm. 16.2]{pi}, all the
approaches coincide for C*-algebras. 
\end{remark}

\noindent \textbf{Note:} When we work with the finite dimensional operator
systems we remove the extra word ``local'', we even remove ``os'' and simply say
``lifting property''.

$ $

It seems to be important to remark that in the definition of osLLP one can can replace ucp maps
by cp maps.
\begin{remark}\label{rem strong lp}
The following are equivalent for an operator system $\cl S$:
\begin{enumerate}
 \item $\cl S$ has osLLP.
 \item For every unital C*-algebra $\cl A$ and ideal $I$ and for every cp map $\varphi:\cl S \rightarrow \cl A / I$, the restriction of $\varphi$ 
on any finite dimensional operator subsystem $\cl S$ has a cp lift on $\cl A$.
\end{enumerate}
\end{remark}

\begin{proof}
(2) implies (1) is clear. Conversely suppose (1) holds. This implies that 
$\cl S \otimes_{min} B(H) = \cl S \otimes_{max} B(H)$. Let  $\varphi:\cl S \rightarrow \cl A / I$
be a cp map and $\cl S_0$ is finite dimensional operator subsystem of $\cl S$.
Now if we represent $\cl S_0^d$ in to a $B(H)$ (and set $\mathbb{B} = B(H)$)
we have that
$$
\cl S_0^d \otimes_{min} \cl S \subset \mathbb{B} \otimes_{min} \cl S =  \mathbb{B} \otimes_{max} \cl S \xrightarrow{id \otimes \varphi}  \mathbb{B} \otimes_{max} \cl A/I,
$$
is cp map where we use the injectivity of minimal tensor product and Proposition \ref{prop strong funt.max}.
By using first remark in Chp. 17 \cite{pi} and \cite[Cor. 5.16]{kptt2}, we have that
$$
 \mathbb{B} \otimes_{max} \cl A / I = \frac{ \mathbb{B} \otimes_{max} \cl A}{ \mathbb{B} \otimes_{max} I}\rightarrow \frac{ \mathbb{B} 
\otimes_{min} \cl A}{ \mathbb{B} \otimes_{min} I}
 \supset  \frac{\cl S_0^d \otimes_{min} \cl A}{\cl S_0^d \otimes I}.
$$
Since the inclusion $i: \cl S_0 \rightarrow \cl S$ is cp, this corresponds to a positive element $u_{i}$ in $\cl S_0^d \otimes_{min} \cl S$.
(See \cite[Lem. 8.4]{kptt2}.) Thus, under the composition of the above maps, the image $v$ of $u_{i}$ is still positive in $ (\cl S_0^d \otimes_{min} \cl A)/(\cl S_0^d \otimes I)$.
Since this quotient is proximinal (see \cite[Cor. 5.15]{kptt2}), there is a positive element  $w$ in $\cl S_0^d \otimes_{min} \cl A$ giving $v$ under the
quotient map. Now, again by using  \cite[Lem. 8.4.]{kptt2}, $w$ corresponds to a cp map $\tilde{\varphi} : \cl S_0 \rightarrow \cl A$. It is easy to verify that $\tilde{\varphi}$
is a lift of $\varphi$ when restricted to $\cl S_0$.
\end{proof}

\subsection{Weak Expectation Property (WEP)} $\;$

\noindent If $\cl A$ is a unital C*-algebra then the bidual C*-algebra $\cl
A^{**}$ is unitally completely order isomorphic to the bidual operator system
$\cl A^{dd}$. This allows one to extend the notion of WEP, which is introduced
and shown to be a fundamental nuclearity property by Lance in \cite{La}, to the
category of operator systems. We say that an operator system $\cl S$ has
\textit{WEP}  if the canonical inclusion $i:\cl S \hookrightarrow \cl S^{dd}$
extends to a ucp map on the injective envelope $I(\cl S)$.
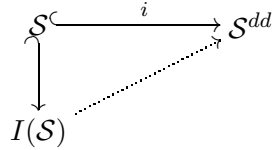
$$
\xymatrix{
 \cl S \ar@{^{(}->}[rr]^{i} \ar@{_{(}->}[d] &     & \cl S^{dd}  \\
 I(\cl S) \ar@{.>}[rru] &     &
}
$$

In \cite{kptt2} it was shown that  WEP implies (el,max)-nuclearity and the
difficult converse is shown in \cite{Han}. Consequently we have that
\begin{theorem}
An operator system has WEP if and only if it is (el,max)-nuclear.
\end{theorem}

\subsection{Double Commutant Expectation Property (DCEP)} $\;$

\noindent Another nuclearity property we want to discuss is DCEP which
coincides with WEP for unital C*-algebras however is different than WEP for
general operator systems. An operator system $\cl S$ is said to have
\textit{DCEP} if every representation $i:\cl S \hookrightarrow B(H)$ extends
to a ucp map from $I(\cl S)$ into $\cl S''$, the double commutant of $\cl S$ in
$B(H)$.
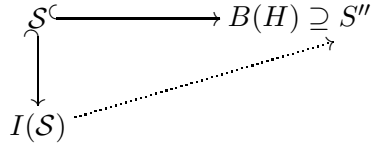
$$\hspace{2cm}
\xymatrix{
 \cl S \ar@{^{(}->}[rr] \ar@{_{(}->}[d] &     &B(H) \supseteq  S''
\hspace{3cm}\\
 I(\cl S) \ar@{.>}[rru] &     &
}
$$

In fact, by using Arveson's commutant lifting theorem \cite{Ar1} (or \cite[Thm.
12.7]{Pa}), it can be directly shown that a unital C*-algebra has WEP if and
only if it has DCEP. Many fundamental results and conjectures concerning WEP in
C*-algebras reduces to DCEP in operator systems. The following is a
direct consequence of Theorem 7.1 and 7.6 in  \cite{kptt2}:
\begin{theorem}
The following are equivalent for an operator system $\cl S$:
\begin{enumerate}
 \item  $\cl S$ is (el,c)-nuclear, that is, $\cl S \otimes_{el} \cl T = \cl S
\otimes_{c} \cl T$ for every $\cl T$.
 \item $\cl S$ has DCEP.
 \item $\cl S \otimes_{min} C^*(\mathbb{F}_\infty) = \cl S \otimes_{max} C^*(\mathbb{F}_\infty)$.
 \item For any $\cl S \subset \cl A$ and $\cl B$, where $\cl A$ and $\cl B$ are
unital C*-algebras, the inclusion $\cl S \otimes_{max} \cl B \hookrightarrow \cl
A \otimes_{max} \cl B $ is a complete order embedding.
\end{enumerate}
\end{theorem}
\noindent Here $C^*(\mathbb{F}_\infty)$ is the full C*-algebra of the free group on countably infinite generators $\mathbb{F}_{\infty}$.
Note that (3) is Kirchberg's WEP characterization in \cite{Ki2} and (4) is
Lance's seminuclearity in \cite{La} for unital C*-algebras.

\subsection{Exactness}\label{subsec exact}$\;$

\noindent The importance of exactness and its connection to the tensor theory of
C*-algebras ensued by Kirchberg \cite{Ki1}, \cite{Ki2}. Exactness is really a
categorical term and requires a correct notion of quotient theory. The operator
system quotients established in \cite{kptt2}, which we reviewed in Section
\ref{sec quotients}, is used to extend the exactness to operator systems. Before
starting the definition we recall a couple of results from \cite{kptt2}: Let $\cl
S$ be an operator system, $\cl A$ be a unital C*-algebra and $I$ be an ideal in
$\cl A$. Then $\cl S \bar{\otimes} I $ is a kernel in $ \cl S 
\hat{\otimes}_{min} \cl A$ where  $\hat{\otimes}_{min}$ represents the completed
minimal tensor product and $\bar{\otimes}$ denotes the closure of the algebraic
tensor product in the larger space. By using the functoriality of the minimal
tensor product it is easy to see that the map
$$
 \cl S \hat{\otimes}_{min} \cl A \xrightarrow{id\otimes q} \cl S
\hat{\otimes}_{min} (\cl A / I),
$$
where $id$ is the identity on $\cl S$ and $q$ is the quotient map from $\cl A$
onto $\cl A / I$, is ucp and its kernel contains $\cl S \bar{\otimes} I $.
Consequently the induced map
$$
 (\cl S \hat{\otimes}_{min} \cl A) /  (\cl S \bar{\otimes} I ) \longrightarrow
\cl S \hat{\otimes}_{min} (\cl A / I)
$$
is still unital and completely positive. An operator system is said to be
\textit{exact} if this induced map is a bijective and a complete order
isomorphism for every C*-algebra $\cl A$ and ideal $I$ in $\cl A$. In other
words we have the equality
$$
 (\cl S \hat{\otimes}_{min} \cl A) /  (\cl S \bar{\otimes} I ) = \cl S
\hat{\otimes}_{min} (\cl A / I).
$$
We remark that the induced map may fail to be surjective or injective, moreover
even if it has these properties it may fail to be a complete order isomorphism.

\begin{remark}\label{rem ex for fd} If $\cl S$ is finite dimensional then we have that $\cl S \otimes_{min} \cl A = \cl S \hat{\otimes}_{min} \cl A$
and $\cl S \bar{\otimes} I = \cl S \otimes I $. Moreover the induced map
$$
(\cl S \otimes_{min} \cl A) /  (\cl S \otimes I ) \longrightarrow
\cl S \otimes_{min} (\cl A / I)
$$
 is always bijective. Thus, for this case, exactness is equivalent to the statement that the induced map is a complete order isomorphism.
\end{remark}

\begin{proof}
Let $\cl S = span\{s_1,...,s_k\}$. Suppose that $\{u_n\} $  is a Cauchy sequence in the algebraic
tensor product $ \cl S \otimes_{min} \cl A$ with limit $u$ in $\cl S \hat{\otimes}_{min} \cl A$. We will show that $u$ belongs to $ \cl S \otimes_{min} \cl A$. Clearly
we can write $u_n = s_1 \otimes a_1^n + \cdots s_k \otimes a_k^n$. We will prove that $\{a_i^n\}_n$ is Cauchy in $\cl A$ for
every $i = 1,...,k$. Let $\delta_{i} : \cl S \rightarrow \mathbb{C} $ be the linear map defined by $\delta_i(s_j) = \delta_{ij}$.
Since each of $\delta_i$ is completely bounded we have that $\delta \otimes id : \cl S \otimes_{min} \cl A \rightarrow \cl A$ given by $s \otimes a \mapsto \delta(s)a$
is a completely bounded map, in particular it is continuous. (Here we use the fact that minimal tensor product of two operator system
is same as the operator space minimal tensor product. This is easy to see as both of them are spatial. We also use the fact that every
linear map defined from a finite dimensional operator space is completely bounded.) Clearly $\{a_i^n\}_n$ is the image of $\{u_n\}$ under this map
and consequently it is Cauchy. Let $a_i$ be the limit of these sequences in $\cl A$ for $i = 1,...,k$. Now it is elementary to show that
$u = s_1 \otimes a_1 + \cdots  s_k \otimes a_k$. This directly follows from the triangle inequality and the cross norm property of the minimal
tensor product, i.e., $\|s \otimes a\| = \|s\|\|a\|$).

The proof of the fact that $\cl S \bar{\otimes} I = \cl S \otimes I $ is similar to this so we skip it. It is elementary to see
that the image of the induced map
$$
(\cl S \otimes_{min} \cl A) /  (\cl S \otimes I ) \longrightarrow
\cl S \otimes_{min} (\cl A / I)
$$
covers the algebraic quotient which is same as its completion for this case. Thus, it is onto. Finally we need to show that
it is injective. More precisely, we need to show that the map $\cl S \otimes_{min}\cl A \rightarrow \cl S \otimes \cl A / I$
has kernel $\cl S \otimes I$. Suppose the image of $\Sigma s_i \otimes a_i$ is $0$, that is, $\Sigma s_i \otimes \dot{a_i}$ is $0$ in
$ \cl S \otimes \cl A / I$. Since $\{s_1,...,s_n\}$ is a linearly independent set we have that each of $\dot{a_1},...,\dot{a_k}$ is $0$.
Thus $a_1,...,a_k$ belongs to $I$. This finishes the proof.
\end{proof}

\noindent \textbf{Note:} The term exactness in this paper coincides with
1-exactness in \cite{kptt2}.

$ $

\noindent A unital C*-algebra is exact (in the sense of Kirchberg) if and only
if it is an exact operator system which follows from the fact that the unital C*-algebra ideal quotient coincides
with the operator system kernel quotient. The following is Theorem 5.7 of \cite{kptt2}:

\begin{theorem}\label{thm exact=(min,el)}
An operator system is exact if and only if it is (min,el)-nuclear.
\end{theorem}
In Theorem \ref{exactdualLP} we will see that  exactness and the lifting property
are dual pairs. We want to finish this subsection with the following stability property:

\begin{proposition}\label{prop exac pass ss}
Exactness passes to operator subsystems. That is, if $\cl S$ is exact then every operator subsystem of $\cl S$
is exact. Conversely, if every finite dimensional operator subsystem of $\cl S$ is exact then $\cl S$ is exact.
\end{proposition}
\begin{proof}
We will use the nuclearity characterization of exactness, i.e., (min,el)-nuclearity. First suppose $\cl S$ is exact
and $\cl S_0$ is an operator subsystem of $\cl S$. By using the injectivity of min and left injectivity of el we have
that
$$
\cl S_0 \otimes_{min} \cl T \subseteq \cl S \otimes_{min} \cl T \mbox{ and } \cl
S_0 \otimes_{el} \cl T \subseteq \cl S \otimes_{el} \cl T
$$
for every operator system $\cl T$. Since the tensors on the right hand side
coincide it follows that $\cl S_0$ is (min,el)-nuclear, equivalently it is exact. 

To prove the second part suppose that $\cl S$ is not exact. This means that there is an operator system $\cl T$
such that the identity
$$
\cl S \otimes_{min} \cl T  \rightarrow  \cl S \otimes_{el} \cl T
$$
is not a cp map, that is, there is an positive element $U$ in $M_n(\cl S\otimes_{min} \cl T)$
which is not positive in $M_n(\cl S\otimes_{el} \cl T)$. Clearly $\cl S$ has a finite dimensional
operator subsystem $\cl S_0$ such that $U$ belongs to $M_n(\cl S_0\otimes \cl T)$. Now again using the fact
that
$$
\cl S_0 \otimes_{min} \cl T \subseteq \cl S \otimes_{min} \cl T \mbox{ and } \cl
S_0 \otimes_{el} \cl T \subseteq \cl S \otimes_{el} \cl T
$$
we see that $U$ is positive in $M_n(\cl S_0\otimes_{min} \cl T)$ but not positive in $M_n(\cl S_0\otimes_{el} \cl T)$.
This means that $\cl S_0$ is not exact. This finishes the proof.
\end{proof}

\subsection{Final Remarks on Nuclearity} $\;$

\noindent Unlike  C*-algebras a finite dimensional operator system may not posses a certain
type of nuclearity. For example $M_2\oplus M_2$ has a five dimensional operator subsystem
which does not have the lifting property (See Corollary \ref{cor M4hasnonlp}, e.g.). Exactness and the local lifting
property of three dimensional operator systems are directly related to the Smith Ward problem which
is currently still open (Sec. \ref{Sec MNR}). Similarly we will see that the Kirchberg Conjecture is a problem about
nuclearity properties of five dimensional operator systems.

$ $

\noindent The following schema summarizes the nuclearity characterizations  that we have
discussed in this section:

$$
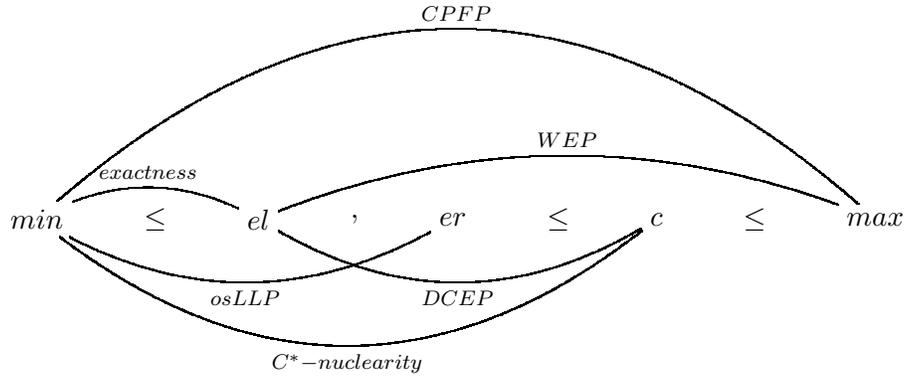
\xymatrix{
min  \ar@{-}@/^1pc/[rr]^{exactness}   \ar@{-}@/^6pc/[rrrrrrrr]^{CPFP} 
\ar@{-}@/_2pc/[rrrr]_{osLLP}    \ar@{-}@/_4pc/[rrrrrr]_{C^*-nuclearity}
&  \leq  &
el  \ar@{-}@/^2pc/[rrrrrr]^{WEP}     \ar@{-}@/_2pc/[rrrr]_{DCEP} 
&  ,   & er &  \leq     &   c & \leq & max
}
$$

\begin{proposition}
 The following are equivalent for an operator system $\cl S$:
\begin{enumerate}
 \item $\cl S$ is (min,c)-nuclear, that is, $\cl S \otimes_{min} \cl T = \cl S
\otimes_{c} \cl T$ for every operator system $\cl T$.
 \item $\cl S$ is C*-nuclear, that is, $\cl S \otimes_{min} \cl A = \cl
S \otimes_{max} \cl A$ for every unital C*-algebra $\cl A$.
\end{enumerate}
\end{proposition}
\begin{proof}
Suppose (1). By using Theorem \ref{thm c=max} we have that
$
\cl S \otimes_{min} \cl A = \cl S \otimes_{c} \cl A = \cl S \otimes_{max} \cl A.
$
Hence we obtain (2). Conversely suppose (2). By the injectivity of
the minimal tensor product and by Proposition \ref{prop rep of c2} we have the
inclusions
$$
\cl S \otimes_{min} \cl T \subseteq \cl S \otimes_{min} C^*_u(\cl T) \mbox{ and
} \cl S \otimes_{c} \cl T \subseteq \cl S \otimes_{max} C^*_u(\cl T).
$$
Since the tensor products on the right hand side coincides (1) follows.
\end{proof}

\noindent \textbf{Remarks:}
\begin{enumerate}
 \item We use the term C*-nuclearity rather than (min,c)-nuclearity.

\item The above table for unital C*-algebras summarizes the classical
discussion for C*-algebras. Recall that in this case c and max coincides and
consequently WEP and DCEP are the same properties. Also osLLP and LLP are the
same. It is also important to remark that if we start with a unital C*-algebra $\cl A$ then
(min,el)-nuclearity, for example, can be verified with unital C*-algebras. That is, $\cl A \otimes_{min} \cl T = \cl A \otimes_{el} \cl T$
for every operator system $\cl T$ if and only if $\cl A \otimes_{min} \cl B = \cl A \otimes_{el} \cl B$ for every
unital C*-algebra $\cl B$. We left the verification of this to the reader. In addition to this, as we pointed 
out before, $\cl A$ is exact (in the sense of Kirchberg) if and only if it is an exact operator system. Similar
properties hold for other nuclearity properties WEP, CPFP  and LLP. Thus, we obtain the following schema:
$$
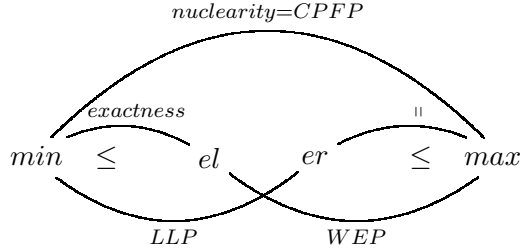
\hspace{4cm}
\xymatrix{
\;\;\;\;\;\;\;\; min \;\;\; \leq \ar@{-}@/^1pc/[r]^{\;\;\;\;\;exactness}  
\ar@{-}@/^4pc/[rrr]^{nuclearity = CPFP} 
\ar@{-}@/_2pc/[rr]_{LLP}  
  &
el    \ar@{-}@/_2pc/[rr]_{WEP} 
   & er \ar@{-}@/^1pc/[r]^{\;\;\;\shortparallel} & \leq \;\;\;  max
\hspace{1cm}&  & &
}
$$
For this case (er,max)-nuclearity of a C*-algebra coincides with the
nuclearity by Lance \cite{La}, (see also \cite[Prop. 7.7]{kptt}). By this
simple schema it is rather easy to see that nuclearity is equivalent to 
exactness and WEP, e.g. Also suppose that $\cl A$ and $\cl B$ are unital
C*-algebras such that $\cl A$ has WEP and $\cl B$ has LLP. Now by using the
fact that LLP is equivalent to (min,er)-nuclearity we have that $\cl A
\otimes_{min} \cl B  = \cl A \otimes_{el} \cl B $. (Note: $\cl B$ is on the
right hand side.) Again by using the fact that WEP is same as
(el,c$=$max)-nuclearity we have $ \cl A \otimes_{el} \cl B =  \cl A \otimes_{max} \cl B$. Thus
we obtain a well known result of Kirchberg:  $ \cl A \otimes_{min} \cl B =  \cl A \otimes_{max} \cl B$.
\end{enumerate}

$ $

We close this section with the following observation about finite dimensional operator systems. 
Roughly speaking it states that the finite dimensional operator systems, except a small portion, 
namely the C*-algebras, are never (c,max)-nuclear. So in this case, (min,c)-nuclearity (i.e. C*-nuclearity)
is the highest nuclearity that that one should expect. (Of course, among the tensor products
min $\leq$ el, er $\leq$ c $\leq$ max.)

\begin{proposition}
The following are equivalent for a finite dimensional operator system $\cl S$:
\begin{enumerate}
 \item $\cl S$ is (c,max)-nuclear.
 \item $\cl S$ is unitally completely order isomorphic to a C*-algebra.
 \item $\cl S \otimes_c \cl S^d = \cl S \otimes_{max} \cl S^d$.
\end{enumerate}
\end{proposition}

\begin{proof} 
Since $c$ and $max$ coincides when one of the tensorants is a
C*-algebra, (2) implies (1). Clearly (1) implies (3). We will show that (3)
implies (2). Consider $id: \cl S \rightarrow \cl S$. This corresponds to a
positive linear
functional $f_{id}: \cl S \otimes_{max} \cl S^d \rightarrow \mathbb{C}$. Since
$max$ and $c$ coincide by the assumption and $\cl S \otimes_c \cl S^d \subset
C^*_u(\cl S) \otimes_{max} \cl S^d$, $f_{id}$ extends to a positive linear
functional $\tilde{f_{id}} : C^*_u(\cl S) \otimes_{max} \cl S^d \rightarrow
\mathbb{C} $ by Arveson's extension theorem. Let $\varphi : C^*_u(\cl S) 
\rightarrow \cl (S^d)^d = \cl S$ be the corresponding cp map. Clearly $\varphi$
extends the identity on $\cl S$. Now by using a slight modification of
\cite[Theorem 15.2]{Pa} we have that $\cl S$ has a structure of a C*-algebra. 
\end{proof}

$ $

\section{WEP and Kirchberg's Conjecture }\label{sec new WEP KC}

In this section we improve Kirchberg's WEP characterization for unital C*-algebras and we express Kirchberg's Conjecture
in terms of a five dimensional operator system system problem. The last schema in the previous section  
still includes many question marks. There is no known example of a non-nuclear C*-algebra which has WEP and LLP.
 One another major open question is whether LLP implies WEP, which is known as the Kirchberg Conjecture. More precisely, in his astonishing paper \cite{Ki2} he proves that:

\begin{theorem}[Kirchberg] The following are equivalent:
\begin{enumerate} 
 \item Every separable $II_1$-factor is a von Neumann subfactor of the ultrapower $R_{\omega}$ of the       
hyperfinite  $II_1$-factor $R$ for some ultrafilter $\omega \in \beta \mathbb{N} \setminus \mathbb{N}$.
 \item For a unital C*-algebra LLP implies WEP.
 \item Every unital C*-algebra is a quotient of a C*-algebra that has WEP (i.e. QWEP).
 \item $C^*( \mathbb{F}_{\infty}) \otimes_{min} C^*( \mathbb{F}_{\infty}) =  C^*( \mathbb{F}_{\infty}) \otimes_{max} 
C^*( \mathbb{F}_{\infty})$.
 \item $C^*( \mathbb{F}_{\infty}) $ has WEP.
\end{enumerate}
\end{theorem}
\noindent The equivalent conditions in this theorem are still unknown. The first one is the Connes' Embedding Problem.
We refer to \cite{Co1} for related definitions on this subject. The remaining equivalent arguments are known as the Kirchberg
Conjecture (or Kirchberg's QWEP Conjecture). As we pointed out before $C^*(\mathbb{F}_{\infty})$ 
(resp., $C^*(\mathbb{F}_{n})$) stands for the full C*-algebra of the free group with
a countably infinite number of (resp., with $n$) generators. As shown in \cite{Ki2}, in the above theorem $C^*(\mathbb{F}_{\infty})$ can be replaced by $C^*(\mathbb{F}_{2})$.
In fact, since there is an injective group homomorphism $\rho: \mathbb{F}_{\infty} \rightarrow \mathbb{F}_{2}$, by using Proposition 8.8. in \cite{pi}, we have that
  $C^*(\mathbb{F}_{\infty})$ can be represented as a C*-subalgebra of  $C^*(\mathbb{F}_{2})$ and, again by using the same theorem,
there is ucp inverse of this representation. Consequently the identity on   $C^*(\mathbb{F}_{\infty})$ factors via ucp maps through  $C^*(\mathbb{F}_{2})$.
Conversely, the identity on  $C^*(\mathbb{F}_{2})$ factors via ucp maps through  $C^*(\mathbb{F}_{\infty})$ in a trivial way.
\begin{lemma}\label{lem iden. decom.}
Let $\cl S$ and $\cl T$ be two operator systems. If the identity on $\cl S$ factors via ucp maps through $\cl T$ then any nuclearity
property of $\cl T$ passes to $\cl S$. That is if $\cl T$ is $(\tau_1,\tau_2)$-nuclear, where $\tau_1$ and $\tau_2$ are functorial
tensor products with $\tau_1 \leq \tau_2$, then $\cl S$ has the same property. 
\end{lemma}

\begin{proof}
Let $\phi: \cl S \rightarrow \cl T$ and $\psi: \cl T \rightarrow \cl S$ be the ucp maps such that $\psi \circ \phi (s) = s$ for every $s$ in $\cl S$. Let
$\cl R$ be any operator system. Then, by using the functoriality we have that
$$
\cl S \otimes_{\tau_1} \cl R \xrightarrow{\phi \otimes id} \cl T \otimes_{\tau_1} \cl R = \cl T \otimes_{\tau_2} \cl R \xrightarrow{\psi \otimes id} 
\cl S \otimes_{\tau_2} \cl R 
$$
is a sequence of ucp maps such that the composition is the identity. Since $\tau_1 \leq \tau_2$ we have that $\cl S \otimes_{\tau_1} \cl R = \cl S \otimes_{\tau_2} \cl R$.
Thus, $\cl S$ is $(\tau_1,\tau_2)$-nuclear.
\end{proof}
Since WEP, equivalently DCEP for C*-algebras, coincides with (el,max)-nuclearity, it follows that $C^*(\mathbb{F}_{\infty})$ has WEP if
and only if $C^*(\mathbb{F}_{2})$ has WEP. By a similar argument the above conditions are equivalent to the statement   
$C^*(\mathbb{F}_{2}) \otimes_{min} C^*(\mathbb{F}_{2}) =  C^*(\mathbb{F}_{2}) \otimes_{max} C^*( \mathbb{F}_{2})$. We also
remark that Kirchberg's WEP characterization can be given as follows, which will be useful when we express WEP in terms of
a tensor product with a lower dimensional operator system:
\begin{theorem}
The following are equivalent for a unital C*-algebra $\cl A$:
\begin{enumerate}
 \item $\cl A$ has WEP.
 \item $\cl A \otimes_{min} C^*(\mathbb{F}_{\infty})         =   \cl A \otimes_{max} C^*(\mathbb{F}_{\infty})$.
 \item  $\cl A \otimes_{min} C^*(\mathbb{F}_2)         =   \cl A \otimes_{max} C^*(\mathbb{F}_2)$.
\end{enumerate}
\end{theorem}
\begin{proof}
Equivalence of (1) and (2) is the Kirchberg's WEP characterization. To see that (3) implies (2)
we again use the fact that the identity on $C^*(\mathbb{F}_{\infty})$ factors through ucp maps
on $C^*(\mathbb{F}_{\infty})$. So let $\phi: C^*(\mathbb{F}_{\infty}) \rightarrow C^*(\mathbb{F}_2)$ and
$\psi: C^*(\mathbb{F}_2) \rightarrow C^*(\mathbb{F}_{\infty})$ be the ucp maps whose composition is the identity
on $C^*(\mathbb{F}_{\infty})$. Now, suppose (3) holds. By using the functoriality of min and max we have that
$$
\cl A \otimes_{min} C^*(\mathbb{F}_{\infty}) \xrightarrow{id \otimes \phi} \cl A \otimes_{min} C^*(\mathbb{F}_2) = 
\cl A \otimes_{max} C^*(\mathbb{F}_2) \xrightarrow{id \otimes \psi} 
\cl A \otimes_{max} C^*(\mathbb{F}_{\infty}).
$$
is a sequence of ucp maps such that the composition is the identity. Thus (2) holds. (2) implies
(3) is similar.
\end{proof}

Since WEP and LLP has natural extensions to general operator systems it is natural to approach Kirchberg's
Conjecture  from this perspective. We define $\cl S_n$ as the operator system in $C^*(\mathbb{F}_n)$
generated by the unitary generators, that is, 
$$
\cl S_n = span \{g_1,...,g_n,e,g_1^*,...,g_n^*\}  \subset C^*(F_n).
$$
$\cl S_n$ can also be considered as the universal operator system generated by
$n$ contractions as it satisfies the following universal property: Every function
$f:\{g_i\}_{i=1}^n \rightarrow \cl T$ with $\|f(g_i)\| \leq 1$ extends uniquely to a ucp map $\varphi_f:S_n\rightarrow
\cl T$
(in an obvious way).
$$
\xymatrix{
\{g_i\}_{i=1}^n     \ar[rrr]^{f}    \ar@{_{(}->}[d]  &  &  &  \cl T \\
\cl S_n \ar@{.>}[rrru]_{\varphi_f} & & &
}
$$
The proof this property relies on the unitary dilation of a contraction and the
reader may refer to the discussion in \cite[Sec. 9]{kptt2}. From this one can
easily deduce that $\cl S_n$ has the lifting property. Indeed, let $\varphi: \cl S_n
\rightarrow \cl A / I$ is a ucp map where $I \subset \cl A$ is an ideal, unital
C*-algebra couple. Let $\varphi(g_i) = a_i +I$ for $i=1,...,n$. Since
C*-algebra ideal quotients are proximinal (see \cite[Lem. 2.4.6.]{pi} e.g.)
there exists $b_i$ in $\cl A$ such that $b_i+I = a_i+I$ with $\|b_i\| =
\|a_i+I\|$. Since a ucp map is contractive we have that $\|a_i+I\| \leq 1$ and
so $\|b_i\|\leq 1$. Therefore the function $g_i\mapsto b_i$ extends uniquely to
a ucp map. It is elementary to show that this map is a lift of $\varphi$.

An operator subsystem $\cl S$ of a C*-algebra $\cl A$ is said to \textit{contain enough unitaries} if
there is a collection of unitaries in $\cl S$ which generates $\cl A$ as a C*-algebra, that is, $\cl A$ is the
smallest C*-algebra that contains these unitaries. This notion is inspired by a work of Pisier (see Chp. 13 of \cite{pi})
and in \cite{kptt2} it was shown that several nuclearity properties of
$\cl A$ can be deduced from $\cl S$ (see \cite[Cor. 9.6]{kptt2}).

\begin{lemma}
Let $\cl A$ and $\cl B$ be unital C*-algebras and $\{u_{\alpha}\}$ be a collection of unitaries in $\cl A$
which generates $\cl A$ as a C*-algebra. If $\varphi: \cl A \rightarrow \cl B$ is a ucp map such that
$\varphi(u_{\alpha})$ is a unitary in $\cl B$ for every $\alpha$ then $\varphi$ is a  $*$-homomorphism.
\end{lemma}

\begin{proof}
This is an application of Choi's work on the multiplicative domains in \cite{Choi MulDom}. Since
$e = \varphi(u_{\alpha} u_\alpha^*)  = \varphi(u_{\alpha}) \varphi( u_\alpha)^* = \varphi(u_{\alpha}^* u_\alpha) = \varphi(u_{\alpha})^* \varphi( u_\alpha)$,
each $u_\alpha$ belongs to multiplicative domain of $\varphi$. These elements generates $\cl A$, thus, $\varphi$ is a $*$-homomorphism.
\end{proof}

\begin{lemma}\label{lem uniq ext}
Let $\cl S \subset \cl A$ contain enough unitaries and let $\cl B$ be a unital C*-algebra. Let
$\{u_{\alpha}\}$ be the collection of unitaries in $\cl S$ which generates $\cl A$. Suppose
$\varphi: \cl S \rightarrow \cl B$ is a ucp map such that $\varphi(u_{\alpha})$ is a unitary in
$\cl B$ for every $\alpha$. Then $\varphi$ extends uniquely to a ucp map on $\cl A$ which is
necessarily a $*$-homomorphism.
 \end{lemma}

\begin{proof}
Lemma 4.16  in \cite{kptt2} ensures that $\varphi$ extends to a $*$-homomorphism. So there exists a ucp extension
of $\varphi$ on $\cl A$. Also the above lemma implies that any ucp extension has to be a $*$-homomorphism.
Since $\{u_{\alpha}\}$ generates $\cl A$ and every extension coincides on  $\{u_{\alpha}\}$ it follows that
extension is unique.
\end{proof}

\begin{proposition}\label{prop enough=envelo.}
Suppose $\cl S \subset \cl A$ contains enough unitaries. Then $\cl A$ coincides with the enveloping
C*-algebra of $\cl S$, that is, the unique unital $*$-homomorphism $\pi: \cl A \rightarrow C^*_e(\cl S)$
which extends the inclusion of $\cl S$ in $ C^*_e(\cl S)$ is bijective.
\end{proposition}

\begin{proof}
Let $\{u_\alpha\}$ be the collection of unitaries in $\cl S$ which generates $\cl A$ as a C*-algebra.
Let $i$ be the inclusion of  $\cl S$ in $ C^*_e(\cl S)$. Note that the image $\{\pi(u_{\alpha}) = i(u_\alpha)\}$ 
of the unitary collection $\{u_\alpha\}$ form a set of unitaries and it generates the image of $\pi$ which coincides with $ C^*_e(\cl S)$.
We can represent $\cl A$ into a $B(H)$ as a C*-subalgebra. Now, by Arveson's extension theorem,  
the inclusion of $\cl S$ in $\cl A \subset B(H)$ extends to ucp
map $\varphi$ on  $ C^*_e(\cl S)$. Note that $\varphi(i(u_{\alpha})) = u_{\alpha}$, that is, $\varphi$ maps a collection
of unitaries, which generates $ C^*_e(\cl S)$, to a collection of unitaries in $B(H)$. Now by using the above lemma
$\varphi$ must be a unital $*$-homomorphism. Moreover, since the image of $\{i(u_\alpha)\}$ stays
in $\cl A$ and generates $\cl A$, the image of $\varphi$ is precisely $\cl A$. The rigidity of the enveloping C*-algebra ensures that $\varphi$ is one to one too. 
Note that $\varphi^{-1}$ is again a unital $*$-homomorphism such that $\varphi^{-1}(s) = i(s)$ for every $s$ in $\cl S$. Now the universal
property of the enveloping C*-algebras ensure that $\pi = \varphi^{-1}$, thus $\pi$ is bijective.
\end{proof}

Despite this result we still prefer to use the term ``contains enough unitaries''. Our
very first example is, of course, $\cl S_n \subset C^*(\mathbb{F}_n)$. This also means that $C^*_e(\cl S_n) = C^*(\mathbb{F}_n)$.
It is also important to
remark that not every operator system contains enough unitaries in its enveloping C*-algebra. 
The following is an improvement of Proposition 9.5 of \cite{kptt2}:

\begin{proposition}\label{pr enough unitary}
Suppose $\cl S \subset \cl A$ and $\cl T \subset \cl B$ contains enough unitaries. Then
$$
\cl S \otimes_{min} \cl T \subset \cl A \otimes_{max} \cl B \Longrightarrow \cl A \otimes_{min} \cl B= \cl A \otimes_{max} \cl B.
$$
\end{proposition}
\begin{proof}
Let $\{u_{\alpha}\}$ and $\{v_{\beta}\}$ be unitaries in $\cl S$ and $\cl T$ that generates $\cl A$ and $\cl B$, respectively.
By using the injectivity of the minimal tensor product we have the inclusion $\cl S \otimes_{min} \cl T \subset \cl A \otimes_{min} \cl B$.
It is not hard to see that the unitaries $\{u_{\alpha} \otimes v_{\beta}\}$, which belongs to $\cl S \otimes_{min} \cl T$, generates $ \cl A \otimes_{min} \cl B$.
It is also clear that the inclusion $\cl S \otimes_{min} \cl T \hookrightarrow \cl A \otimes_{max} \cl B$ maps these unitaries to unitaries again. Thus,
by Lemma \ref{lem uniq ext}, this inclusion extends uniquely to a $*$-homomorphism which is necessarily the identity. So we conclude
that $\cl A \otimes_{min} \cl B= \cl A \otimes_{max} \cl B$.
\end{proof}

\begin{corollary}\label{co enough unitary}
Suppose $\cl S \subset \cl A$ and $\cl T \subset \cl B$ contains enough unitaries. Then
$$
\cl S \otimes_{min} \cl T = \cl S \otimes_{c} \cl T \Longrightarrow \cl A \otimes_{min} \cl B= \cl A \otimes_{max} \cl B.
$$
\end{corollary}
\begin{proof}
Let $\cl S \otimes_\tau \cl T$ be the operator system tensor product arising from the inclusion $\cl A \otimes_{max} \cl B$. Clearly
min $ \leq \; \tau \; \leq $ c. (Note: c is the maximal commuting tensor product.) Since min and c coincides on $\cl S \otimes \cl T$
we have that $\cl S \otimes_{min} \cl T \subset \cl A \otimes_{max} \cl B$. Thus, by Proposition \ref{pr enough unitary}, the result follows.
\end{proof}

\begin{theorem}\label{thm new WEP}
The following are equivalent for a unital C*-algebra $\cl A$:
\begin{enumerate}
 \item $\cl A$ has WEP (equivalently DCEP).
 \item $\cl A \otimes_{min} \cl S_2 = \cl A \otimes_{max} \cl S_2$.
\end{enumerate}
\end{theorem}
\begin{proof}
We already know that WEP and DCEP are equivalent for C*-algebras. Now suppose (1). Since WEP coincides
with (el,max)-nuclearity and $\cl S_2$ has the lifting property (equivalently (min,er)-nuclearity) 
(also keeping in mind that it is written on the right  hand side)
we have
$$
\cl A \otimes_{min} \cl S_2 = \cl A \otimes_{el} \cl S_2 = \cl A \otimes_{max} \cl S_2.
$$
Conversely suppose (2) holds. Since $\cl S_2$ contains enough unitaries in $C^*(\mathbb{F}_2)$ (and
$\cl A$ contains enough unitaries in itself), by the above corollary, we obtain that  $\cl A \otimes_{min} C^*(\mathbb{F}_2) = \cl A \otimes_{max} C^*(\mathbb{F}_2)$.
Thus $\cl A$ has WEP.
\end{proof}

In the following theorem the equivalence of (1)-(4) is Theorem 9.1. and 9.4 of \cite{kptt2}. So we will
only prove that these are equivalent to (5) and (6), which express KC in terms of a five dimensional
operator system problem.

\begin{theorem}\label{thm new KC}
The following are equivalent:
\begin{enumerate}
 \item Kirchberg Conjecture has an affirmative answer.
 \item $\cl S_n$ has DCEP for every $n$.
 \item $\cl S_n \otimes_{min} \cl S_n =  \cl S_n \otimes_{c} \cl S_n$ for every
$n$.
 \item Every finite dimensional operator system with the lifting property has DCEP.
  \item $\cl S_2$ has DCEP.
 \item $\cl S_2 \otimes_{min} \cl S_2 =  \cl S_2 \otimes_{c} \cl S_2$.
\end{enumerate}
\end{theorem}
\begin{proof}
The equivalence of (1),(2),(3) and (4) follows from Theorem 9.1 and 9.4 of \cite{kptt2}. These conditions
clearly imply (5) and (6). Moreover (5) implies (6). In fact, we know that $\cl S_2$ has the lifting property 
(equivalently (min,er)-nuclearity). If we assume that it has DCEP (equivalently (el,c)-nuclearity) then (also
keeping in mind that one of the $\cl S_2$ is written on the right  hand side)  it follows that
$$
\cl S_2 \otimes_{min} \cl S_2 = \cl S_2 \otimes_{er} \cl S_2 = \cl S_2 \otimes_{c} \cl S_2.
$$
Conversely suppose that (6) holds. Since $\cl S_2$ contains enough unitaries in $C^*(\mathbb{F}_2)$, by Corollary \ref{co enough unitary}, 
it follows that  $C^*(\mathbb{F}_2) \otimes_{min} C^*(\mathbb{F}_2) = C^*(\mathbb{F}_2) \otimes_{max} C^*(\mathbb{F}_2)$, that is,
the Kirchberg Conjecture has an affirmative answer.
\end{proof}

$ $

\section{The Representation of the Minimal Tensor Product}

Suppose $V$ and $W$ are vector spaces with $dim(V)<\infty$, then it is well
known that there is a bijective correspondence between $ V \otimes W  
\cong L(V^*,W) $ where  $ L(V^*,W) $ is the vector space of linear
maps from $V^*$ into $W$. The bijective linear map is given by 
$$
 \Sigma v_i\otimes w_i \mapsto \ \widehat{\Sigma v_i\otimes w_i}\;\; \mbox{
where } \;\; \widehat{ \Sigma v_i\otimes w_i} (f) =  \Sigma f(v_i) w_i.
$$
This identification plays an important role in the characterization of minimal
tensor products both in Banach space and operator space theory (see \cite{bp}
e.g.) . (Note that every linear map defined from a finite dimensional operator
space is completely bounded which can be seen in \cite{Pa2}.) The following is
the operator system variant of this well known correspondence. In this section
we will study various application of this equivalence. The first part is \cite[Lem. 8.4]{kptt}.

\begin{proposition}\label{mindual}
Let $\cl S$ and $\cl T$ be operator systems where $dim(\cl S)$ is finite. Then
there is a bijective correspondence between
$$
(\cl S \otimes _{min} \cl T)^+ \longleftrightarrow CP(\cl S^d,\cl T).
$$
That is, a finite sum $\Sigma{s_i \otimes t_i}$ is positive if and only if the
corresponding map $\widehat{\Sigma{s_i \otimes t_i}}$ is completely positive
from
$\cl S^d$ into $\cl T$. In particular every linear map from $\cl S^d$ into $\cl
T$ can be written as a linear combination of completely positive maps.
\end{proposition}

\begin{proof}
The bijective correspondence is already shown in \cite{kptt}. Now let  $\cl S =
span \{e\! =\! s_1,s_2,...,s_n \}$  written in the special basis form as
in Lemma \ref{lem special basis} and let  $\cl S^d =
span\{\delta_1, \delta_2,...,\delta_n\}$ written as the corresponding dual
basis form. Consider a linear map $\varphi : \cl S^d \rightarrow \cl T$ where
$\varphi(\delta_i) = t_i$. Now  $\Sigma (s_i \otimes t_i)$ can be written as
linear combination of positives in $\cl S \otimes_{min} \cl T$, say $\Sigma (s_i
\otimes t_i) = x_1 - x_2 + i x_3 - i x_4$ where each $x_i$ is positive. By the
first part, the corresponding maps $\hat{x_i}$ are completely positive
from $\cl S^d$ into $\cl T$ and clearly $\varphi  = \hat{x_1} -\hat{x_2} +
i\hat{x_3} - i\hat{x_4}$. This finishes the proof.
\end{proof}

\begin{corollary}
If $\cl S$ and $\cl T$ are operator systems with $dim(\cl S) < \infty$
then every linear map from $\cl S$ to $\cl T$ can be written as a linear
combination of completely positive maps.
\end{corollary}

\noindent \textbf{Aside:} Supposing $\cl S$ and $\cl T$ are operator systems
with $dim(\cl S) < \infty$ then $CB(\cl S, \cl T)$ has a structure of an
operator system: The involution is given by $\varphi^*(s) = \varphi(s^*)^*$ and
the positive cones structures can be describe as
$$
(\varphi_{ij})  \in M_n(CB(\cl S,\cl T)) \mbox{ is positive if the map }   
\cl S \ni  s \mapsto (\varphi_{ij}(s)) \in M_n(\cl T) \mbox{ is cp}. 
$$
The non-canonical Archimedean order unit can be chosen to be $\tilde{\delta} =
\delta(\cdot) e_{\cl T}$ where $\delta$ is a faithful state on $\cl S$.
Moreover we obtain the following identity
$$
\cl S^d \otimes_{min} \cl T = CB(\cl S,\cl T)
$$
unitally and completely order isomorphicaly. Of course, this also means that $
\cl S \otimes_{min} \cl T = CB(\cl S^d,\cl T) $ where the identity of $CB(\cl
S^d,\cl T)$ is chosen to be $\hat{e}(\cdot) e_{\cl T}$.

$ $

Proposition \ref{mindual} has several important consequences. We want to start
with the following duality property between the minimal and the maximal tensor products given in
\cite{pf}. We also include the proof as it relies on the representation of the tensor products.

\begin{theorem}[Farenick, Paulsen]\label{thm dualminmax}
For finite dimensional operator systems $\cl S$ and $\cl T$ we have the following unital complete
order isomorphisms:
$$
(\cl S \otimes_{max} \cl T)^d = \cl S^d \otimes_{min} \cl T^d \;\;\mbox{ and } \;\; (\cl S \otimes_{min} \cl T)^d = \cl S^d \otimes_{max} \cl T^d.
$$
More precisely, if $\delta_{\cl S}$ and $\delta_{\cl T}$ are faithful states on $\cl S$ and $\cl T$, resp., which we set as Archimedean
order units,  then $\delta_{\cl S} \otimes \delta_{\cl T}$
is again a faithful state on $\cl S \otimes_{min} \cl T$ and $\cl S \otimes_{max} \cl T$ when considered as a linear functional.
\end{theorem}

\begin{proof}
We first show that  $\cl S \otimes_{min} \cl T$ and $ (\cl S ^d \otimes_{max} \cl T^d)^{d}$ are completely order isomorphic.
Note that
$$
(\cl S \otimes_{min} \cl T ) ^+ = CP(\cl S^d , \cl T) = (\cl S ^d \otimes_{max} \cl T^d)^{d,+}. 
$$
Here the second equation follows from the representation
of the maximal tensor product that we discussed in Subsection \ref{subsec max}. Therefore, we obtain that
a positive linear functional on $\cl S ^d \otimes_{max} \cl T^d$ corresponds to
a positive element in $\cl S \otimes_{min} \cl T$. This shows that the bijective linear map
$$
\cl S \otimes_{min} \cl T \rightarrow  (\cl S ^d \otimes_{max} \cl T^d)^{d}  \;\;\; s \otimes t \mapsto s \dot{\otimes} t 
\mbox{ where } s \dot{\otimes} t (\Sigma f_i \otimes g_i) = \Sigma f_i(s)g_i(t)
$$
is an order isomorphism. To see that it is an complete order isomorphism we can reduce the matricial
levels to a ground level as follows. First note that
$$
M_n(\cl S) \otimes_{min} \cl T \mbox{ and }  \left( M_n(\cl S)^d \otimes_{max} \cl T^d \right)^d
$$
are order isomorphic. The left hand side can be identified with $M_n(\cl S \otimes_{min} \cl T)$. 
On the other hand, for any operator system $\cl R$ we have the  identification  $M_n(\cl R^d) = (M_n(\cl R))^d  $ 
given by $(f_{ij}) \mapsto F$ where $ F (r_{ij}) = \Sigma f_{ij}(r_{ij}) $. In fact, we first
identify $M_n(\cl R^d)$ with linear operators from $\cl R$ into $M_n$ (where we use the definition of positivity) and 
these linear operators are identified with linear functionals on $M_n(\cl R)$ (see \cite[Thm. 6.1.]{Pa}, e.g.). By the associativity of the 
maximal tensor product we have that the right hand side can be identified with
$$
\left( M_n(\cl S^d) \otimes_{max} \cl T^d \right)^d = \left( M_n(\cl S^d \otimes_{max} \cl T^d) \right)^d =M_n\left(( \cl S^d \otimes_{max} \cl T^d)^d \right).
$$
Thus the above map is completely order isomorphic. We may suppose that these operator systems
have the same unit by simply declaring $e_{\cl S} \dot{\otimes} e_{\cl T}$ as the Archimedean order unit on 
$ (\cl S ^d \otimes_{max} \cl T^d)^d$. (Since both of these matrix ordered spaces are completely order isomorphic, clearly,
$e_{\cl S} \dot{\otimes} e_{\cl T}$ plays the same role on $ (\cl S ^d \otimes_{max} \cl T^d)^d$.) Finally by taking appropriate duals,
we obtain both first and second desired identifications. 
\end{proof}

This duality correspondence allows us to recover the following special case about the projectivity of the maximal
tensor product given in \cite{Han}.
\begin{theorem}
Let $\cl S$ and $\cl T$ be finite dimensional operator systems and $J \subset \cl S$ be a null subspace.
Then $J \otimes \cl T \subset \cl S \otimes_{max} \cl T$ is a null subspace and we have that
$$
( \cl S \otimes_{max} \cl T) / (J \otimes \cl T) = (\cl S / J) \otimes_{max} \cl T.
$$
In other words, the induced map $ \cl S \otimes_{max} \cl T\longrightarrow (\cl S / J) \otimes_{max} \cl T$ is
a unital quotient map.
\end{theorem}

\begin{proof}
Proposition \ref{prop dual of quotient} ensures that $(\cl S/J)^d$ is an operator subsystem of $\cl S^d$. Thus,
by using the injectivity of the minimal tensor product, we have that 
$$
(\cl S/J)^d  \otimes_{min} \cl T^d \subset \cl S^d  \otimes_{min} \cl T^d.
$$
Now, Theorem \ref{thm dual of inc} (and the remark thereafter) ensure that the adjoint of this map is a quotient map
whose kernel is a null subspace. Thus, by using the above result, the adjoint of this inclusion, i.e., the natural map below
$$
\cl S \otimes_{max} \cl T\longrightarrow (\cl S / J) \otimes_{max} \cl T
$$
is a quotient map. By a dimension count argument its kernel is $J \otimes \cl T$ which is a null subspace. 
\end{proof}

With the following lemma we resolve some technical issues. Its proof is again based on the representation of the minimal tensor product.

\begin{lemma}\label{lem liftdualquotient}
Let $\cl S$ be a finite dimensional operator system, $\cl A$ be a C*-algebra and $I$ be an ideal in $\cl A$. Then the following are equivalent:
\begin{enumerate}
 \item For all $n$, every ucp map $\varphi: \cl S \rightarrow M_n(\cl A) / M_n(I)$ has cp lift on $M_n(\cl A)$.

 \item For all $n$, every cp map $\varphi: \cl S \rightarrow M_n(\cl A) / M_n(I)$ has cp lift on $M_n(\cl A)$.

\item We have the unital complete order isomorphism
$$
(\cl S^d \otimes_{min} \cl A ) / (\cl S^d \otimes I) = \cl S^d \otimes_{min} (\cl A /I).
$$
\end{enumerate}
\end{lemma}
\begin{proof}
We first remark that if we replace ``for all $n$''  with ``for $n = 1$'' in (1) and (2) and remove ``complete'' in (3) and
prove the lemma this way then the the original arguments automatically satisfied. In fact this follows
from the identifications
$$
M_n((\cl S^d \otimes_{min} \cl A ) / (\cl S^d \otimes I) ) = M_n(\cl S^d \otimes_{min} \cl A ) / M_n(\cl S^d \otimes I ) = (\cl S^d \otimes_{min} M_n(\cl A) ) / (\cl S^d \otimes M_n(I))
$$
and $M_n(\cl S^d \otimes_{min} (\cl A /I)) = \cl S^d \otimes_{min} M_n(\cl A /I)$. So we will prove the equivalences only for the ground level.
Clearly (2) implies (1). (3) implies (2) is also easy. The cp map  $\varphi: \cl S \rightarrow \cl A / I$ corresponds to a positive element
in $\cl S^d \otimes_{min} (\cl A /I)$. Since we assumed (3) and the first quotient in (3) is proximinal (see Cor. 5.15 of \cite{kptt2} e.g.) it follows that $u$ is
quotient of a positive element $v$ in $\cl S^d \otimes_{min} \cl A$. Now, again by using the representation of minimal tensor product,
$v$ corresponds a cp map $\tilde{\varphi}: \cl S \rightarrow \cl A$. It is not hard to show that $\tilde{\varphi}$ is a lift of $\varphi$.

We finally show that (1) implies (3). First note that (1) implies the following: Whenever $\phi: \cl S \rightarrow \cl A / I$ is a cp map with
$\phi(e)$ is invertible then $\phi$ has cp lift on $\cl A$. In fact if we set $\psi = \phi(e)^{-1/2} \phi(\cdot)  \phi(e)^{-1/2}$ then
$\psi$ is a ucp map and hence has a cp lift $\tilde{\psi}$ on $\cl A$. Now, if $a$ is in $\cl A^+$ with $a + I = \phi(e)^{1/2}$ then it is easy to see that
the cp map $a \tilde{\psi}(\cdot) a$ is a lift of $\phi$. Secondly, we remark that the induced map from $(\cl S^d \otimes_{min} \cl A ) / (\cl S^d \otimes I)$
to $\cl S^d \otimes_{min} (\cl A /I)$ is already bijective and ucp (see Remark \ref{rem ex for fd}). Thus we need to show that its inverse is positive. 
So let $u$ be positive in $\cl S^d \otimes_{min} (\cl A /I)$ and set  $u_\epsilon = u + \epsilon 1$ for $\epsilon>0$, where
$1$ is the unit of $\cl S^d \otimes_{min} (\cl A /I)$. (Note: $1 = f \otimes \dot{e}_{\cl A}$ where $f$ is a faithful sate on $\cl S$.)
Since $u$ and $u_\epsilon$ are positive elements they corresponds to cp maps $\varphi$ and $\varphi_\epsilon$ 
from $\cl S $ into $ \cl A / I$, respectively. It is not hard to see that
$\varphi_\epsilon(e_{\cl S}) = \varphi(e_{\cl S}) + \epsilon  \dot{e}_{\cl A}$. This means that $\varphi_\epsilon(e_{\cl S}) $ is invertible
and so it has a cp lift $\tilde{\varphi_\epsilon}$ from $\cl S$ into $\cl A$. This again corresponds to a positive element $U_{\epsilon}$
in $\cl S^d \otimes_{min} \cl A$. Now it is not hard to see that the positive element $U_{\epsilon} + \cl S^d \otimes I$ is the inverse image of
 $u_{\epsilon} = u + \epsilon 1$ for every $\epsilon>0$. This is enough to conclude that two operator systems in (3) are order isomorphic.
\end{proof}

\begin{theorem}\label{exactdualLP}
Let $\cl S$ be a finite dimensional operator system. Then $\cl S$ has the lifting property if
and only if $\cl S^d$ is exact (and vice versa). In other words, $\cl S$ is
(min,er)-nuclear if and only if $\cl S^d$ is (min,el)-nuclear.
\end{theorem}

\begin{proof}
The proof is based on Lemma \ref{lem liftdualquotient}. If $\cl S$ has the lifting property then
(1) in the same lemma will be satisfied for every C*-algebra and ideal. Thus (3) implies that $\cl S^d$
is exact. The reverse direction similar. Since $\cl S^{dd} = \cl S$ we clearly have that
$\cl S$ is exact  if and only if $\cl S^d$ has the lifting property.
\end{proof}

\begin{theorem}
If the Kirchberg conjecture has an affirmative answer then, in the finite dimensional case, 
C*-nuclearity is preserved under duality, that is, if $\cl S$ is C*-nuclear then $\cl S^d$ is again
C*-nuclear.
\end{theorem}

\begin{proof}
Let $\cl S$ be a finite dimensional C*-nuclear operator system. In particular $\cl S$ is exact and
has the lifting property. By the above result $\cl S^d$ has both of these properties.
Now if the Kirchberg conjecture is true then Theorem \ref{thm new KC} implies that $\cl S^d$ has DCEP. It is easy
to see that exactness and DCEP together imply C*-nuclearity. Thus, $\cl S^d$ is C*-nuclear.
\end{proof}

The local lifting property of a C*-algebra, in general, does not pass to its
quotients by ideals. In fact it is well known that every C*-algebra is the 
quotient of a full C*-algebra of a free group which has the local lifting
property however there are C*-algebras without this property. On the
finite dimensional operator systems this situation is different:

\begin{theorem}\label{thm lp-quotient}
Let $\cl S$ be a finite dimensional operator system and let $J$ be a null
subspace of $\cl S$. If $\cl S$ has the lifting property then $\cl S / J$ has
the same property.
\end{theorem}

\begin{proof}
Recall from Proposition \ref{prop dual of quotient} that $(\cl S /J)^d$ is an
operator subsystem of $\cl S^d$. Since $\cl S$ has the lifting property then $\cl
S^d$ is exact by Theorem \ref{exactdualLP}. Proposition \ref{prop exac pass ss} states that exactness passes to operator subsystems so
$(\cl S /J)^d$ is exact and consequently using Theorem \ref{exactdualLP} again
it follows that $\cl S / J$ has the lifting property.
\end{proof}

\begin{example}\label{exam MnJn}
We define $J_n \subset M_n$ as the subspace which includes all the
diagonal operators with 0 trace. Clearly $J_n$ is a null subspace and
consequently, by Proposition \ref{prop nullsubspace}, it is a kernel. Since
$M_n$ is a nuclear C*-algebra, it is a (min,max)-nuclear operator system. In
particular, it is (min,er)-nuclear equivalently has the lifting property. Thus, by
the above theorem $M_n/J_n$ has the lifting property. We will come back to this
example in later sections.
\end{example}

The lifting property is also stable when passing to universal C*-algebras. The
following result is an unpublished work of Ivan Todorov which he informed me of
during this research. The operator space analogue can be seen in \cite{Oz3}.

\begin{theorem}\label{thm lift univ}
Let $\cl S$ be a finite dimensional operator system. Then $\cl S$ has lifting
property if and only if $C^*_u(\cl S)$ has LLP.
\end{theorem}

 \begin{proof}
First suppose that $\cl S$ has the lifting property. Let $\pi : C^*_u(\cl S) 
\rightarrow \cl A / I$ be a unital $*$-homomorphism. (Note: As pointed out in
\cite[Rem. 16.3 (ii)]{pi} it is enough to consider the the unital
representations to verify the LLP of a C*-algebra.) Let $\pi_0$ be the
restriction of $\pi$ on $\cl S$. By using the local lifting property of $\cl S$
we have a ucp map $\varphi$ from $\cl S$ to $\cl A$ which lifts $\pi_0$. Let
$\rho: C_u^*(\cl S) \rightarrow \cl A$ be the unital $*$-homomorphism extending
$\varphi$. It is elementary to show that $\rho$ is a lift of $\pi$. Conversely
suppose that $C^*_u(\cl S)$ has LLP. Let $\varphi : \cl S \rightarrow \cl A /
I$ be a ucp map. Let $\pi : C^*_u(\cl S) \rightarrow \cl A / I$ be the
associated $*$-homomorphism. Now since $\cl S$ is a finite dimensional operator
subsystem of $C^*_u(\cl S)$, the restriction of $\pi$ on $\cl S$, namely
$\varphi$, lifts to a ucp map on $\cl A$. This completes the proof.
\end{proof}

For some other applications the following result will be useful.

\begin{proposition}
Suppose $\cl S$ and $\cl T$ are two finite dimensional operator systems with the
same dimensions. Then there is a surjective ucp map $\varphi : \cl S 
\rightarrow \cl T$.
\end{proposition}

\begin{proof}
Let $\cl S = span \{e \! = \! s_1,s_2,...,s_n \}$  and $\cl T = span \{e  \! = 
\! t_1,t_2,...,t_n \}$  written in the special basis form as in Lemma \ref{lem
special basis}. Let $\cl S^d = span\{\delta_1, \delta_2,...,\delta_n\}$  given
in the corresponding dual basis form. Recall that $\delta_1$ is an Archimedean
order unit for $\cl S^d$. Note that
$$
\delta_2 \otimes t_2 + \cdots \delta_n \otimes t_n
$$
is a self-adjoint element of $\cl S^d \otimes_{min} \cl T$ and consequently
there is a large $M$ such that
$$
\delta_1 \otimes e+ (\delta_2 \otimes t_2 + \cdots \delta_n \otimes t_n) / M
$$
is positive. Now by using Proposition \ref{mindual} it is elementary to see that
the corresponding completely positive map from $\cl S$ to $\cl T$ is unital
and surjective. 
\end{proof}

\begin{corollary}\label{sur}
Suppose $\cl S$ and $\cl T$ are operator systems with $dim(\cl T)$ finite
and $dim(\cl T) \leq dim(\cl S)$. Then there is a surjective ucp map from $\cl
S$ to $\cl T$.
\end{corollary}

\begin{proof}
Suppose $dim(\cl T) = n $ and let $\cl S_0$ be an $n$-dimensional operator
subsystem of $\cl S$. By using the above proposition there is a surjective ucp
map from $\cl S_0$ onto $\mathbb{C}^n$. Since $\mathbb{C}^n$ is injective this
map extends to a ucp map from $\cl S$ on $\mathbb{C}^n$. Now again by using
the above proposition we have surjective ucp map from $\mathbb{C}^n$ onto $\cl T$.
Composition of these two maps is surjective and ucp.
\end{proof}

In \cite{kw} Kirchberg and Wasserman exemplify  the behavior of universal
C*-algebras of some low dimensional operator systems. More
precisely they show that:
\begin{enumerate}
 \item $C^*_u(\mathbb{C}^2)$ is unitally $*$-isomorphic to $C[0,1]$, in
particular, it is nuclear.
 \item $C^*_u(\mathbb{C}^3)$ is not exact.
\end{enumerate}
By using Corollary \ref{sur} we obtain the following:
\begin{proposition} \label{prop two dim}
$\mbox{ }$
\begin{enumerate}
 \item If $\cl S$ is a two dimensional operator system then $C_u^*(\cl S)$ is
nuclear. In particular $\cl S$ is (min,c)-nuclear (equivalently C*-nuclear).
\item If $\cl S$ is an operator system with $dim(\cl S) \geq 3$ then
$C_u^*(\cl S)$ is not exact.
\end{enumerate}
\end{proposition}

\begin{proof}
Both parts of the proof are based on Corollary \ref{sur}. Suppose $\cl S$ is a two
dimensional operator system. Let $\varphi: \mathbb{C}^2 \rightarrow \cl S$ be
a surjective ucp map and let $\pi : C^*_u(\mathbb{C}^2) \rightarrow
C^*_u(\cl S)$ be the corresponding unital $*$-homomorphism. Note that $\pi$ is
surjective so $C^*_u(\mathbb{C}^2)/ker(\pi)$ and $C^*_u(\cl S) $ are
$*$-isomorphic C*-algebras. This means that $C^*_u(\cl S) $ is quotient of a
nuclear C*-algebra and consequently it is nuclear (see \cite{CE} e.g.). To see
that $\cl S$ is (min,c)-nuclear first fix an operator system $\cl T$. We have
the inclusions
$$
\cl S \otimes_{min} \cl T \subset
C^*_u(\cl S) \otimes_{min} \cl T    \mbox{ and } \cl S \otimes_{c} \cl T
\subset
C^*_u(\cl S) \otimes_{max} \cl T.
$$
Since the tensor products on the right coincide it follows that $\cl S$ is
(min,c)-nuclear.

Now let $\cl S$ be an operator system with $\dim(\cl S) \geq 3$. Assume for a
contradiction that $C^*_u(\cl S)$ is exact. Let  $\varphi: \cl S
\rightarrow \mathbb{C}^3$ be a surjective ucp map and let $\pi : C^*_u(\cl S)
\rightarrow  C^*_u(\mathbb{C}^3) $ be the corresponding unital $*$-homomorphism
which is surjective. This means that $C^*_u(\mathbb{C}^3)$ is a quotient of an
exact C*-algebra. So another result of Kirchberg \cite{Ki1}, which states
that exactness passes to quotients by ideals, requires $C^*_u(\mathbb{C}^3)$ to
be exact which is a contradiction. 
\end{proof}

For another application of Corollary \ref{sur} we need some preliminary
results. If $X$ is an operator space then there is an, essentially unique,
operator system $\cl T_X$ together with a  completely isometric inclusion $ i :
X \hookrightarrow \cl T_X$ such that it satisfies the following universal
property: For every completely contractive map $\phi:X\rightarrow \cl S$, where
$\cl S$ is an operator system, there exists a unique ucp map
$\varphi: \cl T_X \rightarrow \cl S$ such that $\varphi(i(x)) = \phi(x)$ for
every $x$ in $X$. 
$$
\xymatrix{
X\ar@{_{(}->}[d]_i    \ar[rr]^{cc \; \phi}  &  &    \cl S\\
\cl T_X \ar@{.>}[rru]_{ucp \; \varphi}
}
$$
To see the existence of $\cl T_X$ one can first consider the universal unital
C*-algebra $C_u^*\langle X \rangle$ of the operator space $X$. Recall that it
has the following universal property: Every completely contractive map defined
from $X$ into a unital C*-algebra $\cl A$ extends uniquely to a unital
$*$-homomorphism. (See \cite[Thm. 8.14]{pi} e.g.) Now let the span of $X$, $X^*$ and the
unit $e$ be
$\cl T_X$. (Also note that the image can be taken to an operator
system.) If $X_0$ is an operator subspace of $X$ then we have a unital
complete order embedding $\cl T_{X_0} \subset \cl T_X$. We leave the
proof of this as an exercise. Also, the following identification is immediate:
$$
C_u^*\langle X \rangle = C_u^*(\cl T_X).
$$

Recall that an operator space $X$ is said to have the $\lambda$-\textit{operator
space local lifting property} ($\lambda$-OLLP) if the following holds for every
unital C*-algebra $\cl A$ and ideal $I$ in $\cl A$. If $\phi: X \rightarrow \cl
A / I$ is a completely contractive (cc) map and $X_0$ is a finite dimensional
operator subspace of $X$ then $\phi|_{X_0}$ has a lift $\tilde{\phi_0}$ on $\cl
A$ with $\|\tilde{\phi_0}\|_{cb} \leq \lambda $. We claim that:

\begin{proposition}
Let $X$ be an operator space. Then $X$ has $1$-OLLP if and only if $\cl T_X$
has osLLP.
\end{proposition}

\begin{proof}
Let $\cl A$ be a unital C*-algebra and $I$ be an ideal in $\cl A$. First
suppose that  $X$ has $1$-OLLP. Let $\varphi: \cl T_X \rightarrow \cl A / I$ be
a ucp map and let $\cl T_0$ be a finite dimensional operator subsystem of $\cl
T_X$. Clearly we can find a finite dimensional subspace $X_0$ of $X$ such that
the operator system generated by $X_0$, which is actually $\cl T_{X_0}$,
contains $\cl T_0$. Note that $\varphi|_{X}$ is cc and so its restriction on
$X_0$ has a cc lift on $\cl A$. Now by using the universal property of $\cl
T_{X_0}$ we obtain a ucp map from $\cl T_{X_0}$ on $\cl A$. Now the restriction
of this map on $\cl T_0$ is a ucp lift on $\cl A$.

Conversely suppose  $\cl T_X$ has osLLP and let $\phi: X \rightarrow \cl A / I$
be a cc map. This map has a ucp extension $\varphi$ on $\cl T_X$. Let $X_0$ be
a finite dimensional operator subspace of $X$. Clearly $\cl T_{X_0}$ is a
finite dimensional operator subsystem of $\cl T_{X}$ and consequently
$\varphi$, when restricted to $\cl T_{X_0}$ has a ucp lift on $\cl A$. Finally
restriction of this lift on $X_0$ is cc. This finishes the proof.
\end{proof}

When $X = \mathbb{C}$, $\cl T_X$ is a three dimensional operator system. The
following is from \cite{Ki2}.
\begin{proposition}
The following are equivalent:
\begin{enumerate}
 \item The Kirchberg conjecture has an affirmative answer.
 \item $C_u^*\langle \mathbb{C} \rangle$ has WEP.
\end{enumerate}
\end{proposition}
Depending heavily on this characterization we can obtain further equivalences.
(The equivalence of (1) and (4) was pointed out by Vern Paulsen.)
\begin{proposition} The following are equivalent:
\begin{enumerate}
 \item The Kirchberg conjecture has an affirmative answer.
 \item There exists a three dimensional operator system $\cl S$ such that
$C_u^*(\cl S)$ has WEP.
 \item There exists an operator system $\cl S$ with $dim(\cl S) \geq 3$ such
that $C_u^*(\cl S)$ has WEP.
 \item $C_u^*(M_2)$ has WEP.
\end{enumerate}
\end{proposition}

\begin{proof}
Clearly (4) implies (3). To see that (3) implies (2), let $\cl S$ be an
operator system with $dim(\cl S) \geq 3$ such that $C_u^*(\cl S)$ has WEP.
Let $\cl T$ be a three dimensional operator system with the lifting property. (For
example $\mathbb{C}^3$). By using Corollary
\ref{sur}, we know that there is surjective ucp map $\varphi$ from $\cl S$ to
$\cl T$. Note that this ucp map extends to surjective *-homomorphism $\pi :
C_u^*(\cl S) \rightarrow C_u^*(\cl T) $. Since $C_u^*(\cl S) / ker(\pi)$ and
$C_u^*(\cl T)$ are *-isomorphic C*-algebras we obtain that $C_u^*(\cl T)$ is
QWEP. Also by Theorem \ref{thm lift univ}, $C_u^*(\cl T)$ has LLP. A well known
result of Kirchberg states that QWEP and LLP together imply WEP (\cite{Ki2}). Thus, (3) implies (2).
Now we will show (2) implies (1). By
using the above result of Kirchberg it is enough to prove that $C_u^*\langle
\mathbb{C} \rangle$ has WEP. Recall that $C_u^*\langle
\mathbb{C} \rangle = C_u^*(\cl T_\mathbb{C})$. Since $\mathbb{C}$ has 1-OLLP it
follows that $\cl T_\mathbb{C}$ has the lifting property. By Theorem \ref{thm
lift univ} $C_u^*(\cl T_\mathbb{C})$ has LLP. By using an argument that we used
in the implication (3) $\Rightarrow$ (2) it is easy to see that existence of a
three dimensional operator system with WEP implies $C_u^*(\cl T_\mathbb{C})$ is
QWEP. Consequently $C_u^*(\cl T_\mathbb{C})   =  C_u^*\langle
\mathbb{C} \rangle$ has WEP. Finally to see that (1) implies (4), note that
$C_u^*(M_2)$  has LLP (since $M_2$ has the lifting property). So assuming KC it
follows that $C_u^*(M_2)$ has WEP.
\end{proof}

$ $

\section{Further Exactness and Lifting Properties}

 We first want to review some instances where the operator space and the
operator system quotients are completely isometric. Then by using a result of
Ozawa \cite{Oz6}, we obtain simpler exactness and local liftability conditions for operator systems.
This follows the track of Pisier and Ozawa's approach for operator spaces (see Thm. 16.10 and Rem. 17.6 of \cite{pi}, e.g.). 
Let $\cl S$ be operator system $\cl A$ be a unital C*-algebra
and $I$ be an ideal in $\cl A$. As we pointed out in Subsection \ref{subsec
exact}, $\cl S \bar{\otimes} I \subset \cl S \hat{\otimes}_{min} \cl A$ is a
kernel. ($ \hat{\otimes}_{min} $ denotes the completed minimal tensor product
and $\bar{\otimes}$ is the closure of the algebraic tensor product in the larger
space.) Moreover, the canonical operator spaces structure on the operator system
quotient
$$
 (\cl S\hat{\otimes}_{min} \cl A) \; / \; (\cl S \bar{\otimes} I) 
$$
coincides with the operator space quotient of $\cl S \hat{\otimes}_{min} \cl A$
by its closed subspace $ \cl S \bar{\otimes} I$. (See \cite[Thm. 5.1]{kptt2}).
Also recall from Remark \ref{rem ex for fd} that when $\cl S$ is finite dimensional then the minimal tensor
of $\cl S$ with a C*-algebra is already a completed object so we will
use $\otimes_{min}$ instead of $\hat{\otimes}_{min}$. Similarly if $I$ is an
ideal in a C*-algebra $\cl A$ then $\cl S \bar{\otimes} \cl I$ coincides with
the algebraic tensor product $\cl S \otimes I$. So we omit the bar over the
tensor product.

$ $

\noindent \textbf{Notation:} For simplicity in the following results we let
$\mathbb{B}$ denote $B(l^2)$ and $\mathbb{K}$ stands for the ideal of compact
operators in $B(l^2)$.

$ $

Suppose $\cl A$ is a unital C*-algebra and $I$ is an ideal in $\cl A$. Let
$$
C = \{ \phi : \cl A \rightarrow \mathbb{B} : \phi \mbox{ is ucp and  } \phi(I)
\subseteq \mathbb{K} \}.
$$
For $\phi$ in $C$ we use the notation $\dot{\phi}$ for the induced map $\cl A /
I \rightarrow \mathbb{B} / \mathbb{K}$. If $X$ is a finite
dimensional operator space then $\bar{\phi}$ denotes the corresponding map
$$
 ( X \hat{\otimes} \cl A)  /  (X {\otimes} I)   \rightarrow 
 (X \hat{\otimes} \mathbb{B}) / ( X \otimes \mathbb{K} )  .
$$
where $\hat{\otimes}$ is the minimal operator space tensor product. We are
ready to state:
\begin{proposition}[Ozawa, \cite{Oz6}],
Let $X$ be a finite dimensional operator space. If $\cl A$ is a unital
separable C*-algebra and $\cl I$ is an ideal in $\cl A$ then for any $u$ in
$X \otimes \cl A / I$ we have
$$
\|  u  \|_{X \hat{\otimes} \cl A / I } = \sup_{\phi\in C} \| (id\otimes
\dot{\phi})(u) \|_{X \hat{\otimes} \mathbb{B}/\mathbb{K}}
$$
and for any $v$ in $ ( X \hat{\otimes} \cl A)  /  (X {\otimes} I) $
$$
\|  v  \|_{( X \hat{\otimes} \cl A)  /  (X {\otimes} I) } = \sup_{\phi\in C} \|
(id\otimes \bar{\phi})(v) \|_{(X \hat{\otimes} \mathbb{B}) / ( X \otimes
\mathbb{K} )}.
$$
\end{proposition}
Before stating the following result we want to emphasize that the the minimal
operator system and the minimal operator space tensor products coincide. (In
fact they are both spatial.) The exactness criteria in the next theorem is true
for every operator system which we included as a corollary.

\begin{theorem}\label{fdexactness}
Suppose $\cl S$ is a finite dimensional
operator system. Then $\cl S$ is exact if and only if 
$$
(\cl S \otimes_{min} \mathbb{B}) /(\cl S \otimes \mathbb{K})  \cong \cl S
\otimes_{min} \,\mathbb{B}  / \mathbb{K}.
$$
\end{theorem}
\begin{proof}
One direction is clear. So suppose that 
$
(\cl S \otimes_{min} \mathbb{B}) /(\cl S \otimes_{min} \mathbb{K})  \cong \cl S
\otimes_{min} \,\mathbb{B}  / \mathbb{K}.
$ 
In particular this implies that the associated map is completely isometric.
(Recall: The operator space quotient and operator system quotient has same
operator space structure.) So using Ozawa's above result we have that for every
separable unital C*-algebra $\cl A$ and ideal $I$ in $\cl A$ the associated map
$$
( \cl S \otimes_{{min}}
\cl A ) / ( \cl S {\otimes} I)   \longrightarrow    \cl S \otimes_{{min}}
\cl A / I 
$$
is isometric. (Note: the minimal tensor product of operator systems coincides
with the minimal operator space tensor product.) To see that it is complete
isometry it is enough to consider the identification $M_n(\cl A/I) = M_n(\cl
A)/M_n(I)$. Since a unital complete isometry is a complete order isomorphism
we have that the exactness is satisfied for the separable case. Now suppose $\cl A$
is an arbitrary unital C*-algebra and $I$ is an ideal in $\cl A$. Assume for a
contradiction that the associated map
$$
( \cl S \otimes_{{min}}
\cl A ) / ( \cl S {\otimes} I)   \longrightarrow    \cl S \otimes_{{min}}
\cl A / I 
$$
is not a complete isometry. Again considering the
identification $M_n(\cl A/I) = M_n(\cl
A)/M_n(I)$ we may suppose that the map is not an isometry. This means that
there is an element $u$ of $ \cl S \otimes_{{min}}
\cl A$ such that the norm of $u +  \cl S {\otimes} I $ under this associated
map is strictly smaller. Clearly $\cl A$ has a separable unital C*-subalgebra $\cl
A_0$ such that $u$ belongs to $\cl S \otimes \cl A_0$. Let $I_0 = \cl A_0 \cap
I$, which is an ideal in $\cl A_0$. Moreover we have  $\cl A_0/I_0 \subset \cl A
/ I$ so the injectivity of  minimal tensor products ensures that 
$$
\cl S
\otimes_{{min}}
\cl A_0 / I_0 \subset \cl S \otimes_{{min}}
\cl A / I .
$$
We also have the following sequence of ucp maps:
$$
 \cl S \otimes_{{min}}  \cl A_0 \hookrightarrow  \cl S \otimes_{{min}}
\cl A \rightarrow ( \cl S \otimes_{{min}}
\cl A ) / ( \cl S {\otimes} I)
$$
which has the kernel $\cl S \otimes I_0$. So the associated map $ ( \cl S
\otimes_{{min}}
\cl A_0 ) / ( \cl S {\otimes} I_0) \rightarrow  ( \cl S \otimes_{{min}}
\cl A ) / ( \cl S {\otimes} I)$ is ucp. Finally when we look at the following
sequence of ucp maps
$$
 ( \cl S
\otimes_{{min}}
\cl A_0 ) / ( \cl S {\otimes} I_0) \rightarrow  ( \cl S \otimes_{{min}}
\cl A ) / ( \cl S {\otimes} I) \longrightarrow
 \cl S \otimes_{{min}}
\cl A / I 
\supset
 \cl S
\otimes_{{min}} 
\cl A_0 / I_0 
$$
the norm of the element $u +  \cl S {\otimes} I_0 $ is smaller. This
is a contradiction as the exactness of $\cl S$ fails for a separable
C*-algebra and ideal in it.
\end{proof}

\begin{corollary}
Let $\cl S$ be an operator system. Then $\cl S$ is exact if and only if
$$
(\cl S \otimes_{\hat{min}} \mathbb{B}) /(\cl S \bar{\otimes} \mathbb{K})  \cong
\cl S
\otimes_{\hat{min}} \,\mathbb{B}  / \mathbb{K}.
$$
\end{corollary}
\begin{proof}
One direction is trivial. So suppose exactness in $\mathbb{K} \subset
\mathbb{B}$ satisfied. Let $\cl S_0 $ be a finite dimensional operator
subsystem. By using Corollary 5.6 of \cite{kptt2} we have that
$$
\frac{\cl S_0 \otimes_{min} \mathbb{B} }{  \cl S_0 \otimes \mathbb{K}  }
\subset \frac{\cl S \otimes_{\hat{min}} \mathbb{B}} {\cl S \bar{\otimes}
\mathbb{K}}.
$$
Similarly, by the injectivity of the minimal tensor product, we have $\cl S_0
\otimes_{min} \mathbb{B}/\mathbb{K}   \subset  \cl S
\otimes_{\hat{min}} \mathbb{B}/\mathbb{K}$. So for the operator system $\cl
S_0$ the exactness condition for $\mathbb{K} \subset  \mathbb{B}$ in 
Proposition \ref{fdexactness} is satisfied and consequently it is exact. Since
$\cl S_0$ is an arbitrary finite dimensional operator subsystem of $\cl S$, by Proposition \ref{prop exac pass ss},
it follows that $\cl S$ is exact.
\end{proof}

Recall from Theorem \ref{exactdualLP} that exactness and the lifting property are dual pairs. That is
a finite dimensional  operator system $\cl S$ has the lifting property if and only $\cl S^d$ is exact.
This, together with Lemma \ref{lem liftdualquotient}, lead to the following simplification of the lifting property.

\begin{proposition}\label{lpB(H)}
A finite dimensional operator system $\cl S$ has the lifting property if and only
if every ucp map defined from $\cl S$ into $\mathbb{B}/\mathbb{K}$ has a ucp
lift on $\mathbb{B}$.
\end{proposition}

\begin{proof}
Recall from Section \ref{sec osLLP} that in the definition of osLLP the local cp lifts can be taken to be unital.
This proves one direction. Conversely suppose that every ucp map defined from $\cl S$ into
$\mathbb{B}/\mathbb{K}$ has a ucp lift on $\mathbb{B}$. This means that, for every $n$, every ucp 
map $\varphi: \cl S \rightarrow M_n( \mathbb{B}) / M_n( \mathbb{K})$ has a ucp lift on $M_n( \mathbb{B})$.
In fact this directly follows from the fact that $M_n( \mathbb{B})$ can be identified with $\mathbb{B}$ via
a C*-algebraic isomorphism which preserves the compactness. Thus, $\cl S$ satisfies the property (1) in Lemma \ref{lem liftdualquotient}.
So the dual system $\cl S^d$ has the property (3) in the same lemma. Now, Theorem \ref{fdexactness} implies that $\cl S^d$ is exact. Finally, by Theorem
\ref{exactdualLP}, $\cl S$ has the lifting property.
\end{proof}

$ $

\section{Coproducts of Operator Systems}

In this chapter we recall basic facts on the amalgamated direct sum of two operator systems over
the unit introduced in \cite{KL} (or with the language of \cite{TF} coproduct of two operator systems) and
 we will show that it can be formed directly
by using the operator system quotient theory. We show that the lifting property is
preserved under coproducts. However the stability of the double commutant
expectation property turns out to be related to Kirchberg Conjecture.  Recall that
 if $\cl S$ and $\cl T$ are two operator systems then the coproduct $\cl S
\oplus_1 \cl T$ of $\cl S$ and $\cl T$ is an operator system together with
unital complete order embeddings $i:\cl S \hookrightarrow \cl S \oplus_1 \cl T$
and $j:\cl T \hookrightarrow \cl S \oplus_1 \cl T$ which satisfies the following
universal property: For every ucp map $\phi: \cl S \rightarrow \cl R$ and ucp
map $\psi:\cl T \rightarrow \cl R$, where $\cl R$ is an operator system, there
exists a unique ucp map $\varphi:\cl S  \oplus_1 \cl T \rightarrow R$ such that
$\varphi(i(s)) = \phi(s)$  and $\varphi(j(t)) = \psi(t)$ for every $s$ in $\cl
S$ and $t$ in $\cl T$.
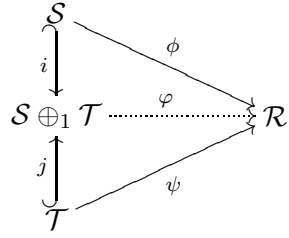
$$
\xymatrix{
\cl S \ar@{_{(}->}[d]_i   \ar[drr]^{\phi} && \\
\cl S   \oplus_1 \cl T  \ar@{.>}[rr]^{\varphi}  & & \cl R  \\
\cl T \ar@{^{(}->}[u]^j   \ar[urr]_{\psi} & &
}
$$
One way to construct this object can be described as
follows: Consider C*-algebra free product of $C_{u}^*(\cl S) \ast_1
C_{u}^*(\cl T)$ amalgamated over the identity. Define $\cl S  \oplus_1 \cl T$ as the operator system generated
by $\cl S$ and $\cl T$ in $C_{u}^*(\cl S) \ast_1 C_{u}^*(\cl T)$. We leave the
verification that this span has the above universal property as an exercise. We
also refer to \cite[Sec. 3]{TF} for a different construction of the coproducts.
Below we will obtain coproducts in terms of operator system quotients.

$ $

Consider $\cl S \oplus \cl T$. Since $(e,-e)$ is a selfadjoint element which is
neither positive nor negative, by Theorem \ref{npnn}, $J = span\{(e,-e)\}$ is a
kernel in $\cl S \oplus \cl T$ (in fact it is a null subspace and hence a
completely proximinal kernel by Proposition \ref{prop nullsubspace}). So we have
a quotient operator system $(\cl S \oplus \cl T) /J$. Note that in the quotient
we have
$$
(e,e) + J = (2e,0) + J = (0,2e) + J.
$$
Consider $i: \cl S \rightarrow \cl S \oplus \cl T / J$ by $s\mapsto (2s,0)+J$.
We claim that $i$ is a unital complete order isomorphism. Clearly it is unital
and completely positivity follows from the fact that it can be written as a
composition of cp maps, namely $\cl S \rightarrow \cl S \oplus \cl T$, $s\mapsto
(2s,0)$ and the quotient map. Now suppose that the image of $(s_{ij}) \in
M_n(\cl S)$ is positive. That is, $( (2(s_{ij},0)+J) )$ is positive in
$M_n( \cl S \oplus \cl T /J )$. Since $J$ is completely proximinal there are
scalars $ \alpha_{ij} $ such that $( (2s_{ij} + \alpha_{ij} e, -\alpha_{ij} e
))$ is positive in $M_n(\cl S \oplus \cl T)$. Note that this forces
$(-\alpha_{ij}e)$ to be positive in $M_n(\cl T)$. So we have that $(2s_{ij} +
\alpha_{ij} e) + (-\alpha_{ij}e) = 2(s_{ij})$ must be positive in $M_n(\cl S)$.
Hence $(s_{ij})$ is positive and it follows that $i$ is a complete order
isomorphism.

Similarly $j: \cl T \rightarrow \cl S \oplus \cl T / J$, $t\mapsto (0,2t) + J$
is also a unital complete order isomorphism. Finally let  $\phi: \cl S
\rightarrow \cl R$ and $\psi:\cl T \rightarrow \cl R$ be ucp maps. Consider
$\varphi:\cl S \oplus
\cl T  / J \rightarrow R$ given by $\varphi((s,t)+J) = (\phi(s) + \psi(t))/2$.
It is elementary to check $\varphi$ is ucp, $\varphi(i(\cdot)) = \phi$ and
$\varphi(j(\cdot)) = \psi$. Consequently with the above mentioned inclusions we
have
$$
\cl S  \oplus_1  \cl T =  \cl S \oplus \cl T\,/\,span\{(e,-e)\}.
$$
We also remark that $C_u^*(\cl S \oplus_1 \cl T) = C_u^*(\cl S) \ast_1 C_u^*(\cl
T)$, which in fact follows from the universal property of the coproduct of
operator systems and unital free products of C*-algebras. It is also clear
that when $\cl S$ and $\cl T$ are finite dimensional then $dim(\cl S \oplus_1
\cl T) =dim (\cl S) + \dim(\cl T) - 1$.

$ $

The lifting property is preserved under coproducts:

\begin{proposition}
The following are equivalent for finite dimensional operator systems $\cl S$
and $\cl T$:
\begin{enumerate}
 \item $\cl S$ and $\cl T$ have the lifting property.
 \item $\cl S \oplus_1 \cl T$ has the lifting property.
\end{enumerate}
\end{proposition}

\begin{proof} Suppose  $\cl S \oplus_1 \cl T$ has the lifting property. Let $\phi:
\cl S \rightarrow \cl A /I$ be a ucp map where $I \subset \cl A$ is a
C*-algebra, ideal couple. Suppose $f$ is a state on $\cl T$ and set $\psi: \cl T
\rightarrow \cl A/I$ by $\psi = f(\cdot) (e+I) $. By using the universal
property of  $\cl S \oplus_1 \cl T$ we obtain a ucp map $\varphi : \cl S \ast
\cl T \rightarrow \cl A / I$. For simplicity we will identify the $\cl S$ and
$\cl T$ with their canonical images in $\cl S  \oplus_1 \cl T$.  Clearly a ucp
lift of $\varphi$ on $\cl A$ is a ucp lift of $\phi$ when restricted to $\cl S$.
Thus $\cl S$ has osLLP. A similar argument shows that $\cl T$ has the same
property.

Conversely suppose $\cl S$ and $\cl T$ have the lifting property. Let $\varphi: \cl
S \oplus_1 \cl T \rightarrow \cl A /I$ be a ucp map. Again  we will identify the
$\cl S$ and $\cl T$ with their canonical images in $\cl S \oplus_1 \cl T$. Let
$\phi: \cl S \rightarrow \cl A$ be a ucp lift of  $\varphi|_{\cl S}$, the
restriction of $\varphi$ on $\cl S$. Similarly let $\psi$ be the ucp lift of
$\varphi|_{\cl T}$. Finally by using the universal property of $\cl S  \oplus_1
\cl T$ let $\tilde{\varphi}$ be the ucp map from $\cl S  \oplus_1 \cl T$ into
$\cl A$ associated with $\phi$ and $\psi$. It is elementary to see that
$\tilde{\varphi}$ is a lift of $\varphi$. This finishes the proof.
\end{proof}

Recall that we define $\cl S_n$ as
the operator system generated by the
unitary generator of $C^*(\mathbb{F}_n)$, that is,
$$
\cl S_n = span \{g_1,...,g_n,e,g_1^*,...,g_n^*\}  \subset C^*(F_n).
$$
We remind the reader that $\cl S_n$ can also be considered as the universal operator system generated by
$n$ contractions as it satisfies the following universal property: Every function
$f:\{g_i\}_{i=1}^n \rightarrow \cl T$
with $\|f(g_i)\| \leq 1$ extends uniquely to a ucp map $\varphi_f:S_n\rightarrow
\cl T$
(in an obvious way).
$$
\xymatrix{
\{g_i\}_{i=1}^n     \ar[rrr]^{f}    \ar@{_{(}->}[d]  &  &  &  \cl T \\
\cl S_n \ar@{.>}[rrru]_{\varphi_f} & & &
}
$$

It is easy to see that $\cl S_n$ is naturally included in $\cl S_{n+k}$ where the inclusion is given
by the map $g_i\mapsto g_i $ for $i = 1 , ... , n$. In a similar way, $\cl S_{k}$ can also be represented
in $\cl S_{n+k}$ via the map $g_i\mapsto g_{n+i}$ for $i = 1,...,k$. Thus, there is a map
from $\cl S_n  \oplus_1 \cl S_k  $ to $\cl S_{n+k}$. The following result states that this natural
map is a complete order isomorphism. We skip its elementary proof. In fact, it is easy to show that
$\cl S_{n+k}$ satisfies the universal property that $\cl S_n  \oplus_1 \cl S_k$ has.
\begin{lemma}
$\cl S_n  \oplus_1 \cl S_k = \cl S_{n+k}$.
\end{lemma}

\begin{proof}[Example]
 We wish to show that $\cl S_1 =span \{g,e,g^*\} \subset C^*(\mathbb{F}_1)$ is C*-nuclear. This
is based on Sz.-Nagy's dilation theorem (see \cite[Thm. 1.1]{Pa}, e.g.): If $T \in B(H)$ is a contraction
then there is a Hilbert space $K$ containing $H$ as a subspace and a unitary operator $U$ in $B(K)$ such that
$T^n = P_{H} U^n |_{H}$ for every positive $n$. Of course, by taking the adjoint, we also have that $(T^*)^n =  P_{H} (U^*)^n |_{H}$
for every positive $n$.  This means that there is a ucp map defined from $C^*(\mathbb{F}_1)$ into $C^*\{I, T,T^*\}$, the C*-algebra generated
by $T$ in $B(H)$, which is given by the compression of the unital $*$-homomorphism extending the representation $g\mapsto U$.
That is, the map $\gamma_T$ defined from $C^*(\mathbb{F}_1)$ into $B(H)$ given by $g^n \mapsto T^n$, $e \mapsto I$ and $g^{-n} \mapsto (T^*)^n$ is
ucp. Now we wish to show that $\cl S_1 \otimes_{max} \cl A \subset C^*(\mathbb{F}_1) \otimes_{max} \cl A$ for every $\cl A$. 
Let $\varphi: \cl S_1 \otimes_{max} \cl A \rightarrow B(K)$ be a ucp map. Then by Proposition \ref{prop ucp on c}, There is a Hilbert space
$K_1$ containing $K$ as a subspace and ucp maps $\phi: \cl S_1 \rightarrow B(K_1)$ and $\psi: \cl A \rightarrow B(K_1)$ with
commuting ranges such that $\varphi = P_K \phi \cdot \psi |_K$. Note that $\phi(g)$ must be a contraction. The map
$\gamma_{\phi(g)}$ is a ucp extension of $\phi$ on $C^*(\mathbb{F}_1)$. Clearly $\gamma_{\phi(g)}$ and $\psi$ have commuting ranges. Thus
$P_K \gamma_{\phi(g)} \cdot \psi |_K$ is a ucp extension of $\varphi$ on $C^*(\mathbb{F}_1) \otimes_{max} \cl A$. In conclusion
we have that every ucp map defined from $\cl S_1 \otimes_{max} \cl A$ into a $B(K)$ extends to a ucp map on $C^*(\mathbb{F}_1) \otimes_{max} \cl A$.
This is enough to conclude that $\cl S_1 \otimes_{max} \cl A \subset C^*(\mathbb{F}_1) \otimes_{max} \cl A$. It is well known that
$C^*(\mathbb{F}_1) = C^*(\mathbb{Z}) = C(\mathbb{T})$ (see \cite{Peder}, e.g.) and the C*-algebra of continuous functions on a compact set is nuclear
(see \cite{Pa}, e.g.). Since $\cl S_1 \otimes_{min} \cl A \subset  C^*(\mathbb{F}_1) \otimes_{min} \cl A$ and $ C^*(\mathbb{F}_1)$ is nuclear we conclude
that $\cl S_1$ is C*-nuclear.
 \renewcommand{\qedsymbol}{}\end{proof}

\begin{question}
In the previous example we have shown that the three dimensional operator system $span\{1,z,z^*\} \subset C(\mathbb{T})$, where $z$
is the coordinate function, is C*-nuclear. In general, if $X$ is a compact set then is every three dimensional operator subsystem $span\{1,f,f^*\} \subset C(X)$
C*-nuclear? In fact by using spectral theorem it is enough to consider the case when $X$ is subset of $\{z: |z|\leq 1\}$. So,
when this subset is the unit circle then the answer is affirmative.
\end{question}

Since the Kirchberg Conjecture (KC) is equivalent to the statement that $S_2$ has DCEP it is natural to
raise the following questions:

\begin{question} \label{firstQ}
Suppose $\cl S$ and $\cl T$ are two finite dimensional operator systems with
DCEP. Does $\cl S \oplus_1 \cl T$ have DCEP?
\end{question}

\begin{question}\label{secondQ}
Suppose $\cl S$ and $\cl T$ are two finite dimensional C*-nuclear operator systems. Does $\cl S \oplus_1 \cl T$  have DCEP?
\end{question}

\noindent \textbf{Results:} An affirmative answer to the Question \ref{firstQ} implies an
affirmative answer to the KC. This follows from the fact that $\cl S_2 = \cl S_1
\oplus_1 \cl S_1$ and $\cl S_1$ is C*-nuclear, in particular it has DCEP.
On the other hand Question \ref{secondQ} is equivalent to the KC. First suppose that KC is true. If $\cl S$ and $\cl T$ are
C*-nuclear operator systems then, in particular, they have the lifting property and so $\cl S \oplus_1 \cl T$ has the lifting property.
Since we assumed KC, by using Theorem \ref{thm new KC},  $\cl S \oplus_1 \cl T$  must have DCEP. Conversely
if we suppose that Question \ref{secondQ} is true then in particular $\cl S_2 = \cl S_1
\oplus_1 \cl S_1$ has DCEP.

$ $

\section{k-minimality and k-maximality}

In this section we review k-minimality and k-maximality in
the category of operator systems introduced by Xhabli in \cite{blerina}. This
theory and a similar construction in the category of operator spaces are used
extensively in the understanding of entanglement breaking maps and separability
problems in quantum information theory (\cite{blerina}, \cite{blerina2} and
\cite{Pa3}, e.g.). Our interest in k-minimality and k-maximality arises from
their compatiblity with exactness and the lifting property which will be
apparent in this section. We start with the following observation:

\begin{proposition}\label{unu}
Let $\varphi: \cl S \rightarrow B(H) $ be a linear map. Then $\varphi$ is
$k$-positive if and only if there is a unital $k$-positive map $\psi : \cl S
\rightarrow B(H)$ and $R\geq 0$ in $B(H)$ such that $\varphi = R \psi(\cdot) R$.
\end{proposition}

\begin{proof} We will show only the non-trivial direction.  Let $\varphi: \cl S
\rightarrow B(H) $ be a $k$-positive map. We assume that $\varphi(e) = A$
satisfies $0\leq A \leq I$, where $I$ is the identity in $B(H)$. For any
$\epsilon > 0$ let $\varphi_\epsilon : \cl S \rightarrow B(H)$ be the map
defined by $\varphi_\epsilon = (A+\epsilon I)^{-1/2} \varphi(\cdot)  (A+\epsilon
I)^{-1/2}$. Since $B(\cl S, B(H))$ is a dual object, which arises from the fact that
$B(H)$ is dual of a Banach space, the net 
$\{\varphi_\epsilon\}$ has a $w^*$-limit, say $\psi$.
First note that $\psi$ is unital. Indeed, $\varphi_\epsilon (e) = A(A+\epsilon
I)^{-1}$ converges to the identity $I$ in the $w^*$-topology of $B(H)$. 
Consequently $\psi$ is unital.
We also claim that $\psi$ is $k$-positive. To see this let $(s_{ij})$ be
positive in $M_k(\cl S)$. Since $\varphi_{\epsilon}$ is $k$-positive we have
that $(\varphi_{\epsilon}(s_{ij}))$ is positive in $M_k(B(H))$. The weak
convergence $\varphi_{\epsilon}\rightarrow \psi$ ensures that, for fixed $i,j$, 
$\varphi_{\epsilon}(s_{ij})$ has a limit in the $w^*$-topology of $B(H)$
which is necessarily $\psi(s_{ij})$. Now the result follows from the fact that
positives cones are closed in the $w^*$-topology of $B(H)$. Finally we claim
that $\varphi = A^{1/2} \psi(\cdot) A^{1/2}$. Indeed this follows from the
uniqueness of the $w^*$-limit in $B(\cl S, B(H))$. In fact we have that 
$A^{1/2}
\varphi_{\epsilon}(\cdot) A^{1/2}$ converges to $A^{1/2} \psi(\cdot)
A^{1/2}$. On the other hand for fixed $s$ in $\cl S$, $A^{1/2}
\varphi_{\epsilon}(s) A^{1/2}$ converges to $ \varphi(s)$ (in the
$w^*$-topology of $B(H)$). So the proof is done.
\end{proof}

\begin{corollary} \label{cor unital non unital}
The following properties of an operator system $\cl S$ are equivalent:
\begin{enumerate}
 \item Every k-positive map defined from $\cl S$ into an operator system is cp.
 \item Every unital  k-positive map defined from $\cl S$ into an operator system
is cp.
\end{enumerate}
\end{corollary}

Before getting started with the k-minimality and k-maximality we also recall the
following result  (see Thm. 6.1 of \cite{Pa}).
\begin{lemma}
Suppose $\phi:\cl S \rightarrow M_k$ is a linear map. Then $\phi$ is k-positive
if and only if it is completely positive.
\end{lemma}

Following Xhabli \cite{blerina}, for an operator system $\cl S$ we define the
k-minimal cone structure as follows:
$$
C^{k\mbox{-}min}_n = \{ (s_{ij})\in M_n(\cl S) : (\phi(s_{ij})) \geq 0 \mbox{
for every ucp } \phi:\cl S \rightarrow M_k   \}.
$$
By considering Proposition \ref{unu} one can replace ucp by cp in this
definition. Now, the $*$-vector space $\cl S$ together
with the matricial cone structure $\{C^{k\mbox{-}min}_n\}_{n=1}^\infty$ and the
unit $e$ form an operator system which is called the \textit{k-minimal
operator system structure generated by $\cl S$} and denoted by $\OMIN_k(\cl S)$.
We refer \cite[Section 2.3]{blerina} for the proof of these results
and we remark that $\OMIN_k(\cl S)$ is named as \textit{ super k-minimal
structure} so we drop the term ``super'' in this paper. 
Roughly speaking $\OMIN_k(\cl S)$ is (possibly) a new operator system whose
positive cones coincide with the positive cones of
$\cl S$ up to the $k^{th}$ level and after the $k^{th}$ level they are the
largest cones
so that the total matricial cone structure is still an operator system. Note
that larger cones generate smaller canonical operator space structure  so this
construction is named the k-minimal structure.
We list a couple of remarkable results from \cite{blerina}:

\begin{theorem}[Xhabli] Suppose $\cl S$ is an operator system and $k$ is a fixed
number. Then:
\begin{enumerate}
 \item $\OMIN_k(\cl S)$ can be represented in $M_k( C(X))$ for some compact
space $X$.
\item If $\varphi : \cl T \rightarrow \OMIN_k(\cl S)$ is a $k$-positive map
then $\varphi$ is completely positive.
\item The identity $id:\OMIN_k(\cl S) \rightarrow \cl S$ is k-positive.
\item For any $m \leq k$ the identities $\cl S \rightarrow \OMIN_k(\cl
S)\rightarrow \OMIN_m(\cl S) $ are completely positive. 
\end{enumerate}
\end{theorem}

\begin{lemma}\label{lem SisOminS}
Let $\cl S$ be an operator system. Then $\cl S = \OMIN_k(\cl S)$ if and only if
every k-positive map defined from an operator system into $\cl S$ is completely
positive. 
\end{lemma}
\begin{proof}
One direction follows from the above result of Xhabli. Conversely, suppose that
every
k-positive map defined into $\cl S$ is cp. This, in particular, implies that
the identity $id: \OMIN_k(\cl S) \rightarrow \cl S$, which is k-positive, is
cp.   Since the inverse of this map is also cp it follows that $\cl S =
\OMIN_k(\cl S)$.
\end{proof}

Let $\cl S$ be an operator system and $k$ be a fixed natural number. To define
the k-maximal structure we first consider the following cones:
$$
D^{k\mbox{-}max}_n = \{ A^*DA:\; A \in M_{mk,n} \mbox{ and  } D =
diagonal(D_1,...,D_m) \hspace{2cm}
$$
$$
  \hspace{4cm} \mbox{ where }D_i \in M_k(\cl S)^+ \mbox{ for } i =
1,...,m\}. 
$$
$\{D^{k\mbox{-}max}_n\}$ forms a strict compatible matricial order
structure on $\cl S$ and $e$ is an matricial order unit. However, $e$ fails to
be Archimedean and to resolve this problem we use the Archimedeanization process
(see \cite{pt}):
$$
C^{k\mbox{-}max}_n = \{(s_{ij}) \in M_n(\cl S): \;\; (s_{ij}) + \epsilon e_n
\in D^{k\mbox{-}max}_n \mbox{ for every } \epsilon >0\}.
$$
Note that $D^{k\mbox{-}max}_n \subset C^{k\mbox{-}max}_n$. The
$*$-vector space $\cl S$ together
with the matricial cone structure $\{C^{k\mbox{-}max}_n\}_{n=1}^\infty$ and the
unit $e$ form an operator system which is called \textit{k-maximal
operator system structure generated by $\cl S$} and denoted by $\OMAX_k(\cl
S)$. For related proof we refer \cite[Sec. 2.3.]{blerina}. (We again drop the
term ``super''.) The $\OMAX_k(\cl S)$ is (possibly) a new operator system
structure on the $*$-vector space $\cl S$ such that the matricial cones
coincide with the matricial cones of the operator system $\cl S$ up to
$k^{th}$-level and after $k$, the cones are the smallest possible cones such a
way that the total structure makes $\cl S$ an operator system with unit $e$.

\begin{theorem}[Xhabli]
Let $\cl S$ be an operator system and $k$ be a fixed
number. Then: 
\begin{enumerate}
 \item Every $k$-positive map defined from $\OMAX_k(\cl S)$ into an
operator system is completely positive.
 \item The identity $id :\cl S \rightarrow \OMAX_k(\cl S)$ is $k$-positive.
 \item For any $m \leq k$ the identities $\OMAX_m(\cl S) \rightarrow \OMAX_k
(\cl S) \rightarrow \cl S$ are completely positive.
\end{enumerate}
\end{theorem}

The proof of the following lemma is similar to Lemma \ref{lem SisOminS} so we
skip it.

\begin{lemma}\label{lem SisOmaxS}
Let $\cl S$ be an operator system. Then $\cl S = \OMAX_k(\cl S)$ if and only if
every k-positive map defined from $\cl S$ into another operator system is
completely positive. 
\end{lemma}

After these preliminary results we are ready to examine the role of
k-minimality and the k-maximality in the nuclearity theory. We start with the
following easy observation:

\begin{lemma}\label{lem ominexactness}
$\OMIN_k(\cl S)$ is exact for any operator system $\cl S$ and $k$.
\end{lemma}
\begin{proof}
Recall that $\OMIN_k(\cl S)$ can be represented in $M_k(C(X))$ for some
compact space $X$. Note that $M_k(C(X))$ is a nuclear C*-algebra and
consequently it is (min,max)-nuclear operator system. Clearly (min,max)-nuclearity
implies (min,el)-nuclearity (equivalently exactness) and, by Proposition \ref{prop exac pass ss}, exactness passes to
operator subsystems so we have that $\OMIN_k(\cl S)$ is exact.
\end{proof}

Note that if $\cl S$ is a finite dimensional operator system then a faithful
state on $\cl S$ still has the same property when $\cl S$ is
equipped with $\OMIN_k$ or $\OMAX_k$ structure. Keeping this observation in
mind we are ready to state:
\begin{theorem}
Let  $\cl S$ be a finite dimensional operator system. Then we have the unital complete order isomorphisms
$$
\OMIN{} _k (\cl S)^d = \OMAX{ }_k(\cl S^d) \;\;\mbox{ and }\;\; \OMAX{ }_k(\cl S)^d = \OMIN{ }_k(\cl S^d).
$$
\end{theorem}

\begin{proof}
We only prove the fist equality. The second equality follows from the first
one if we replace $\cl S$ by $\cl S^d$ and take the dual of both
side. To show the first one we set $\cl R = \OMIN_k (\cl
S)$ and we will first prove the following: Whenever $\varphi: \cl R^d \rightarrow \cl
T$ is a k-positive map then $\varphi$ is cp. So by using Lemma \ref{lem
SisOmaxS} we conclude that $\cl R^d = \OMAX_k(\cl R^d)$. Assume for a
contradiction that there is a k-positive map $\varphi: \cl R^d \rightarrow \cl
T$ which is not cp. Clearly we may assume that $\cl T$ is finite dimensional. (If not
we can consider an operator subsystem of $\cl T$ containing the image of $\varphi$.)
Now by using Lemma \ref{lem dualofkpos.map} we have that $\varphi^d : \cl T^d
\rightarrow \cl R$ is a k-positive map but it is not cp. This is a contradiction
as Lemma \ref{lem SisOminS} requires that $\varphi$ is a cp map. Thus $\cl R^d = \OMAX_k(\cl R^d)$.
Next we  show that $\OMAX_k(\cl R^d) = \OMAX_k(\cl S^d)$ which finishes the proof. To see
this note that the identity $id: \cl S \rightarrow \cl R$ is cp and its inverse is $k$-positive. This
implies that $id^d: \cl R^d \rightarrow \cl S$ is cp and its inverse is $k$-positive. (We skip the
elementary proof of the fact that $(\varphi^d)^{-1} = (\varphi^{-1})^{d}$.) Thus up to $k^{th}$ level
$\cl R^d$ and $\cl S^d$ are order isomorphic. Hence $\OMAX_k(\cl R^d) = \OMAX_k(\cl S^d)$. Finally
by using the observation that we mentioned before the theorem we may assume that this
identification is also unital.
\end{proof}

\begin{lemma}\label{omaxhaslp}
Suppose $\cl S$ is a finite dimensional operator system. Then $\OMAX_k(\cl S)$
has the lifting property  for any natural number $k$.
\end{lemma}
\begin{proof}
Lemma \ref{lem ominexactness} states that $\OMIN_k(\cl S^d)$ is exact. By the
above theorem we see that $\OMIN_k(\cl S^d)^d =\OMAX_k(\cl S)$ and
by using Theorem \ref{exactdualLP} we conclude that this dual has the lifting property.
\end{proof}

We are now ready to establish a weaker lifting property:
\begin{theorem}\label{klproperty}
Every finite dimensional operator system $\cl S$ has the $k$-lifting property (for every $k$)
in the sense that whenever $I$ is an ideal in a 
unital C*-algebra $\cl A$ and $\varphi:\cl S \rightarrow \cl A/I$ is a
ucp map then, for every $k$, there exists a unital k-positive
map $\tilde{\varphi}:\cl S \rightarrow \cl A$ such that $q\circ \tilde{\varphi}
= \varphi$ where $q:\cl A \rightarrow \cl A /I$ is the quotient map.
\end{theorem}

\begin{proof}
Let $\varphi: \cl S \rightarrow \cl A/I$ be a ucp map. When $\cl S$ is equipped with OMAX$_k$
structure $\varphi$ remains to be a ucp map. By Lemma \ref{omaxhaslp}  $\OMAX_k(\cl S)$ has the lifting
property so $\varphi$ lifts to a cp map  $\tilde{\varphi}$ on $\cl A$ which can taken to be unital.
Now when we consider $\tilde{\varphi}$
as a map defined from $\cl S$ it is $k$-positive. This completes the proof.
\end{proof}

We want to remark that if $\cl S$ is a finite dimensional operator system then
the $k$-lifting property, which $\cl S$ has for every $k$,  does not
imply the lifting property. In \cite[Theorem 3.3.]{Pa4} it was shown that there is
a five dimensional operator subsystem of the Calkin algebra
$\mathbb{B}/\mathbb{K}$ such that  the inclusion does not  have a ucp lift (or
cp lift) on $\mathbb{B}$. In the next section we will see that even $M_2 \oplus M_2$ has a five dimensional
operator system that does not have the lifting property. For three dimensional operator systems a similar
problem turns out to be equivalent to the Smith Ward problem which we will study
in Section \ref{Sec MNR}.

\begin{corollary}
Let $\cl S$ be a finite dimensional operator system, $\cl A$ be a C*-algebra and $I$ be an ideal. Then
every $k$-positive map $\cl S \rightarrow \cl A / I$ has a $k$-positive lifting to $\cl A$. If $\varphi$
is unital one can take the lift unital too.
\end{corollary}
\begin{proof}
If we equip $\cl S$ with $\OMAX_k$ structure then $\varphi$ is completely positive. Since
$\OMAX_k(\cl S)$ has the lifting property, by using Remark \ref{rem strong lp}, $\varphi$ can be lifted as a completely positive map on $\cl A$.
If $\varphi$ is unital one can pick the lift unital as well. Now when $\cl S$ is considered with its initial structure
this lift is $k$-positive.
\end{proof}

\begin{corollary}
Let $X$ be a finite dimensional operator space, $\cl A$ be a unital C*-algebra and $I$ be an ideal in $\cl A$. Then
every completely contractive (cc) map $\phi: X \rightarrow \cl A /I$ has a k-contractive lift on $\cl A$ for every $k$. 
\end{corollary}

\begin{proof}
Recall that the universal operator system $\cl R_X \supset X$ has the property that every cc map
defined from $X$ into an operator system extends uniquely to a ucp map.  Now,
$\phi$ extends to a ucp map $\varphi: \cl R_X \rightarrow \cl A / I$. By the above corollary this map
has a unital $2k$-positive lift on $\cl A$. Since a  unital $2k$-positive map is $k$-contractive,
the restriction of this lift on $X$ has the desired property.
\end{proof}

$ $

\section{Quotients of the Matrix Algebra $M_n$}

In this section we obtain new proofs of some of the results of \cite{pf} and discuss some new formulations of the Kirchberg Conjecture (KC) in terms
of operator system quotients of the matrix algebras. The duality and the quotient theory when applied to
some special operator subsystems of $M_n$ raise difficult stability problems which will
be apparent in this section. We will also consider the problem about the minimal and the maximal tensor product
of three copies of $C^*(\mathbb{F}_\infty)$ from an operator system perspective. 

$ $

Recall from Example \ref{exam MnJn} that we define $J_n \subset M_n$ as the diagonal
matrices with 0 trace. As we pointed out, $J_n$ is a null subspace of $M_n$ and consequently,
by Proposition \ref{prop nullsubspace}, it is a completely proximinal kernel. (Also recall that $M_n/J_n$ has the lifting property.)
However, with the following result of Farenick and Paulsen we directly see that $J_n$ is a 
kernel and, moreover, we obtain an identification of $M_n/J_n$ as well as its enveloping C*-algebra.

$ $

As usual $C^*(\mathbb{F}_n)$ stands for the full C*-algebra of the free group $\mathbb{F}_n$ on $n$
generators, say $g_1,...,g_n$. Let $\cl W_n$ be the operator subsystem of $C^*(\mathbb{F}_n)$ given by
$$
\cl W_n = \{g_ig_j^*: 1\leq i,j \leq n\}.
$$
We are now ready to establish the connection between these operator systems given
in \cite{pf}. As usual $\{E_{ij}\}$ denotes the standard matrix units for $M_n$.
Consider $\varphi: M_n \rightarrow \cl W_{n}$ given by $\varphi(E_{ij}) =
g_ig_j^*/n $. Then

\begin{theorem}[Farenick, Paulsen]
The above map $\varphi: M_n \rightarrow \cl W_{n}$ is a quotient map with kernel
$J_n$. That is, the induced map $\bar{\varphi} : M_n / J_n \rightarrow \cl W_n$
is a bijective unital complete order isomorphism. Moreover, $C^*_e(M_n/J_n) =
C^*(\mathbb{F}_{n-1})$.
\end{theorem}

Now we are ready to state:
\begin{theorem}
The following are equivalent:
\begin{enumerate}
 \item KC has an affirmative answer.
 \item $M_3/J_3$ has DCEP.
 \item $M_3/J_3 \otimes_{min} M_3/J_3 =  M_3/J_3 \otimes_{c} M_3/J_3 $.
\end{enumerate}
\end{theorem}

\begin{proof}
Example \ref{exam MnJn} states that $M_3/J_3$ has the lifting property. So if we assume (1) then, by
Theorem \ref{thm new KC}, $M_3/J_3$ has DCEP. This proves that (1) implies (2). To see that
(2) implies (3) we recall that the lifting property is characterized by (min,er)-nuclearity. Thus
we readily have that $M_3/J_3 \otimes_{min} M_3/J_3 =  M_3/J_3 \otimes_{er} M_3/J_3 $. Now, by
our assumption, $M_3/J_3$ has DCEP, equivalently, (el,c)-nuclearity. Now, applying this to $M_3/J_3$
on the right hand side, we have that $M_3/J_3 \otimes_{er} M_3/J_3 =  M_3/J_3 \otimes_{c} M_3/J_3 $.
Thus, (2) implies (3). We finally show that (3) implies (1). In fact, $M_3/J_3$ contains enough
unitaries in its enveloping C*-algebra, namely, $C^*(\mathbb{F}_2)$ (see Section \ref{sec new WEP KC} for the related definition). 
This simply follows from the fact that $\cl W_3$ is linear span of unitaries, thus, it contains enough unitaries in the C*-algebra generated
by itself (in $C^*(\mathbb{F}_3)$). So, by Proposition \ref{prop enough=envelo.}, this C*-algebra must be coincides with its
enveloping C*-algebra. Now, by identifying $M_n/J_n$ with $\cl W_n$, we conclude that $M_3/J_3$ contains enough
unitaries in its enveloping C*-algebra, namely, $C^*(\mathbb{F}_2)$. Thus assuming (3),
by Corollary \ref{co enough unitary}, we have that 
$C^*(\mathbb{F}_2) \otimes_{min} C^*(\mathbb{F}_2) =  C^*(\mathbb{F}_2) \otimes_{max} C^*(\mathbb{F}_2) $. Thus (3) implies (1).
\end{proof}

We remark that Theorem 5.2. of \cite{pf} states that if $M_n/J_n \otimes_{min}
M_n/J_n =  M_n/J_n \otimes_{max} M_n/J_n$ for every $n$ then it follows that KC
has an affirmative answer.

\begin{question}
Is $M_n/J_n \otimes_{c} M_n/J_n =  M_n/J_n \otimes_{max} M_n/J_n$ for every $n$? What about $n=3$?
\end{question}

Recall that we define $S_n$ as the operator subsystem of $C^*(\mathbb{F}_n)$ which contains
the unitary generators. More precisely, $\cl S_n = \{g_1,...,g_n,e,g_1^*,...,g_n^*\}$. 
 Another important operator subsystem of $M_n$, which is related to $\cl S_n$, is the tridiagonal matrices $T_n$. We define
$$
T_n = \{ A\! = \! (a_{ij}) \in M_n\;: \;   a_{ij} = 0 \mbox{ if } |i-j| \geq 1  \}.
$$
The study on the nuclearity properties of these operator systems goes back to
\cite{kptt}. In Theorem 5.16 it was shown that $T_3$ is C*-nuclear (i.e.
(min,c)-nuclear). In general, Proposition 6.11 states that if $\cl S$ is an
operator subsystem of $M_n$ associated with a chordal graph $G$ then $\cl S$ is
C*-nuclear. We refer to Section 5 of \cite{kptt} for related definitions and
discussions. Since $T_n$ is associated with the chordal graph (over vertices
$\{1,2,...,n\}$)
$$
\{(1,1),\,(1,2),\,(2,1),\,(2,2),\,(2,3),\, (3,2), (3,3), \, (3,4),...,(n,n) \}
$$
we have that 
\begin{proposition}
$T_n$ is C*-nuclear for every $n$.
\end{proposition}
As we mentioned at the end of Section \ref{sec tensor}, a finite dimensional operator system is (c,max)-nuclear if
and only if it completely order isomorphic to a C*-algebra. Consequently, for an
operator system which is not a C*-algebra, such as $T_n$, C*-nuclearity is the
highest nuclearity that one should expect.

$ $

Since $J_n$, the diagonal $n \times n$ matrices with 0 trace, is a null subspace
of $T_n$, by Proposition \ref{prop nullsubspace}, it is a completely proximinal kernel. 
Also note that C*-nuclearity clearly implies the lifting property and so, by Theorem \ref{thm lp-quotient}, we have that
$T_n/J_n$ has lifting  property. 
The following is from \cite{pf}:

\begin{theorem}
$T_n / J_n$ is unitally completely order isomorphic to $\cl S_{n-1}$. More
precisely, the ucp map $ \gamma: T_n \rightarrow \cl S_{n-1} $  given by
$$
 \begin{array}{lcl}
 E_{i,i}        &  \mapsto &  e/n \mbox{ for } i = 1,...,n \\
 E_{i,i+1}   &  \mapsto  &  g_i/n \mbox{ for } i = 1,...,n-1\\
 E_{i+1,i}   &  \mapsto  &  g_i^*/n \mbox{ for } i = 1,...,n-1
\end{array}
$$
is a quotient map with kernel $J_n$.
\end{theorem}

\noindent This again brings difficult stability problems we have considered in the
last section:
\begin{corollary}
The following are equivalent:
\begin{enumerate}
 \item KC has an affirmative answer.
 \item For any finite dimensional C*-nuclear operator system $\cl S$ and null
subspace $J$ of $\cl S$ one has $\cl S / J$ has DCEP.
 \item $T_n / J_n$ has DCEP for every n.
 \item $T_3 / J_3$ has DCEP.
\end{enumerate}
\end{corollary}
\begin{proof}
 Since $T_3/J_3 = \cl S_2$, (1) and (4) are equivalent by Theorem \ref{thm new KC}. Also, as we mentioned, $\cl T_n/J_n$ has lifting
property. So if we assume (1) we must have that  $\cl T_n/J_n$ has DCEP. (3) implies (4) is clear. Now we need to show 
that (2) is equivalent to remaining. Clearly (2) implies (4) (or (3)). On the other hand if $\cl S$ is C*-nuclear then, in particular,
it has the lifting property and so, by Theorem \ref{thm lp-quotient}, $\cl S/J$ has the lifting property. So assuming (1) we must
have that this quotient has DCEP.
\end{proof}

\noindent This corollary indicates that KC is indeed an operator system quotient problem. DCEP is one 
of the extensions of WEP from unital C*-algebras to general operator systems. In addition
to being equivalent to (el,c)-nuclearity we have seen that it is an important property in the understanding
of KC. However, the following definition will allow us to relax DCEP to another property:

\begin{definition}
We say that an operator system $\cl S$ has property $\mathbb{S}_2$ if $\cl S \otimes_{min} \cl S_2 = \cl S \otimes_{c} \cl S_2$.
\end{definition}

We remark that, for unital C*-algebras, property $\mathbb{S}_2$ coincides with WEP. That is,
a unital C*-algebra has WEP if and only if it has property $\mathbb{S}_2$. This directly
follows from Theorem \ref{thm new WEP}. It is also worth mentioning that, again for unital C*-algebras,  
property $\mathbb{S}_2$ coincides with property $\mathfrak{W}$ and property $\mathfrak{S}$ in \cite{pf}.
We refer the reader to Section 3 and 6 in \cite{pf} for related definitions. For the operator systems we have that
$$
WEP \;\;\;\Longrightarrow\;\;\; DCEP \;\;\;\Longrightarrow\;\;\;\; property \; \mathbb{S}_2.
$$
We know that DCEP, in general, does not imply WEP. For example if $\cl S$ is a finite
dimensional operator system then WEP is equivalent to $\cl S$ having the structure of a C*-algebra
(which follows from the fact that (el,max)-nuclearity implies (c,max)-nuclearity). On the other hand 
$T_n$ is a C*-nuclear operator system for every $n$ and in particular it has DCEP. So this family forms an example
that DCEP is weaker than WEP. To see that DCEP implies property $ \mathbb{S}_2$, let $\cl S$ be an operator system
with DCEP (equivalently (el,c)-nuclearity). Since $ \cl S_2$ has the lifting property (i.e. (min,er)-nuclearity) (and keeping
in mind that it is written on the right hand side) we have
$$
\cl S \otimes_{min} \cl S_2 = \cl S \otimes_{el}  \cl S_2 = \cl S \otimes_{c}   \cl S_2.
$$
Thus, $\cl S$ has property $ \mathbb{S}_2$. However we don't know whether property $\mathbb{S}_2$ implies DCEP.
\begin{question}
Does property $\mathbb{S}_2$ imply DCEP? 
\end{question}

\begin{proposition}\label{pro tensor prop.S2}
Suppose $\cl S \otimes_{\tau} \cl T$ has property $\mathbb{S}_2$ (resp. has DCEP) where $\tau$ is any functorial
tensor product. Then both $\cl S$ and $\cl T$ have property $\mathbb{S}_2$ (resp. have DCEP).
\end{proposition}

\begin{proof}
This follows from a very basic principle: The identity on $\cl S$ factors via ucp maps through $\cl S \otimes_{\tau} \cl T$.
More precisely, the inclusion $i: \cl S \rightarrow \cl S \otimes_{\tau} \cl T$ given by $s \mapsto s\otimes e_{\cl T}$ is
a ucp map. Conversely, if $g$ is a state on $\cl T$ then $id \otimes g : \cl S \otimes_{\tau} \cl T \rightarrow \cl S \otimes \mathbb{C} \cong \cl S$ 
is again a ucp map such that $(id \otimes g) \circ i$ is the identity on $\cl S$. This shows that if $ \cl S \otimes_{\tau} \cl T$
has DCEP (equivalently (el,c)-nuclearity) then by Lemma \ref{lem iden. decom.} $\cl S$ has DCEP. Clearly a similar
argument shows that $\cl T$ has the same property. Now suppose that $ \cl S \otimes_{\tau} \cl T$ has property $\mathbb{S}_2$.
By using the functoriality of min and c tensor products we have that
$$
\cl S \otimes_{min} \cl S_2 \xrightarrow{i \otimes id}  (\cl S \otimes_{\tau} \cl T)   \otimes_{min} \cl S_2 =   (\cl S \otimes_{\tau} \cl T)   \otimes_{c} \cl S_2
\xrightarrow{ (id\otimes g) \otimes id} \cl S \otimes_{c} \cl S_2
$$
is a sequence of ucp maps such that their composition is the identity on $\cl S \otimes \cl S_2$. Since min $\leq$ c we obtain that $\cl S$ has property $\mathbb{S}_2$.
The proof for $\cl T$ is similar.
\end{proof}

The fact that $\cl T_3 / J_3 = \cl S_2$ together with Theorem \ref{thm new WEP} allow us characterize WEP as follows (we refer the reader to \cite{FKP}
for further applications of this result):

\begin{theorem}
A unital C*-algebra $\cl A$ has WEP if and only if the associated map $\cl T_3 \otimes_{min} \cl A  \rightarrow ( \cl T_3 / J_3) \otimes_{min} \cl A $
is a quotient map. In other words we have the complete order isomorphism
$$
(\cl T_3 \otimes_{min} \cl A) / (J_3 \otimes \cl A ) = ( \cl T_3 / J_3) \otimes_{min} \cl A.
$$
\end{theorem}

\begin{proof}
By using the projectivity of the maximal tensor product and C*-nuclearity of $T_3 $ we have that
$$
 \cl S_2  \otimes_{max} \cl A = \cl T_3 / J_3 \otimes_{max} \cl A  =  (\cl T_3 \otimes_{max} \cl A) / (J_3 \otimes \cl A ) = (\cl T_3 \otimes_{min} \cl A) / (J_3 \otimes \cl A ).
$$
Now if $\cl A$ has WEP then it has property $\mathbb{S}_2$ and the equality in the theorem satisfies. Conversely
if the equality is satisfied then $\cl A$ must have property $\mathbb{S}_2$, equivalently, WEP.
\end{proof}

We now discuss some duality results from \cite{pf}. Recall that
we write $\cl S_n$ in the following basis form:  $\cl S_n = span  \{g_1,...,g_n,
e, g_1^*,...,g_n^*\}$. When we pass to dual basis we have that
$$
\cl S_n^d = span \{\delta_1,...,\delta_n, \delta, \delta_1^*,...,\delta_n^* \}.
$$
We leave the elementary proof of the fact that $\delta_{g_i^*} = \delta_i^*$ to
the reader. We also remind that $\delta$ is a faithful state and we consider it
as the Archimedean matrix order unit for the dual operator system. We now
see that $\cl S_n^d$ can be identified with an operator subsystem of $M_2 \oplus M_2 \oplus \cdots \oplus M_2$ (the direct sum of $n$ copies of $M_2$).
To avoid the excessive notation we use the following:
$$
e = (I_2,....,I_2), \;\;\; e_1 = (E_{12},0,...,0),\;\;\; e_2 =
(0,E_{12},0,..., 0), \; ...\;  e_n = (0,...,0,E_{12}).
$$
Consider the
following map:
$$
\gamma : \cl S_n^d \rightarrow \oplus_{i=1}^n M_2 \mbox{ given by }\delta
\mapsto e, \;\;\;\delta_i \mapsto e_i \; \mbox{ and } \; \delta_i^* \mapsto
e_i^* \mbox{ for } i = 1,...,n.
$$ 
Now we are ready to state:
\begin{theorem}[Farenick, Paulsen]
The above map $\gamma : \cl S_n^d \rightarrow \oplus_{i=1}^n M_2$ is a unital
complete order embedding.
\end{theorem}
\noindent By using the diagonal identification of $M_2 \oplus M_2$ in $M_4$, in particular, we have that
$$
\cl S_2^d =  \{ \left(
\begin{array}{cccc}
a & b & 0 & 0 \\
c & a & 0 & 0 \\
0 & 0 & a & d \\
0 & 0 & e & a
\end{array}
\right)   \;\;:\;\; a,b,c,d,e \in \mathbb{C} 
\}.
$$ 

In \cite{Was2} it was shown by Wasserman that $C^*(F_n)$ is not exact for any $n \geq 2$. Clearly
$\cl S_n$ contains enough unitaries in $C^*(F_n)$. The following
is Corollary 9.6 in \cite{kptt2}:
\begin{proposition} \label{prop enoughexact}
Suppose that $\cl S \subset \cl A$ contains enough unitaries. If $\cl S$ is exact then $\cl A$ is exact.
\end{proposition}

\begin{corollary}
$\cl S_n$ is not exact for any $n \geq 2$.
\end{corollary}

Exactness is stable under C*-algebra ideal
quotients, that is, if a C*-algebra is exact then any of its quotient by an
ideal has the same property (see \cite{Ki1} and \cite{Was}). This stability property is not valid for
general operator system quotients even under the favorable conditions: The
dimension of the operator system is finite and the kernel is a null subspace.
In fact since $\cl T_n$ is C*-nuclear (i.e. (min,c)-nuclear) then in particular
it is exact (equivalently (min,el)-nuclear). However, its quotient by the null
subspace $J_n$, namely $\cl S_n = T_n / J_n$, is not exact.

\begin{corollary}\label{cor M4hasnonlp}
$M_2 \oplus M_2$ (or $M_4$) has a five dimensional operator subsystem (namely $\cl S_2^d$) which does
not possess the lifting property.
\end{corollary}
\begin{proof}
Since $\cl S_2$ is not exact then, by Theorem \ref{exactdualLP}, its dual can not have the lifting property.
\end{proof}

The following result is perhaps well known but we are unable to provide a reference.

\begin{corollary}
The Calkin algebra $\mathbb{B}/\mathbb{K}$ does not have WEP.
\end{corollary}
\begin{proof}
Assume for a contradiction that $\mathbb{B}/\mathbb{K}$ has WEP. This means that
$\cl S_2 \otimes_{min} \mathbb{B}/\mathbb{K} = \cl S_2 \otimes_{c=max} \mathbb{B}/\mathbb{K} $.
Since $\cl S_2$ has the lifting property we also have that $\cl S_2 \otimes_{min} \mathbb{B}= \cl S_2 \otimes_{max} \mathbb{B} $.
Thus,
$$
\cl S_2 \otimes_{max} \mathbb{B}/\mathbb{K} =  (\cl S_2 \otimes_{max} \mathbb{B}) / (\cl S_2 \otimes \mathbb{K}) = 
 (\cl S_2 \otimes_{min} \mathbb{B}) / (\cl S_2 \otimes \mathbb{K}) = \cl S_2 \otimes_{min} \mathbb{B}/\mathbb{K}.
$$
This means that, by Theorem \ref{fdexactness}, $\cl S_2$ is exact which is a contradiction.
\end{proof}

\begin{corollary}
$\cl S_2 \otimes_{max} \cl S_2$ has the lifting property.
\end{corollary}

\begin{proof}
Note that $(\cl S_2 \otimes_{max} \cl S_2)^d = \cl S_2^d \otimes_{min} \cl S_2^d \subset M_4 \otimes_{min} M_4$. Since exactness
passes to operator subsystems we have that $(\cl S_2 \otimes_{max} \cl S_2)^d$ is exact. Thus, by Theorem \ref{exactdualLP},
$\cl S_2 \otimes_{max} \cl S_2$ has the lifting property.
\end{proof}

\noindent {\bf Remark:} We don't know whether the lifting property is preserved under the maximal tensor product.
For finite dimensional operator systems,  by using Theorem \ref{exactdualLP} and \ref{thm dualminmax}, the same question
can be reformulated as follows: Is exactness preserved under the minimal tensor product? If $\cl S$ and $\cl T$
contain enough unitaries in their enveloping C*-algebras (also under the assumption that both $\cl S$ and $\cl T$
are separable) the answer is affirmative. In fact, by Proposition \ref{prop enoughexact}, 
both $C^*_e(\cl S)$ and $C^*_e(\cl T)$
must be exact. Also note that both of these C*-algebras are separable. We know that
every separable exact C*-algebra can be represented in a nuclear  C*-algebra \cite{Ki3}. So 
$\cl S$ and $\cl T$ can be represented in nuclear C*-algebras, say $\cl A$ and $\cl B$, respectively. 
Note that $\cl S \otimes_{min} \cl T \subset \cl A \otimes_{min} \cl B$ and it is elementary to show 
that $\cl A \otimes_{min} \cl B$ is again nuclear. Thus,
$\cl S \otimes_{min} \cl T$ embeds in a  nuclear C*-algebra. Since nuclearity implies exactness
and exactness passes to operator subsystems it follows that $\cl S \otimes_{min} \cl T$ is exact. However,
 in general, the exactness of $\cl S$ may not pass to $C^*_e(\cl S)$. In \cite{kw}, Kirchberg
and Wassermann construct a separable, (min,max)-nuclear operator system $\cl S$ with the property that
$C_u^*(\cl S) = C_e^*(\cl S)$. Clearly $\cl S$ is exact. However, since $dim(\cl S) \geq 3$, $C_u^*(\cl S)$
equivalently $C_e^*(\cl S)$, is not exact. 

$ $

After these results we also relate the property $\mathbb{S}_2$ and the KC.

\begin{theorem}\label{thm KCandpS2}
The following are equivalent:
\begin{enumerate}
 \item KC has an affirmative answer.
 \item $\cl S_2$ has property $\mathbb{S}_2$.
 \item Every finite dimensional operator system with the lifting property has property $\mathbb{S}_2$.
 \item If $\cl S$ is a finite dimensional exact operator system then $\cl S^d$ has property $\mathbb{S}_2$.
 \item If $\cl S$ is a finite dimensional C*-nuclear operator system and $J$ is a null subspace of $\cl S$ then
$\cl S / J$ has property $\mathbb{S}_2$.
 \item $\cl S_2 \otimes_{max} \cl S_2$ has property $\mathbb{S}_2$.
\end{enumerate}
\end{theorem}

\begin{proof}
The equivalence of (1) and (2) is simply a restatement of Theorem \ref{thm new KC}.
If we assume (1) then it follows that every finite dimensional operator system with lifting property has DCEP, in particular, property $\mathbb{S}_2$. This
proves that (1) implies (3). Clearly (3) implies (6). If we assume (6) then Proposition \ref{pro tensor prop.S2} implies that $\cl S_2$
has property $\mathbb{S}_2$. Thus, (2) is true. So we need to show that these are all equivalent to (4) and (5). 

\noindent (1) $\Rightarrow$ (4): Let $\cl S$ be an exact operator system. By Theorem \ref{exactdualLP}, $\cl S^d$ has the lifting property and consequently, it has DCEP.
DCEP implies property $\mathbb{S}_2$ thus (1) implies (4).

\noindent (4) $\Rightarrow$ (2): In fact $\cl S_2^d$ is exact so its dual, namely $\cl S_2$,
has property $\mathbb{S}_2$. 

\noindent (1) $\Rightarrow$ (5): If $\cl S$ is C*-nuclear, in particular, it has the lifting property. Thus, by Theorem \ref{thm lp-quotient},  
$\cl S/J$ has the lifting property. By using Theorem \ref{thm new KC},  $\cl S / J$ must have DCEP and, thus, it must have property $\mathbb{S}_2$.

\noindent (5) $\Rightarrow$ (2): So, in particular, $T_3/J_3 = \cl S_2$ has property $\mathbb{S}_2$. This finishes the proof.
\end{proof}

$ $

In quantum mechanics, one of the basic problems in modeling an experiment is determining whether
by using the classical probabilistic approach we can approximate all outcomes arising from the non-commutative
setting. More precisely, Tsirelson's problem asks whether the  non-relativistic behaviors  in a
quantum experiment can be described by relativistic approach. The proper definitions and basic result
in this question are beyond the scope of this paper and we refer the reader to \cite{SWT}, \cite{Tsi1}, \cite{TFTsirelson}. In  \cite{Tsi1} it was shown that 
when the actors are Alice and Bob (that is, in the bipartite scenario) the question is reduced to whether
$$
C^*(\mathbb{F}_{\infty}) \otimes_{min} C^*(\mathbb{F}_{\infty}) = C^*(\mathbb{F}_{\infty}) \otimes_{max} C^*(\mathbb{F}_{\infty}),
$$
in other words, the Kirchberg Conjecture. When Charlie  is also included, i.e. with three actors, Tsirelson's problem
is known to be related to whether the minimal and the maximal tensor products of three copies of  $C^*(\mathbb{F}_{\infty})$
coincide. So we want to close this section with a discussion on this topic from an operator system perspective.
\begin{conjecture}\label{con tripleKC}
$$
\bigotimes_{i=1}^3 {}_{min} C^*(\mathbb{F}_{\infty}) = \bigotimes_{i=1}^3 {}_{max} C^*(\mathbb{F}_{\infty}).
$$
\end{conjecture}
This should be considered as an extended version of the Kirchberg Conjecture. An affirmative answer of Conjecture \ref{con tripleKC}
implies that the Kirchberg Conjecture is true. In fact this follows from the fact that for any functorial tensor product $\tau$ and operator
systems $\cl S$ and $\cl T$ we have that $\cl S \cong \cl S \otimes \mathbb{C} \subset \cl S \otimes_{\tau} \cl T$. So if we put $C = C^*(\mathbb{F}_{\infty})$ then
$$
C \otimes_{min} C \subset (C \otimes_{min} C) \otimes_{min} C \;\mbox{ and }\; C \otimes_{max} C \subset (C \otimes_{max} C) \otimes_{max} C.
$$
Thus, if Conjecture \ref{con tripleKC} is true then KC is also true. On the other hand even if we assume that KC has an affirmative answer it is still
unknown whether Conjecture \ref{con tripleKC} is true or not. We want to start with the following observations which are perhaps well known
and will be more convenient when we express this problem in terms  of lower dimensional operator systems.
\begin{theorem} The following are equivalent:
\begin{enumerate}
 \item Conjecture \ref{con tripleKC} has an affirmative answer.
 \item $C^*(\mathbb{F}_{\infty}) \otimes_{max}C^*(\mathbb{F}_{\infty})$ has WEP.
 \item We have that
$$
\bigotimes_{i=1}^3 {}_{min} C^*(\mathbb{F}_2) = \bigotimes_{i=1}^3 {}_{max} C^*(\mathbb{F}_2).
$$
\item $C^*(\mathbb{F}_2) \otimes_{max}C^*(\mathbb{F}_2)$ has WEP.
\end{enumerate}
\end{theorem}
\begin{proof}
Since the identity on $C^*(\mathbb{F}_{\infty})$  factors via ucp 
maps through $C^*(\mathbb{F}_2)$, by using the functoriality of the max tensor product, it follows  that
the identity on $C^*(\mathbb{F}_{\infty}) \otimes_{max}C^*(\mathbb{F}_{\infty})$ factors via ucp maps
through $C^*(\mathbb{F}_2) \otimes_{max}C^*(\mathbb{F}_2)$. So by using Lemma \ref{lem iden. decom.}
we obtain that (4) implies (2). Since the  identity on $C^*(\mathbb{F}_2)$  factors via ucp 
maps through $C^*(\mathbb{F}_\infty)$, we similarly obtain that (2) implies (4). The proof of the equivalence of (1) and (3) 
is based on the same fact. In general, if the identity on $\cl S$ decomposes
into ucp maps through $\cl T$ (say $id = \psi \circ \phi$), also assuming that $\cl T \otimes_{min} \cl T \otimes_{min} \cl T 
= \cl T \otimes_{max} \cl T \otimes_{max} \cl T$, we have that the maps
$$
\bigotimes_{i=1}^3 {}_{min} \cl S \xrightarrow{\phi\otimes \phi\otimes \phi} \bigotimes_{i=1}^3 {}_{min} \cl T
= \bigotimes_{i=1}^3 {}_{max} \cl T \xrightarrow{\psi\otimes \psi\otimes \psi} \bigotimes_{i=1}^3 {}_{max} \cl S
$$
are ucp and their composition is the identity from triple minimal tensor product of $\cl S$ to maximal tensor
product of $\cl S$. Thus these two tensor products coincide. This proves that (1) and (3) are equivalent. Now let $C$ stand
for $C^*(\mathbb{F}_{\infty})$. We will show that (2) implies (1). Since $C \otimes_{max} C$ has WEP then
in particular, by Lemma \ref{lem iden. decom.}, this implies that $C$ has WEP, equivalently $C \otimes_{min} C = C \otimes_{max} C$.
(Recall: These are all equivalent arguments in Kirchberg's theorem that we mentioned at the beginning of Section \ref{sec new WEP KC}.)
By using Kirchberg's WEP characterization we readily have that $(C \otimes_{min} C) \otimes_{min} C = (C \otimes_{min} C) \otimes_{max} C$.
If we replace the min by max on the right hand side of this equation we obtain (1). Conversely suppose (1) is true. As we pointed
out earlier,  this, in particular, implies KC. Thus $C \otimes_{min} C = C \otimes_{max} C$. Since we assumed that
the triple minimal and the maximal tensor product of $C$ coincide, by replacing a max by min (as seen below)
$$
(C \otimes_{max} C) \otimes_{max} C =  (C \otimes_{min =max } C) \otimes_{min} C
$$
we have that , $C \otimes_{max} C$ satisfies Kirchberg's WEP characterization. So we obtain (2).
\end{proof}

\begin{theorem}
The following implications hold:

$\cl S_2 \otimes_{min} \cl S_2$ has DCEP $\; \Rightarrow\; $ $\cl S_2 \otimes_{min} \cl S_2$ has property $\mathbb{S}_2$ $\; \Rightarrow \;$ Conjecture \ref{con tripleKC} is true.
\end{theorem}
\begin{proof} Clearly DCEP implies property $\mathbb{S}_2$. Now, suppose that $\cl S_2 \otimes_{min} \cl S_2$ has property $\mathbb{S}_2$.
It is not hard to see that $\cl S_2 \otimes_{min} \cl S_2$ contains enough unitaries in $C^*(\mathbb{F}_2) \otimes_{min} C^*(\mathbb{F}_2)$.
We also remark that our assumption implies KC, that is, if $\cl S_2 \otimes_{min} \cl S_2$ has property $\mathbb{S}_2$
then, in particular,  $\cl S_2$ has property $\mathbb{S}_2$ and by the above result
KC has an affirmative answer. So we also have that $C^*(\mathbb{F}_2) \otimes_{min} C^*(\mathbb{F}_2) = C^*(\mathbb{F}_2) \otimes_{max} C^*(\mathbb{F}_2) $.
Now by using Proposition \ref{pr enough unitary}, $(\cl S_2 \otimes_{min} \cl S_2) \otimes_{min} \cl S_2 = (\cl S_2 \otimes_{min} \cl S_2) \otimes_{c} \cl S_2$
implies that $(C^*(\mathbb{F}_2) \otimes_{min} C^*(\mathbb{F}_2)) \otimes_{min}    C^*(\mathbb{F}_2) = 
(C^*(\mathbb{F}_2) \otimes_{min} C^*(\mathbb{F}_2))  \otimes_{max}  C^*(\mathbb{F}_2) $. Since the min on the right
hand side can be replaced by max it follows that Conjecture \ref{con tripleKC} has an affirmative answer.
\end{proof}

We don't know whether any of the converse implications in the above theorem hold or not.

\begin{question}\label{ques first} 
Is $\cl S_2 \otimes_c \cl S_2 = \cl S_2 \otimes_{max} \cl S_2 $? 
\end{question}

\begin{question}
Are DCEP or property $\mathbb{S}_2$ preserved under commuting tensor product? That is, if $\cl S$
and $\cl T$ are operator systems with DCEP (or having  property $\mathbb{S}_2$) then does $\cl S \otimes_{c} \cl T$
have the same property?
\end{question}

\noindent An affirmative answer to any of these questions implies that KC is equivalent to Conjecture \ref{con tripleKC}.
We first remark that in the above theorem the min can be replaced by c, this
follows from the fact that any of the arguments implies KC is true and, thus, $\cl S_2 \otimes_{min} \cl S_2 = \cl S_2 \otimes_{c} \cl S_2 $.
Now if we suppose that the first question has an affirmative answer then Theorem \ref{thm KCandpS2} (6) and the second
argument in the above theorem gives this equivalence. Now suppose that the second question is true. If
we suppose KC has an affirmative answer (so that $\cl S_2$ has DCEP) then $\cl S_2 \otimes_{min} \cl S_2 =  \cl S_2 \otimes_{c} \cl S_2 $ and
this tensor product has DCEP (or property $\mathbb{S}_2$), thus,  Conjecture \ref{con tripleKC} is also true.

$ $

In \cite{Oz5}, Ozawa proved that $B(H) \otimes_{min} B(H)$ does not have WEP where $H = l^2$. Since WEP and DCEP coincide for C*-algebras and
$B(l^2)$ has WEP we see that DCEP, in general, does not preserved under the minimal tensor product.

$ $

\noindent Let $
\cl T =  span \{ I, E_{12}, E_{34}, E_{21}, E_{43} \} \subset M_4.
$ Recall that $\cl S_2^d$ and $\cl T$ are unitally completely order isomorphic. 
Thus we have that
$$
\cl S_2 \otimes_{min} \cl S_2 = \cl S_2 \otimes_{max} \cl S_2\;\;\; \Longleftrightarrow \;\;\; \cl T \otimes_{min} \cl T = \cl T \otimes_{max} \cl T
$$
which follows from the duality result in Theorem \ref{thm dualminmax}.

\begin{question}\label{ques 2}
Is $\cl T \otimes_{min} \cl T = \cl T \otimes_{max} \cl T$? Equivalently, is $\cl S_2 \otimes_{min} \cl S_2 = \cl S_2 \otimes_{max} \cl S_2$?
\end{question}

\noindent  Since KC is equivalent to the statement that $\cl S_2 \otimes_{min} \cl S_2 = \cl S_2 \otimes_{c} \cl S_2$ a positive answer
to this question provides an affirmative answer to KC. In addition to this it also proves that Conjecture \ref{con tripleKC} is true
since the condition in the Question \ref{ques first}
is satisfied and thus, by the previous paragraph, KC and  Conjecture \ref{con tripleKC} are equivalent.

\section{Matricial Numerical Range of an Operator}\label{Sec MNR}

Let $\cl S$ be an operator system. For $x\in \cl S$ we define the $n^{th}$
\textit{matricial numerical range of} $x$ by $w_n(x) = \{\varphi (x): \varphi : \cl S
\rightarrow M_n \mbox{ is ucp}\}$. Note that if we consider the
operator subsystem $\cl S_x = span \{e,x,x^*\}$ of $\cl S$ then, by using
Arveson's extension theorem, the matricial ranges of $x$ remain same when it is
considered as an element of $\cl S_x$. We
also remark that if $T$ is an operator in $B(H)$ then its numerical range $W(T) =
\{\langle Tx,x \rangle\,:\; \|x\| \leq 1   \} $ has the property that
$\overline{W(T)} = w_1(T) $ (see \cite{Ar2}, e.g.). For several properties and
results regarding matricial ranges we refer the reader to \cite{Ar2}, \cite{Pa4}
and \cite{SW}.  We include some of
these results in the sequel. We start with the following well known fact (see \cite[Lem. 4.1]{kptt} e.g.).

\begin{lemma}
Let $\cl S$ be an operator system and $A \in M_n(\cl S)$. Then $A$ is positive if and only if for every $k$ and
for every ucp map $\varphi: \cl S \rightarrow M_k$ one has $\varphi^{(n)}(A)$
is positive in $M_n(M_k)$.
\end{lemma}

\noindent This lemma indicates that the matricial ranges of an element $x$ in an operator system 
carry all the information of the operator subsystem $\cl S_x = span\{e,x,x^*\}$ as $A \in M_n$
belongs to $w_n(x)$ if and only if there is a ucp map $\varphi: \cl
S_x\rightarrow M_n$ such that $\varphi(x) = A$. Since $\varphi$ is ucp the image
of any element in $\cl S_x$ can be determined by the value $\varphi(x)$. We can
also state this as follows:
\begin{proposition}\label{pro numrange x y}
Let $\cl S = span\{e,x,x^*\}$ and $\cl T = span\{e, y , y^*\}$ be two operator systems. Then the linear map 
$\varphi: \cl S \rightarrow \cl T$ given by $\varphi(e)= e$, $\varphi(x) = y$ and $\varphi(x^*) =y^*$, provided it is well-defined, is ucp if and only
if $w_n(y) \subseteq w_n(x)$ for every $n$. Consequently, $\varphi$ is a complete order isomorphism
if and only if $w_n(x) = w_n(y)$ for every $n$.
\end{proposition}

\begin{proof}
First suppose that $\varphi$ is ucp and let $A \in w_n(y)$. So there is ucp map $\psi: \cl T \rightarrow M_n$
such that $\psi(y) = A$. Clearly $\psi \circ \varphi$ is a ucp map from $\cl S$ into $M_n$ which maps $x$ to $A$.
Thus, $A$ belongs to $w_n(x)$. Since $n$ was arbitrary this completes the proof of one direction. Now suppose
that $w_n(y) \subseteq w_n(x)$ for every $n$. We will show that $\varphi$ is a cp map.
The above lemma states that if $u$ is in $M_n(\cl R)$, where $\cl R$
is any operator system, then $u$ is positive if and only if for every $k$ and for 
every ucp map $\phi: \cl R \rightarrow M_k$ one has $\phi^{(n)}(u)$ is positive.
From this we deduce that an element of the form $ u = e\otimes A + x \otimes B + x^* \otimes C$ in $M_n(\cl S)$ is
positive if and only if for every ucp map  map $\phi: \cl R \rightarrow M_k$ one has $\phi^{(n)}(u) = I_k \otimes A + \varphi(x) \otimes B + \varphi(x)^* \otimes C$
is positive in $M_k \otimes M_n$ for every $k$, equivalently, $I_k \otimes A + X \otimes B + X^* \otimes C$ is positive in $M_k \otimes M_n$ for every
$k$ and for every $X$ in $w_k(x)$. Of course, same property holds $M_n(\cl T)$ in when $x$ is replaced by $y$.
Now, by using the assumption $w_k(y) \subseteq w_k(x)$ for every $k$, it is easy to see that $\varphi$ is a cp map.
The final part follows from the fact that $\varphi^{-1}$ is ucp if and only if $w_n(x) \subseteq w_n(y)$ for every $n$.
\end{proof}

$ $

In this section we again use the notations $\mathbb{B}$ for $B(l^2)$ and
$\mathbb{K}$ for the ideal of compact operators. A dot over an element
will represents its image under the quotient map. We start with the following
result given in \cite{SW}.

\begin{theorem}[Smith, Ward]
Let $ \dot{T} \in \mathbb{B}/\mathbb{K}$ and $n$ be an integer. Then there is a
compact operator $K$ such that $w_n(T+K) = w_n(\dot{T})$.
\end{theorem}

\noindent {\bf Remark:} In fact this theorem follows by using the $k$-lifting property of a finite dimensional operator system
(Theorem \ref{klproperty}). Moreover, we can deduce a more general form of this result: If
$\cl A$ is a unital C*-algebra and $I \subset \cl A$ is an ideal then for any
$\dot{a}$ in $\cl A/I$, and for any $k$ there is an element $x$ in $I$ such that
$w_k(a+x) = w_k(\dot{a})$. This directly follows from the $k$-lifting property of
the operator system $\cl S_{\dot{a}} = \{\dot{e},\dot{a}, \dot{a}^*\}$ and the fact that every
$k$-positive map defined from an operator system into $M_k$ is completely
positive.

$ $

\noindent Turning back to the above result, we see that for a fixed $n$, an
operator $T \in \mathbb{B}$ can be compactly perturbed such that the resulting
operator and its its residue under the quotient map have the same $n^{th}$
matricial range. Then the authors stated the following conjecture which is
currently still open. 

$ $

\noindent \textbf{Smith Ward Problem (SWP):} For every $T$ in $B(H)$ there is a
compact operator $K$ such that $ w_n(T+K) = w_n(\dot{T}) \mbox{ for every } n. $

$ $

This question is also considered in \cite{Pa4} and several equivalent
formulations have been given. In particular it was shown that it is enough to
consider block diagonal operators, and for this case,  the problem reduces to a
certain distance question \cite[Thm. 3.16]{Pa4}. However, the following
remark which depends on an observation in \cite{Ar2} will be more relevant
to us. We include the proof of this for the completeness of the paper.
\begin{proposition}[Paulsen]
The following are equivalent:
\begin{enumerate}
 \item SWP has an affirmative answer.
 \item For every operator subsystem of the form $\cl S_{\dot{T}} = \{ \dot{I},
\dot{T},
\dot{T}^*\} $ in the Calkin algebra $ \mathbb{B} /  \mathbb{K} $, the inclusion
$ \cl S_{\dot{T}}  \hookrightarrow \mathbb{B} /  \mathbb{K}  $ has a ucp lift on
$\mathbb{B}$.
\end{enumerate}
\end{proposition}

\begin{proof}
First suppose that $ \cl S_{\dot{T}}$ has a ucp lift $\varphi$ on $\mathbb{B}$. Since $q(\varphi(\dot{T})) = T$, where $q$ is the quotient map
from $ \mathbb{B}$ into $ \mathbb{B} /  \mathbb{K}$, $\varphi(\dot{T}) = T + K$ for some compact operator $K$. It is not hard to show that $w_n(T+K) = w_n(\dot{T})$ for every $n$. 
In fact, if $A \in w_n(\dot{T})$, say $A$ is the image $\phi(\dot{T})$ of some ucp map $\phi: \mathbb{B} /  \mathbb{K} \rightarrow M_n$, then the
composition $\phi\circ q$ is a ucp map from $ \mathbb{B}$ into $M_n$ which maps $T$ to $A$. Conversely if $B$ is in $w_n(T)$,
say $\psi(T) = B$ where $\psi:  \mathbb{B} \rightarrow M_n$ is ucp, then $\psi \circ \varphi : \cl S_{\dot{T}} \rightarrow M_n$ is ucp that maps $\dot{T}$ to $B$.
Since $T$ was arbitrary it follows that (1) is true. Conversely suppose that $(1)$ holds. So for $\dot{T}$ in $ \mathbb{B} /  \mathbb{K}$
we can find $K$ in $\mathbb{K}$ such that $w_n(\dot{T}) = w_n(T+K)$ for every $n$. Now, by using Proposition \ref{pro numrange x y},
$\cl S_{\dot{T}}$ and $\cl S_{T+K} \subset \mathbb{B}$ are unitally completely order isomorphic via $\dot{T} \mapsto T+ K$.
This map is ucp and a lift of the inclusion $ \cl S_{\dot{T}}  \hookrightarrow \mathbb{B} /  \mathbb{K} $. So proof is done.
\end{proof}

Depending on Proposition \ref{lpB(H)} and Theorem \ref{exactdualLP} we obtain
the following formulations of the SWP:

\begin{theorem}
The following are equivalent:
\begin{enumerate}
 \item SWP has an affirmative answer.
 \item Every three dimensional operator system has the lifting property.
 \item Every three dimensional operator system is exact.
\end{enumerate}
\end{theorem}
\begin{proof}
Equivalence of (2) and (3) follows from Theorem \ref{exactdualLP}. If every
three dimensional operator system is exact then then their duals, which covers
all three dimensional operator systems, must have the lifting property and vice
versa. Now suppose (2). This in particular implies that every operator subsystem
of the form $\cl S_{\dot{T}} = \{ \dot{I}, \dot{T}, \dot{T}^*\} $ in the Calkin
algebra $ \mathbb{B} /  \mathbb{K} $, the inclusion $ \cl S_{\dot{T}} 
\hookrightarrow \mathbb{B} /  \mathbb{K}  $ has a ucp lift on $\mathbb{B}$.
Hence by the above result of Paulsen, we conclude that SWP has an affirmative answer. Now suppose
(1) holds. Let $\cl S$ be a three dimensional operator system. We will show that
$\cl S$ has the lifting property. Let $\varphi:\cl S \rightarrow \mathbb{B} / 
\mathbb{K} $ be a ucp map. Clearly the image $\varphi(\cl S)$ is of the form
$\cl S_{\dot{T}} = \{ \dot{I}, \dot{T}, \dot{T}^*\} $ for some $T$ in
$\mathbb{B}$. Since we assumed SWP, the above result of Paulsen ensures that $\cl
S_{\dot{T}}$ has a ucp lift on $\mathbb{B}$, say $\psi$. Now $\psi \circ
\varphi$ is a ucp lift of $\varphi$ on $\mathbb{B}$. Finally by using
Proposition \ref{lpB(H)} we conclude that $\cl S$ has the lifting property.
\end{proof}

Recall from Proposition \ref{prop two dim} that every two dimensional operator
system is C*-nuclear and consequently they are all exact and have the lifting
property. On the other hand there is a five dimensional operator system, namely $\cl S_2$,  which
is not exact and, by Theorem \ref{exactdualLP}, its dual $\cl S_2^d$, which embeds in $M_2 \oplus M_2$, does not posses the lifting property.

\begin{remark}
There is a four dimensional operator system which is not exact and consequently, by Theorem  \ref{exactdualLP},
its dual does not have the lifting property.
\end{remark}
\begin{proof}
It is well known that $\mathbb{F}_2$ embeds in $\mathbb{Z}_2 * \mathbb{Z}_3$ (see \cite[pg. 24]{PDLH} ,e.g.). So by using
Proposition 8.8 of \cite{pi}, $C^*(\mathbb{F}_2)$ embeds in $C^*(\mathbb{Z}_2 * \mathbb{Z}_3)$ with a ucp
inverse. Thus, the identity on $C^*(\mathbb{F}_2)$ decomposes via ucp maps through $C^*(\mathbb{Z}_2 * \mathbb{Z}_3)$.
This means that, by Lemma \ref{lem iden. decom.}, any nuclearity property of $C^*(\mathbb{Z}_2 * \mathbb{Z}_3)$ passes to $C^*(\mathbb{F}_2)$.
Since $C^*(\mathbb{F}_2)$ is not exact we obtain that $C^*(\mathbb{Z}_2 * \mathbb{Z}_3)$ cannot be exact. 
Note that  $\mathbb{Z}_2 * \mathbb{Z}_3$ can be described by $\langle a,b: a^2 = b^3 = e\rangle$ so, necessarily,
$a$ must be a self adjoint unitary in $C^*(\mathbb{Z}_2 * \mathbb{Z}_3)$. Let $\cl S = span\{e,a,b,b^*\}$.
$\cl S$ is a four dimensional operator subsystem of $C^*(\mathbb{Z}_3 * \mathbb{Z}_2)$ that
contains enough unitaries. By Corollary 9.6 of \cite{kptt2}  $\cl S$ cannot be exact.  The remaining part follows from Theorem  \ref{exactdualLP}.
\end{proof}

Now we turn back to the Kirchberg Conjecture (KC). Before we establish a connection
between SWP and KC we recall that an operator system is (min,c)-nuclear if and
only if it is C*-nuclear. We refer back to Section \ref{sec nuclearity} for
related discussion. We also remind the reader that  KC is equivalent to the statement that every finite dimensional
operator system that has the lifting property has the double commutant expectation
property (DCEP). Now if we assume that both SWP and KC have affirmative answers
then it follows that every operator system with dimension three is exact and has the
lifting property, equivalently, they are  all (min,el)-nuclear and
(min,er)-nuclear. Since we assumed KC it follows that all  three dimensional
operator systems must have DCEP, equivalently (el,c)-nuclearity. Finally,
(min,el)-nuclearity and (el,c)-nuclearity implies (min,c)-nuclearity,
that is, C*-nuclearity. Conversely if every operator system of dimension three
is C*-nuclear this in particular implies they are all exact, (or have lifting
property). Hence we obtain that
$$
\mbox{ KC } + \mbox{ SWP } 
\Longrightarrow
\begin{array}{c}
                           \mbox{ every three dimensional }\\
			    \mbox{ operator system is C*-nuclear }
                           \end{array}
\Longrightarrow
\mbox{ SWP }
$$
Consequently forming an example of a three dimensional operator system
which is not C*-nuclear shows that both KC and SWP cannot be true.
Showing indeed that they are all C*-nuclear provides an affirmative answer to
SWP.

\begin{question}
We repeat a question we considered before: If $X$ is a compact subset of $\{z:\;|z| \leq 1\}$ then
is $\cl S = \{1,z,z^*\}$, where $z$ is the coordinate function, C*-nuclear? When $X$ is the unit circle
$\mathbb{T}$ then $\cl S$ coincides with $\cl S_1$ and for this case we know that $\cl S$ is C*-nuclear. 
\end{question}

\noindent \textbf{Acknowledgment: } This paper is based on my Ph.D. Dissertation ``Tensor Products of Operator
Systems and Applications'' under the supervision of Dr. Vern I. Paulsen. I would like to express my sincere gratitude to him
for his guidance at every stage of this work.

\end{document}